\documentclass[12pt, leqno]{article}
\usepackage{amsmath}
\usepackage{amsfonts}
\usepackage{amssymb}
\usepackage{graphicx}
\usepackage{color}
\usepackage[latin1]{inputenc}
\usepackage{amsthm}
\usepackage{enumerate}
\usepackage{url}
\usepackage[labelfont=bf]{caption}

\usepackage[dvipsnames]{xcolor}

\newtheorem{thm}{Theorem}[section]

\newtheorem{lem}[thm]{Lemma}

\newtheorem{defn}[thm]{Definition}
\newtheorem{rem}[thm]{Remark}
\numberwithin{equation}{section}

\newcommand{\beq}{\begin{equation}}
\newcommand{\eeq}{\end{equation}}

\newcommand{\rd}{{\rm d}}

\def\LL{\mathrm{L}} %per gli spazi L^p
\newcommand{\Ubar}{\overline{U}}
\newcommand{\ubar}{\underline{U}}

\def\qed{\,\unskip\kern 6pt \penalty 500
\raise -2pt\hbox{\vrule \vbox to8pt{\hrule width 6pt \vfill\hrule}\vrule}\par}

\definecolor{darkblue}{rgb}{0.05, .05, .65}
\definecolor{darkgreen}{rgb}{0.05, .70, .05}
\definecolor{darkred}{rgb}{0.8,0,0}
\def\qed{\unskip\kern 6pt \penalty 500
\raise -2pt\hbox{\vrule \vbox to8pt{\hrule width 6pt \vfill\hrule}\vrule}\par} %%%%%%%%%%%%%%%%%%%%%%%%%%%%%%%%%%%%%%%%%%%%%%%%%%%%%%%%%%%%%%%%%%%
\begin{document}

% \title{\textbf{ \Large  The porous medium equation \\ on manifolds with very negative curvature. \\ The large time behaviour}\\[10mm]}

\title{\textbf{ \Large  The porous medium equation on Riemannian manifolds with negative curvature. \\ The large-time behaviour}\\[10mm]}

\author{Gabriele Grillo\footnote{Dipartimento di Matematica, Politecnico di Milano, Piazza Leonardo da Vinci 32, 20133 Milano, Italy, \it email address\rm: \tt gabriele.grillo@polimi.it}, Matteo Muratori\footnote{Dipartimento di Matematica ``F. Casorati'', Universit\`a degli Studi di Pavia, via A. Ferrata 5, 27100 Pavia, Italy, \it email address\rm: \tt matteo.muratori@unipv.it}, and Juan Luis V\'azquez\footnote{Departamento de Matem\'aticas, Universidad Aut\'onoma de Madrid, 28049 Madrid, Spain, \it email address\rm: \tt juanluis.vazquez@uam.es}
}

\date{} %%  this cancels date in article format

\maketitle

\begin{abstract}

\noindent We consider nonnegative solutions of the porous medium equation (PME) on a Cartan-Hadamard manifold whose negative curvature can be unbounded. We take compactly supported initial data because we are also interested in free boundaries. We classify the geometrical cases we study into quasi-hyperbolic, quasi-Euclidean and critical cases, depending on the growth rate of the curvature at infinity.  We prove sharp upper and lower bounds on the long-time behaviour of the solutions in terms of  corresponding  bounds on the curvature. In particular we obtain a sharp form of the smoothing effect on such manifolds. We also estimate the location of the free boundary. A global Harnack principle follows.

We also present a change of variables that allows to transform radially symmetric solutions of the PME on model manifolds into radially symmetric solutions of a corresponding weighted PME on Euclidean space and back. This equivalence turns out to be an important tool of the theory.

\end{abstract}
%%%%%%%%%%%%%%%%%%%%%%%%%%%%%%%%%%%%%%%%%%%%%%%%%%%%%%%%
%%%%\end{document}

\newpage

\section{Introduction and outline of results}

This paper is concerned with the porous medium equation (PME for short)
\begin{equation}\label{pme}
\begin{cases}
u_t=\Delta u^m  & \textrm{in } M \times \mathbb{R}^+ \, , \\
u(\cdot,0)=u_0  & \textrm{in } M  \, ,
\end{cases}
\end{equation}
where $m>1$, $\Delta$ denotes the Laplace-Beltrami operator on a Riemannian manifold $(M,g)$ without border and the initial datum $u_0$ is assumed to be
nonnegative, bounded and compactly supported. The main assumption on $M$ is that it is a
\it Cartan-Hadamard manifold\rm, namely that it is complete, simply connected and has everywhere \it nonpositive \rm sectional curvature.

The study of the PME on such kind of manifolds is quite recent, and the first results in this connection concern the special case in which $M={\mathbb H}^n$, the $n$-dimensional hyperbolic space, a manifold having special significance since its sectional curvature is $-1$ everywhere. In fact, two of the present authors considered recently in \cite{GM2} the case of the \it fast diffusion equation \rm  (namely, \eqref{pme} with $m<1$) in ${\mathbb H}^n$, proving precise asymptotics for positive solutions. In \cite{V}, the last of the present authors constructed and studied the fundamental tool for studying the asymptotic behaviour of general solutions of \eqref{pme} on ${\mathbb H}^n$, namely the \it Barenblatt solutions\rm, which are solutions of \eqref{pme} corresponding to a Dirac delta as initial datum. Two of the most important results of \cite{V} can be summarized as follows:

$\bullet$ The decay estimate
\begin{equation}\label{smoothing H^n}
\|u(t)\|_\infty\le C\left({\log t}/t\right)^{\frac1{m-1}}
\end{equation}
holds for all $t$ sufficiently large. This bound is quite different from the corresponding Euclidean one, where the sharp upper bound takes the form $u(x,t)\le Ct^{-\alpha(n)}$, with exponent $\alpha(n)=(m-1+(2/n))^{-1}$ which is strictly less than
$1/(m-1)$ for all $n$; the difference is smaller as $n\to\infty$. Estimate  \eqref{smoothing H^n} bears closer similarity to the estimate that one obtains when
the problem is posed in a \it bounded  Euclidean  domain \rm and zero Dirichlet boundary data are assumed, since then the sharp estimate is $u(x,t)\sim
C(x)t^{-1/(m-1)}$, differing from the \eqref{smoothing H^n} only in the logarithmic time correction.

$\bullet$ Solutions corresponding to a compactly supported initial datum are compactly supported for all times, and in particular the free boundary $\mathsf{R}(t)$ of the Barenblatt solution behaves for very large times like $\mathsf{R}(t)\sim \gamma\log t+b$ for some precise constants $\gamma, b$. In fact, the free boundary of \it general \rm solutions tends to be a sphere of radius $\mathsf{R}(t)$ with $\mathsf{R}(t)$ as above.

\medskip

Later on, it has been shown in \cite{GM3} that \eqref{smoothing H^n} holds on
a large class of Cartan-Hadamard manifolds, namely those which support a \it Poincar\'e inequality\rm, namely such that the spectrum of $\Delta$ is bounded away from zero, see e.g. \cite{H}. This includes of course ${\mathbb H}^n$, as well as every Cartan-Hadamard manifold whose sectional curvature is bounded away from zero. See \cite{BG, BGV} for some previous results.

No result is however available as concerns analogues of \eqref{smoothing H^n} when a Poincar\'e
inequality does not hold, which may be the case when the sectional curvature is negative but tends
to zero at infinity. Nor it is known if a stronger smoothing effect holds if the curvature is allowed to be unbounded below. Besides, no analogue of the free boundary behaviour, nor any related \it lower \rm bound on solutions, is known but on ${\mathbb H}^n$.

\medskip
\noindent{\bf Notations for asymptotics.} We use the following notations for asymptotic approximations as $r\to\infty$: by $ f(r) \approx g(r) $ we mean that there exist two positive constants $ c_1\ge c_0>0 $ such that $ c_0 \le f(r)/g(r) \le c_1 $, whereas by $ f(r) \sim g(r) $ we mean the more precise behaviour: $ \lim_{r \to + \infty} f(r)/g(r)=1 $.

\subsection{Description of the asymptotic results}\label{desc-res}  In this paper we aim at dealing with certain Riemannian manifolds which, to some extent, generalize hyperbolic space. More precisely, we shall consider Cartan-Hadamard manifolds whose negative curvatures have bounds from above and below that are powers of the geodetic distance.

\medskip

\noindent {\bf Quasi-hyperbolic range.} To start with, we assume that both bounds behave like $ - r^{2 \mu} $ as $ r \to \infty $, with $ \mu \in (-1,1) $ and $r=\operatorname{dist}(x,o)$, $o$ being a given pole and $\operatorname{dist}$ being Riemannian distance. We call this range  for convenience the {\sl  quasi-hyperbolic range}, since the results bear a resemblance with the study of the equation on the hyperbolic space $\mathbb{H}^n$ done in \cite{V}. In that case our results imply that positive solutions of equation \eqref{pme} starting from bounded and compactly supported initial data satisfy for all large times
\begin{equation}\label{eq: intro-infty}
\| u(t) \|_\infty^{m-1} \approx t^{-1} \, (\log t)^{\frac{1-\mu}{1+\mu}}.
\end{equation}
Moreover, if we denote by $ \mathsf{R}(t) $  the radius of the smallest ball that contains the support of the solution at time $t$ (or, equivalently, the biggest ball contained in the support), then
\begin{equation}\label{eq: intro-supp}
\mathsf{R}(t) \approx (\log t)^{\frac{1}{1+\mu}}
\end{equation}
as $ t \to +\infty $. Note that \eqref{eq: intro-infty} covers all powers of $ \log t $ between $ 0 $ and $ +\infty $ as $ \mu $ ranges in $ (-1,1) $, while \eqref{eq: intro-supp} covers all the powers between $1/2$ and $ +\infty $. These results follow from the precise space-time bound: \begin{equation*}\label{bound-alpha-intro}
\begin{aligned}
&\frac{C_0}{(t+t_0)^{\frac1{m-1}}} \! \left(\gamma_0\left[\log(t+t_0)\right]^{\frac{1-\mu}{1+\mu}}-r^{1-\mu}\right)^{\frac1{m-1}}_+ \\
 &\le u(x,t) \le \frac{C_1}{(t+t_0)^{\frac1{m-1}}} \! \left( \gamma_1 \left[\log(t+t_0)\right]^{\frac{1-\mu}{1+\mu}}-r^{1-\mu}\right)^{\frac1{m-1}}_+
\end{aligned}
\end{equation*}
$ \forall x \in M \setminus B_{R_0}(o) \, ,  \forall t \ge 0 \, $ and $R_0$ large enough. This is a \it global Harnack principle\rm, using a terminology introduced, in the case of Euclidean fast diffusion, in \cite{DKV}, \cite{BV}.

In the special case $ \mu = 0 $ we are dealing with a variation of the hyperbolic space.
In the precise case of $\mathbb{H}^n$ we recover the estimates obtained  in \cite{V}, although a more accurate analysis was carried out there in order to get a sharp estimate over \eqref{eq: intro-supp}, namely
\begin{equation}\label{eq: hyperb-intro}
\mathsf{R}(t) \sim \frac{\log t}{(m-1)(n-1)}
\end{equation}
(where the curvature is taken  as $-1$). On the other hand, in $n$-dimensional Euclidean space it is well known (see e.g. \cite{Vaz07}) that
\begin{equation}\label{eq: euclid-intro}
\| u(t) \|_\infty^{m-1} \approx t^{-\frac{n(m-1)}{n(m-1)+2}} \, , \quad \mathsf{R}(t) \approx t^{\frac{1}{n(m-1)+2}} \, ,
\end{equation}
so that, informally, estimates \eqref{eq: intro-infty} and \eqref{eq: intro-supp} can be seen as limits of \eqref{eq: euclid-intro} as $n\to+\infty$ plus some logarithmic corrections.
If we let $ \mu \to 1 $, then estimate \eqref{eq: intro-supp} reads
\begin{equation} \label{eq: intro-supp-lim}
\mathsf{R}(t) \approx (\log t)^{\frac{1}{2}} \, ,
\end{equation}
while in \eqref{eq: intro-infty} the logarithmic correction seems to disappear. In fact, by carefully tracking multiplicative constants in our barrier methods when $\mu\to1$, we are able to handle the critical case $ \mu=1 $ in the limit. We shall prove that \eqref{eq: intro-supp-lim} does hold, whereas \eqref{eq: intro-infty} yields a $ \log\log $ correction, namely
\begin{equation}\label{eq: intro-infty-lim}
\| u(t) \|_\infty^{m-1} \approx t^{-1} \, (\log \log t) \, .
\end{equation}

\noindent Again, \eqref{eq: intro-supp-lim} and \eqref{eq: intro-infty-lim} follow by a precise global Harnack principle, see formula \eqref{bound-23} below.

So far, we have not been too precise about our requirements on the behaviour of curvatures as $ r \to +\infty $. In Section \ref{sec.smr} below we shall make those assumptions clear and we shall give precise statements of our separate estimates from above and below. In fact, we shall see that a bound from \emph{below} of the type of $ -r^{2\mu} $ on the \emph{radial Ricci curvature} of $M$ is enough in order to provide \emph{pointwise bounds} from below for positive solutions to \eqref{pme} with compactly supported initial data, whereas an analogous bound from \emph{above} on \emph{radial sectional curvatures} allows us to get pointwise bounds from above. The proofs are contained in Sections \ref{sec.proofs}-\ref{sect: proofs-2}.

\medskip

\noindent {\bf Weighted PME.} Another major contribution of the paper is the application of these results on the PME posed on  model manifolds (a particular case of the Cartan-Hadamard manifolds) into the study of the weighted version of the same equation posed in Euclidean space. This is done in Section \ref{sec: chvar} and uses an interesting  radial change of variables introduced in \cite{V}. It leads to an equivalent formulation of the radial problem in a different context. In the end, the particular weight depends on the starting geometry, see the summarizing table in Section \ref{sect: chvar-exa}. In the standard examples the weight is a \emph{logarithmic} correction of the critical density $\rho(x)=|x|^{-2}$, a case where the mathematical analysis is particularly difficult due to the complete homogeneity under scaling, see e.g. \cite{NR, IS}.

In this paper we use the equivalence, in some cases to transfer results from the geometric setting to the Euclidean weighted setting (we refer to Theorem \ref{thm: peso} below), in other cases to transfer them in the opposite direction, as we will see in the next paragraph.

\medskip

\noindent {\bf Quasi-Euclidean range.} The situation is completely different in the range $\mu\le -1$ that corresponds to very small curvatures at infinity and can be seen as {\sl quasi-Euclidean} range. In fact, if $\mu<-1$ we show that the Euclidean-type bounds \eqref{eq: euclid-intro} still hold.  This uses known results for the PME in Euclidean space and the equivalence of formulations explained above. Manifolds $M$ whose curvature decays faster than $\operatorname{dist}(x,o)^{2\mu}$ with $\mu<-1$ and $o$ fixed have been widely investigated in the literature since it is known that such condition implies that $M$ has finite topological type, see \cite{A} (a fact which need not be true when $\mu=-1$, see \cite{LS})  and in some precise sense such manifolds are ``close'' to ${\mathbb R}^n$, see e.g. \cite{GPZ}.

\medskip

There are two critical exponents in this classification, $\mu=1$ and $\mu=-1$.
The lower critical exponent $ \mu = -1 $ is very interesting: assuming that curvatures behave like $ -Q_1 r^{-2} $ as $r\to\infty$,  we will show that the decay estimates have exponents that depend on the coefficient $Q_1$. In particular, the sup norm estimates fill the gap between $t^{-1/(m-1)}$ and the Euclidean exponent $t^{-n/(n(m-1)+2)}$. This is stated in Section \ref{sec.smr.qEr} and proved  in full detail in Section \ref{sec.proofs-euclid}.

 %Basically amounts to considering a Euclidean-like space in a ``fractional'' dimension, so that we expect solutions to exhibit a behaviour that resembles the (by now) well-established Euclidean one. \normalcolor

\medskip
%
%\color{red} {\bf To do now.} Calculate the radius for the Free Boundaries in the quasi euclidean cases. Start from the FB of weighted Euclidean case that is given in [RV] as $|x(t)|\approx t^{1/(n(m-1)+2-mp)}$.
%\normalcolor

The figure below describes the variation with $\mu$  of the exponents $\alpha,\beta$ in the bound
$$\|u(t)\|_\infty^{m-1}\approx t^{-\alpha}(\log t)^{\beta} \, . $$
It should be remarked that when $\mu>1$ it will be shown elsewhere that a bound of the form $\|u(t)\|_\infty\approx t^{-1/(1-m)}$, which bears close analogy with the case of the PME posed on bounded domains with Dirichlet boundary conditions, holds, thus justifying the corresponding part of the figure.
\begin{center}
\begin{figure}[!ht]
\centering
\includegraphics[width=0.98\columnwidth, height=0.41\textheight]{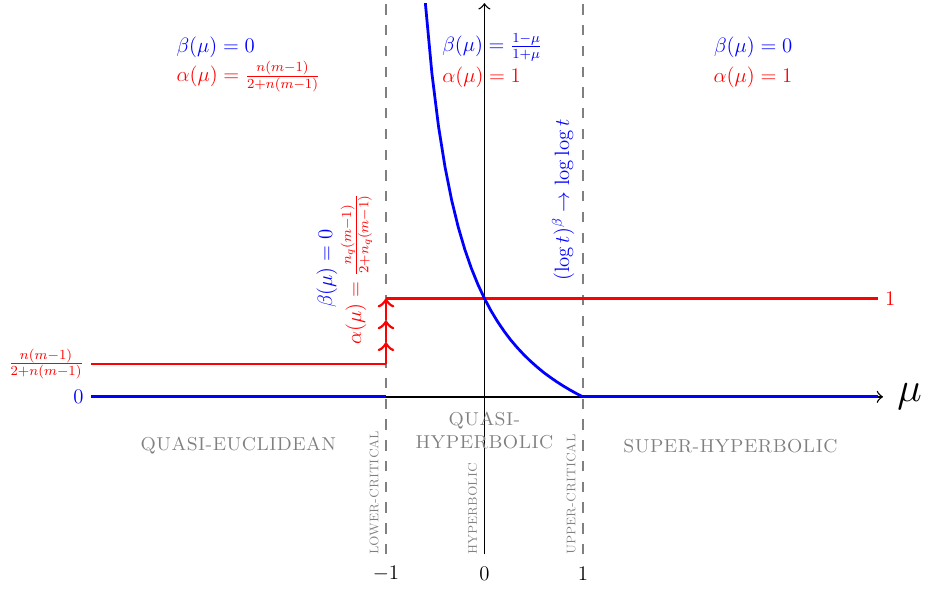}
\caption{The exponents $\alpha,\beta$ in the bound $\|u(t)\|_\infty^{m-1} \approx \, t^{-\alpha}(\log t)^{\beta}$.}
\end{figure}
\end{center}

\medskip

\noindent {\bf Organization of the paper.} After a section on functional and geometric preliminaries, Section \ref{geom}, we present the detailed statement of the results in the quasi-hyperbolic case in Section \ref{sec.smr}, followed by the proofs in the subcritical range in Section \ref{sec.proofs}, while the critical case $\mu=1$ is covered in Section
\ref{sect: proofs-2}. We then state the results of the Euclidean range in Section \ref{sec.smr.qEr}, but the proof is delayed becuase it relies on the study of the transformation of equations for radial solutions that leads to the equivalent problem for a weighted version of the Euclidean PME, which is done in Section \ref{sec: chvar}. The proof of the main results in the quasi-Euclidean range is done in Section \ref{sec.proofs-euclid}. Section \ref{sect: chvar-asymp} is devoted to the corresponding asymptotic results for the weighted, Euclidean PME.

In the last section we gather a number of comments and extensions. Two of them are worth mentioning here:

(1) We do not cover the supercritical case $ \mu>1 $ which is related to manifolds with very negative curvature. It has not been included because it represents a yet different situation that needs new techniques.

(2) We make a first attempt at comparison with the linear case, $m=1$, where there is of course some available information. To our surprise, this information seems not to be explicit at least at the level of detail of our nonlinear results above.

\normalcolor

\medskip

%%%%%%%%%%%%%%%%%%%%%%%%%%%%%%%%%%%%%%%%%%%%%%%%%%%%%%%

\section{Preliminaries}\label{geom}

\subsection{Existence and uniqueness of solutions} A satisfactory theory of existence and uniqueness for \eqref{pme} when initial data are positive, finite Radon measure, has been given very recently in \cite{GMP},  when $n\ge3$, on the class of manifolds we study here, namely Cartan-Hadamard manifolds such that \normalcolor
\[
\textrm{Ric}_{o}(x)\geq -C(1+\operatorname{dist}(x,o)^2)\ \textrm{for some }C\ge0,
\]
where $o$ is a given pole, $\operatorname{dist}$ denotes Riemannian distance, and $\textrm{Ric}_{o}$ is Ricci curvature in the radial direction. The case $n=2$ has some peculiarities, due to the methods of proof that we have used.  Notice that the above condition is known to be sharp for stochastic completeness, and hence for conservation of mass, in the linear setting, see \cite{Grig}. However,  for any $n\ge2$ \normalcolor existence and uniqueness of strong solutions e.g. for smooth compactly supported data hold by standard arguments, see \cite{Vaz07}, where we use the following concept of solution:

\begin{defn} We say that $u$ is a  weak solution to \eqref{pme}, given $u_0\in L^1(M)$, if the following holds:
\begin{equation*}\label{e65}
u\in C([0,+\infty); L^1(M)) \cap
L^\infty(M\times(\tau,\infty)) \quad \textrm{for all}\;\,
\tau>0\,,
\end{equation*}
\begin{equation*}\label{e66}
\nabla\left(u^m\right) \in L^2_{\rm loc}((0,\infty);L^2(M))\,,
\end{equation*}
\begin{equation*}\label{e67}
-\int_0^\infty\int_M u(x,t)\varphi_t(x,t) \, d\mathcal V(x) dt +
\int_0^\infty\int_M \langle \nabla (u^m)(x,t), \nabla
\varphi(x,t)\rangle \, d\mathcal V(x) dt \,=\,0
\end{equation*}
for any $\varphi\in C^\infty_c(M\times (0,\infty))$, with $\lim_{t\to0}u(t)=u_0$ in $L^1(M)$. We also say that $u$ is strong if in addition $u_t\in L^\infty_{\textrm{loc}}((0,+\infty);L^1(M))$.
%\begin{equation}\label{e68}
%\lim_{t\to 0} \int_M u(x,t)\phi(x) \, d\mathcal
%V(x)\,=\,\int_{M}\phi(x) \, d\mu(x)\,\quad \textrm{for any}\;\;
%\phi\in C_b(M):=C(M)\cap L^\infty(M)\,.
%\end{equation}
\end{defn}
Notice that standard comparison theorems hold for sub- and supersolutions, provided they are strong, as well. This will be crucial below, and in fact the sub- and supersolutions we shall construct will satisfy such property.

%\color{red} I think nothing else is necessary. Please tell me if you agree or not \normalcolor

%\color{red} / We should put some references for the kind of solutions we consider. But here we have to be very minimal: we are not interested in a general theory, since we only deal with bounded and compactly supported initial data. /
%\normalcolor

\subsection{Geometric preliminaries}\label{sect: geom-pre}

We remind the reader that {\it Cartan-Hadamard} manifolds are defined to be complete, simply connected Riemannian manifolds with nonpositive sectional curvatures. It is known that they are diffeomorphic to ${\mathbb R}^n$ and that the \it cut locus \rm of any point $o$ is empty (see, e.g. \cite{Grig}, \cite{Grig3}).

One then defines {\it polar coordinates} around any given pole $o$. Set indeed, given any $x\in M\setminus \{o\}$, $r(x) := d(x, o)$ and define $\theta\in \mathbb S^{n-1}$ so that the geodesic from $o$ to $x$ starts
at $o$ with the direction $\theta$ in the tangent space $T_oM$. By construction, $\theta$ can be regarded as a point of $\mathbb S^{n-1}:=\{x\in \mathbb R^{n}:\,|x|=1\}.$

The Riemannian metric in $M\setminus \{o\}$
in such coordinate system will then read
\[ds^2 = dr^2+A_{ij}(r, \theta)d\theta^i d\theta^j,\]
where $(\theta^1, \ldots, \theta^{n-1})$ are coordinates in $\mathbb S^{n-1}$ and $(A_{ij})$ is a positive definite matrix. As a consequence,
the Laplace-Beltrami operator in the given coordinates has the form
\begin{equation}\label{e1}
\Delta = \frac{\partial^2}{\partial r^2} + m(r, \theta)\frac{\partial}{\partial
r}+\Delta_{S_{r}},
\end{equation}
where $m(r, \theta):=\frac{\partial}{\partial r}\big(\log\sqrt{A}\big)$, $ A:=\det (A_{ij})$, $ \Delta_{S_r}$ is the Laplace-Beltrami operator on $S_{r}:=\partial B(o, r)$, $ B(o,r) $ denoting the Riemannian ball of radius $r$ centered at $o$.
We say that $M$ is a {\it model manifold}, if it is endowed with a Riemannian metric of the form
\begin{equation*}\label{e2}
ds^2 = dr^2+\psi^2(r)d\theta^2,
\end{equation*}
where $r$ then denotes Riemannian distance from a given pole $o\in M$, $d\theta^2$ is the standard metric on $\mathbb S^{n-1}$, and $\psi$ belongs to the set of functions $\mathcal A$
defined as
\begin{equation}\label{eq: class-mod}
\mathcal A:=\left\{ \psi \in C^\infty((0,+\infty))\cap C^1([0,+\infty)): \ \psi'(0)=1 \, , \ \psi(0)=0 \, , \ \psi>0 \ \textrm{in} \ (0,+\infty)\right\} .
\end{equation}
 That is, one requires spherical symmetry of the given metric. To be precise, we may write $M\equiv M_\psi$; furthermore, we have $\sqrt{A(r,\theta)}=\psi^{n-1}(r) \, \eta(\theta) $ (for a suitable function $\eta$), so that
\begin{equation}\label{formula-laplacian-model}
\Delta = \frac{\partial^2}{\partial r^2}+ (n-1)\frac{\psi'}{\psi}\frac{\partial}{\partial r}+ \frac1{\psi^2}\Delta_{\mathbb S^{n-1}}\,,\end{equation}
where $\Delta_{\mathbb S^{n-1}}$ is the Laplace-Beltrami operator in $\mathbb S^{n-1}\,.$

It is well known that for $\psi(r)=r$ one recovers the Euclidean space, namely $M=\mathbb R^n$, while for $\psi(r)=\sinh r$ one gets the $n-$dimensional hyperbolic
space, namely $M=\mathbb H^n$. The choice $\psi(r)=\sin r$ yields the unit sphere, although we shall not be interested in positive curvature here.

For any $x\in M\setminus\{o\}$, denote by $\textrm{Ric}_o(x)$ the Ricci curvature at $x$ in the radial direction $\frac{\partial}{\partial r}$. Moreover,
denote by $\textrm{K}_{\omega}(x)$ the {\it sectional curvature} at the point $x\in M$ of the $2$-section determined by a tangent plane $\omega$ including
$\frac{\partial}{\partial r}$. Note that if $M_\psi$ is a model manifold, then for any $x=(r, \theta)\in M_\psi\setminus\{o\}$ we have
\[\textrm{K}_{\omega}(x)=-\frac{\psi''(r)}{\psi(r)}. \] Sectional curvature $\textrm{H}_\omega(x)$ w.r.t. planes orthogonal to such direction is given by
\[
\textrm{H}_\omega(x)=(1-\psi'^2)/\psi^2.
\]
\normalcolor Finally,
\[\textrm{Ric}_{o}(x)=-(n-1)\frac{\psi''(r)}{\psi(r)}\,,\]
where by $\textrm{Ric}_o(x)$ we denote the Ricci curvature at $x$ in the radial direction $\frac{\partial}{\partial r}$.
\smallskip

We shall make an essential use of comparison results for sectional and Ricci curvatures. By classical results (see e.g. \cite{GW},
\cite[Section 15]{Grig}), if
\begin{equation*}\label{e3a}
\textrm{K}_{\omega}(x)\leq -\frac{\widetilde \psi''(r)}{\widetilde \psi(r)}\quad \textrm{for all}\;\; x=(r,\theta)\in M\setminus\{o\},
\end{equation*}
for a given $\widetilde \psi\in \mathcal A$, then
\begin{equation*}\label{e3}
m(r, \theta)\geq (n-1)\frac{\widetilde \psi'(r)}{\widetilde \psi(r)}\quad \textrm{for all}\;\; r>0, \, \theta \in \mathbb S^{n-1}\,.
\end{equation*}

On the other hand, if
\[\textrm{Ric}_{o}(x)\geq
-(n-1)\frac{\psi''(r)}{\psi(r)}\quad \textrm{for all}\;\; x=(r,\theta)\in M\setminus\{o\},\] for some function $\psi\in \mathcal A$, then
\begin{equation*}\label{e4}
m(r, \theta)\leq (n-1)\frac{\psi'(r)}{ \psi(r)}\quad \textrm{for all}\;\; r>0, \, \theta \in \mathbb S^{n-1}\,.
\end{equation*}

Of course the two above bounds and the expression \eqref{e1} for the Riemannian Laplacian yield, for radial (i.e. depending only on geodesic distance from $o$) and monotonic functions, comparison results for the Laplacian in terms of the Laplacian of a model manifold defined by the function $\psi$ in terms of which the curvature bounds are supposed to hold. This will be of crucial importance below.

\begin{rem}\label{oss: varieta-con-polo} \rm
Our results still hold under milder assumptions on $M$. In fact, it suffices to suppose that $M$ is a manifold with a pole $o$ in the sense (see e.g.~\cite{GW}) that the exponential map at $o$ is a diffeomorphism between $T_oM$ and $M$, and that the appropriate bounds on \it radial \rm curvatures required throughout the present paper are satisfied.
\end{rem}
%%%%%%%%%%%%%%%%%%%%%%%%%%%%%%%%%%%%%%%%%%%%%%%%%%%%%%%%%%%%%%%%%

\subsection{Relative behaviour of $\boldsymbol \psi$ and the curvature}\label{sect: exa-mod}

Note first that  the quantity that appears in the Laplacian of a model manifold is $ g=\psi'/\psi=(\log(\psi))'$. Then the (sectional, radial) curvatures are given by
\begin{equation}\label{eq: riccati-g}
-K(r)=\frac{\psi''(r)}{\psi(r)}= g(r)^2+ g'(r)
\end{equation}
(up to constants).  Since we are interested in nonnegative curvatures it follows that $\psi$ must be convex, hence increasing since $\psi'(0)=1$.  For the reader's benefit we shall use a cut-and-paste approach to show the relation between functions $\psi$, $g$ and $K$ that will be useful later on. We will basically consider only large enough values of $r$.

\medskip

\noindent {\bf Type I.} We first consider functions $\psi$ with exponential growth, more precisely, we take $\psi(r)= Ae^{a_1 \, r^{\alpha}}+B$ with $a_1, \alpha, A > 0 $, $B\in{\mathbb R}$, for $r\ge\overline{r}$ and $\overline{r}$ large enough so that in particular $\psi>0$ for $r\ge\overline{r}$. Then we have, again for $r\ge\overline{r}$:
$$
g(r)\sim a_1 \, \alpha \, r^{\alpha -1} \, .
$$
For $\alpha>0$ we get  $K(r)\sim g^2(r)$, hence
$$
K(r)\sim -a_1^2 \, \alpha^2 \, r^{2\alpha-2} \, ,
$$
i.e.~we get $-K(r)\sim r^{2\mu}$ with $\mu=\alpha-1>-1$. Note that $\alpha=1$ gives us the hyperbolic space. An explicit choice of $\psi$ on $[0,+\infty)$ which gives rise to a $C^1$ Cartan-Hadamard manifold matching such asymptotic behaviour is the following:
\[
\psi(r)=\begin{cases}
r &\textrm{if}\ r\in[0,\overline{r}]\\
%\frac{\overline{r}^{1-\alpha}}{\alpha a_1}e^{a_1(r^\alpha-\overline{r}^\alpha)}+\overline{r}-\frac{\overline{r}^{1-\alpha}}{\alpha a_1}
A\left(e^{a_1r^\alpha}-e^{a_1\overline{r}^\alpha}\right)+\overline{r}&\textrm{if}\ r>\overline{r},
\end{cases}
\]
where $\overline{r}$ is the unique solution of
\[
A\alpha a_1\overline{r}^{\alpha-1}e^{a_1\overline{r}^\alpha}=1
\]
when $\alpha>1$, whereas when $\alpha\in(0,1)$ one has to require that
 \[
 A\le\frac{1}{\alpha a_1}\left(\frac{1-\alpha}{e\alpha a_1}\right)^{\frac{1-\alpha}\alpha}
 \]
 and then $\overline{r}$ is the unique solution to the above equation satisfying the condition $\overline{r}>\left[(1-\alpha)/(\alpha a_1)\right]^{1/\alpha}$. If $\alpha=1$ one requires $A< 1/a_1$ and then takes $\overline{r}=-[\log (a_1A)]/a_1$.
% \[\overline{r}:=\begin{cases}1&\textrm{if}\ \alpha\ge1\\
%\left(\frac{1-\alpha}{a_1\alpha}\right)^{\frac1{\alpha}}&\textrm{if}\ \alpha\in(0,1).\end{cases}\]
%On the other hand, if we take $\alpha<0$ we get for large $r$:
%$$
%K(r)\sim g'(r) = - a_1 \, \alpha(\alpha-1) \, r^{\alpha-2},
%$$
%so that $2\mu=\alpha-2$ and apparently we cover all the cases $\mu<-1$. But we cannot have at the same time negative curvature and positive $g$, so it does not correspond to a Cartan-Hadamard manifold. So we will not consider here such an alternative.

\medskip

\noindent {\bf Type II.} In order to get the critical case $\mu=-1$ we have to change the form of $\psi$ into a lesser growth. We thus consider  powers, i.e. for large $r$ we set $\psi(r)=a_1 \, r^{\alpha}+B$ with $ a_1, \alpha>0 $, $B\in{\mathbb R}$. Then,
$$
g(r)\sim \alpha/r \, ,
$$
so that
$$
K(r)\sim-\alpha(\alpha-1) \, r^{-2} \,.
$$
We get zero curvature for $\alpha=1$ (Euclidean space), while negative curvature implies that $\alpha>1$ (or  $\alpha<0$ that is excluded since $\psi$ is increasing). Note that $\alpha=0$ gives the 1-dimensional Laplacian (in fact the manifold in this case is a sort of cylinder), which is also, formally, a case of zero (radial) curvature.

An explicit choice of $\psi$ on $[0,+\infty)$ which gives rise to a $C^1$ Cartan-Hadamard manifold matching such asymptotic behaviour, in the relevant case $\alpha>1$, is the following:
\[
\psi(r)=\begin{cases}
r &\textrm{if}\ r\in[0,\overline{r}]\\
a_1\,r^\alpha+\overline{r}\,\frac{\alpha-1}\alpha&\textrm{if}\ r>\overline{r},
\end{cases}
\]
where $\overline{r}:=(a_1\alpha)^{-\frac1{\alpha-1}}.$
%Again, the case $ \alpha<0 $ does not correspond to a Cartan-Hadamard manifold, so we drop it here.

\medskip

\noindent {\bf Type III.}  A limit case of the previous calculation is obtained by perturbing the Euclidean case $\psi(r)=c r$ and considering, for large $r$, $\psi(r)=c r\log(r)^{\alpha}+B$ for some $ c>0 $, $ \alpha \neq 0 $, $B\in\mathbb{R}$. We obtain
$$
g(r) \sim\frac1{r}+\frac{\alpha}{r\log(r)}
$$
and
$$
-K(r)\sim\frac{\alpha}{r^2\log(r)}+\frac{\alpha(\alpha-1)}{r^2(\log r)^2} \,,
$$
so that $-K(r)$  is positive for large $r$ only if $\alpha>0$. In that case
$$ -K(r)\sim \alpha r^{-2}(\log r)^{-1} \, . $$
An explicit choice of $\psi$ on $[0,+\infty)$ which gives rise to a $C^1$ Cartan-Hadamard manifold matching such asymptotic behaviour can be done by matching the outer behaviour with a standard behaviour near $o$, as done before. We omit the easy details.

\medskip

\noindent {\bf Type IV.}  Finally, we want to find Cartan-Hadamard models such that $ K(r) \approx r^{2\mu} $ as $ r \to +\infty $ with $ \mu < -1 $. To this aim, we consider a still smaller perturbation of the Euclidean case of the form $ \psi(r)=Ar e^{c r^{-\alpha}} +B$ for large $r$ and for some $ c \neq 0 $, $ \alpha>0 $, $B\in\mathbb{R}$. We get
$$  g(r)\sim\frac{1}{r} - \frac{\alpha \, c}{r^{\alpha+1}} $$
and
$$ -K(r)\sim \frac{\alpha(\alpha-1)c}{r^{\alpha+2}} + \frac{\alpha^2 c^2}{r^{2\alpha+2}} \, . $$
Let first $ \alpha \neq 1 $. In that case we have $2\mu= -(\alpha+2)$. Negative curvature at infinity is obtained for  $ \alpha>1 $ if  $ c>0 $ and then we cover the range
$ \mu < -3/2 $, or $ \alpha \in (0,1) $ and $ c<0 $, and then we get all $ \mu \in (-3/2,-1) $. As for the case $ \mu = -3/2 $, may take  $ \psi(r)=Ar e^{-c \, \frac{\log r}{ r}} +B$, and then
$$
g(r)\sim \frac{1}{r} - \frac{c}{r^2} \, (1-\log r)
$$
and
$$ -K(r) \sim  \frac{c}{r^3} + \frac{c^2(1-\log r)^2}{r^4} \, , $$
so that for $ c>0 $ we recover, approximately, the case $ \mu = -3/2 $ as well.

An explicit choice of $\psi$ on $[0,+\infty)$ which gives rise to a $C^1$ Cartan-Hadamard manifold matching such asymptotic behaviour is constructed as follows. Let $\alpha>1, c>0$ or $\alpha\in(0,1),c<0$ hold. Assume also that $A>1$. Then we set
\begin{equation*}
\psi(r)=\begin{cases}
r &\textrm{if}\ r\in[0,\overline{r}]\\
A\left(r\,e^{cr^{-\alpha}}-\overline{r}\,e^{c\overline{r}^{-\alpha}}\right)+\overline{r}&\textrm{if}\ r>\overline{r},
\end{cases}
\end{equation*}
where $\overline{r}$ is the unique solution to the equation
\[
1=A\,e^{cr^{-\alpha}}\left(1-\alpha c r^{-\alpha}\right).
\]
In the case $\mu=-3/2$, hence $c>0$ to get a negative $K$, we let $A>1$ and set
\begin{equation*}
\psi(r)=\begin{cases}
r &\textrm{if}\ r\in[0,\overline{r}]\\
A\left(r\,e^{-c\frac{\log r}r}-\overline{r}\,e^{-c\frac{\log \overline{r}}{\overline{r}}}\right)+\overline{r}&\textrm{if}\ r>\overline{r},
\end{cases}
\end{equation*}
where $\overline{r}$ is the unique solution to
\[
1=A\,e^{-c\frac{\log r}r}\left[1-\frac cr\left(1-\log r\right)\right].
\]

Finally, note that the inverse problem of finding $\psi$ given $-K(r)$ can be viewed as the problem of finding first $g$ and then integrating. Now, the equation for $g$ given $K$, namely \eqref{eq: riccati-g}, is a Riccati equation that does not have closed formulas for the general solution. It can also be seen as the inverse problem for the Schr\"odinger equation $\psi''(r)+K(r)\psi(r)=0$. %\color{red} Expand this?
%\normalcolor

%%%%%%%%%%%%%%%%%%%%%%%%%%%%%%%%%
\section{Statement of the results in the quasi-hyperbolic range}\label{sec.smr}

Throughout this whole section and on, it is understood that $ o \in M $ is a given pole, and we implicitly set
$$ r=r(x):=\operatorname{dist}(x,o) \quad \forall x \in M \, . $$
Here we state our main results concerning lower and upper bounds for solutions to \eqref{pme} when our power-type assumptions on the curvatures of $M$ at infinity are quasi-hyperbolic and subcritical, namely when the power lies \emph{strictly} between $-2$ and $2$.

\begin{thm}[Upper bounds, quasi-hyperbolic subcritical powers]\label{pro: ltb-upper}
Let $ M $ be an $n$-dimensional ($ n \ge 2 $) Cartan-Hadamard manifold. Let $ u $ be the solution of the porous medium equation \eqref{pme} starting from a nonnegative, bounded and compactly supported initial datum $u_0 \not\equiv 0 $. Suppose that
\begin{equation}\label{hp-sect}
\mathrm{K}_\omega(x) \le -Q_1 \, r^{2 \mu_1} \quad \forall x \in M \setminus B_R(o)
\end{equation}
for some $ \mu_1 \in (-1,1) $ and $ Q_1,R>0 $, where $ \mathrm{K}_\omega(x) $ denotes the sectional curvature at $ x $ corresponding to any $2$-dimensional tangent plane $ \omega $ containing the radial direction w.r.t.~$o$. Then the following bound holds:
\begin{equation}\label{upper1}
u(x,t) \le \frac{C_1}{(t+t_1)^{\frac1{m-1}}} \left(\gamma_1\left[\log(t+t_1)\right]^{\frac{1-\mu_1}{1+\mu_1}}-r^{1-\mu_1}\right)^{\frac1{m-1}}_+
\end{equation}
$$ \forall x \in M \setminus B_{R_1}(o) \, , \quad \forall t \ge 0 \, , $$
where $ C_1 ,  \gamma_1 , R_1 ,  t_1 $ are positive constants depending on $ n,m,\mu_1,Q_1,R,u_0 $. In particular, the estimates
\begin{equation}\label{bound-smooth-infty}
\| u(t) \|_\infty^{m-1} \le \widetilde{C}_1 \, \frac{(\log t)^{\frac{1-\mu_1}{1+\mu_1}}}{t} \quad \textrm{and} \quad \mathsf{R}(t) \le  \widetilde{\gamma}_1 \, (\log t)^{\frac{1}{1+\mu_1}}
\end{equation}
hold for suitable $ \widetilde{C}_1 , \widetilde{\gamma}_1  >0 $ and all $t$ large enough, where $ \mathsf{R}(t) $ is the radius of the smallest ball that contains the support of the solution at time $t$.
\end{thm}

\begin{thm}[Lower bounds, quasi-hyperbolic subcritical powers]\label{pro: ltb-lower}
Let $M$ and $u$ be as in Theorem \ref{pro: ltb-upper}. Suppose that
\begin{equation}\label{hp-ricc}
\mathrm{Ric}_o(x) \ge -(n-1)\,Q_2 \,  r^{2 \mu_2} \quad \forall x \in M \setminus B_R(o)
\end{equation}
for some $ \mu_2 \in (-1,1) $ and $ Q_2,R>0 $, where $ \mathrm{Ric}_o(x) $ is the Ricci curvature at $x$ in the radial direction w.r.t.~$o$. Then the following bound holds:
\begin{equation}\label{bound-1}
u(x,t) \ge \frac{C_2}{(t+t_2)^{\frac1{m-1}}} \left(\gamma_2\left[\log(t+t_2)\right]^{\frac{1-\mu_2}{1+\mu_2}}-r^{1-\mu_2}\right)^{\frac1{m-1}}_+
\end{equation}
$$ \forall x \in M \setminus B_{R_2}(o) \, , \quad \forall t \ge 0 \, , $$
where $ C_2 ,  \gamma_2 , R_2 ,  t_2 $ are positive constants depending on $ n,m,\mu_2,Q_2,R,u_0 $ and on $ \inf_{x \in B_R(o)} \mathrm{Ric}_o(x) $. In particular, the estimates
\begin{equation}\label{bound-smooth-infty-below}
\| u(t) \|_\infty^{m-1} \ge \widetilde{C}_2 \, \frac{(\log t)^{\frac{1-\mu_2}{1+\mu_2}}}{t} \quad \textrm{and} \quad \mathsf{R}(t) \ge  \widetilde{\gamma}_2 \, (\log t)^{\frac{1}{1+\mu_2}}
\end{equation}
hold for suitable $ \widetilde{C}_2 , \widetilde{\gamma}_2  >0 $ and all $t$ large enough, where $ \mathsf{R}(t) $ is the radius of the biggest ball that is contained in the support of the solution at time $t$.
%, s.t. $C_0,\gamma_0$ are sufficiently small and $t_0$
%is sufficiently large.
\end{thm}

In the quasi-hyperbolic and critical case, that is when curvatures are quadratic, we have significantly different bounds.

\begin{thm}[Upper bounds, quasi-hyperbolic critical power]\label{pro: ltb-cri-upper}
Let $ M $, $ \mathrm{K}_\omega $, $u$ and $ \mathsf{R}(t) $ be as in  Theorem \ref{pro: ltb-upper}. Suppose that
\begin{equation}\label{hp-sect-cri}
\mathrm{K}_\omega(x) \le -Q_1 \, r^{2} \quad \forall x \in M \setminus B_R(o)
\end{equation}
for some $ Q_1,R>0 $. Then the following bound holds:
\begin{equation}\label{bound-22}
u(x,t) \le \frac{\kappa_1}{(t+t_1)^{\frac1{m-1}}} \left[ \eta_1  + \frac12 \, \log\log(t+t_1)  - \log r  \right]_+^{\frac1{m-1}}
\end{equation}
$$ \forall x \in M \setminus B_{R_1}(o) \, , \quad \forall t \ge 0 \, , $$
where $  \kappa_1 , \eta_1 , R_1 , t_1 $ are positive constants depending on $ n,m,Q_1,R,u_0 $. In particular, the estimates
\begin{equation}\label{smoothing2}
\|u(t)\|_\infty \le \widetilde{\kappa}_1 \left(\frac{\log \log t}t\right)^{\frac1{m-1}} \quad \textrm{and} \quad \mathsf{R}(t) \le \widetilde{\eta}_1 \, (\log t)^{\frac12}
\end{equation}
hold for suitable $ \widetilde{\kappa}_1 , \widetilde{\eta}_1 >0 $ and all $t$ large enough.
\end{thm}

\begin{thm}[Lower bounds, quasi-hyperbolic critical power]\label{pro: ltb-cri-lower}
Let $M,u$ be as in Theorem \ref{pro: ltb-upper} and $ \mathrm{Ric}_o , \mathsf{R}(t) $ as in Theorem \ref{pro: ltb-lower}. Suppose that
\begin{equation}\label{hp-ricc-cri}
\mathrm{Ric}_o(x) \ge -(n-1)\,Q_2 \,  r^{2} \quad \forall x \in M \setminus B_R(o)
\end{equation}
for some $ Q_2,R>0 $. Then the following bound holds:
\begin{equation}\label{bound-2}
u(x,t) \ge\frac{\kappa_2}{(t+t_2)^{\frac1{m-1}}} \left[ \eta_2  + \frac12 \, \log\log(t+t_2) - \log r  \right]_+^{\frac1{m-1}}
\end{equation}
$$ \forall x \in M \setminus B_{R_2}(o) \, , \quad \forall t \ge 0 \, , $$
where $  \kappa_2 , R_2 , t_2 $ are positive constants and $ \eta_2  $ is a negative constant depending on $ n,m,Q_2,R,u_0 $ and on $ \inf_{x \in B_R(o)} \mathrm{Ric}_o(x) $. In particular, the estimates
\begin{equation}\label{smoothing2-low}
\|u(t)\|_\infty \ge \widetilde{\kappa}_2 \left(\frac{\log \log t}t\right)^{\frac1{m-1}} \quad \textrm{and} \quad \mathsf{R}(t) \ge \widetilde{\eta}_2 \, (\log t)^{\frac12}
\end{equation}
hold for suitable $ \widetilde{\kappa}_2 , \widetilde{\eta}_2 >0 $ and all $t$ large enough.
\end{thm}

\begin{rem}\rm
These results motivate a number of comments.
\begin{enumerate}[i)]
\item The case of the hyperbolic space treated in \cite{V} basically corresponds to the choice $\mu_1=\mu_2=0$.
\item The situation is particularly clear when $\mu_1=\mu_2=\mu$. In that case upper and lower bounds match and our results read
\begin{equation}\label{bound-alpha}
\begin{aligned}
&\frac{C_0}{(t+t_0)^{\frac1{m-1}}} \! \left(\gamma_0\left[\log(t+t_0)\right]^{\frac{1-\mu}{1+\mu}}-r^{1-\mu}\right)^{\frac1{m-1}}_+ \\
 &\le u(x,t) \le \frac{C_1}{(t+t_0)^{\frac1{m-1}}} \! \left( \gamma_1 \left[\log(t+t_0)\right]^{\frac{1-\mu}{1+\mu}}-r^{1-\mu}\right)^{\frac1{m-1}}_+
\end{aligned}
\end{equation}
$ \forall x \in M \setminus B_{R_0}(o) \, ,  \forall t \ge 0 \, , $, $R_0$ large enough, which bears some similarity with the \it global Harnack principle \rm introduced, in the case of Euclidean fast diffusion, in \cite{DKV}, \cite{BV}.

\item In particular, in that case the support of the solution is contained in a ball whose radius is comparable, for large $t$, with
    $$\mathsf{R}(t)=[\log(t)]^{1/(\mu+1)}\,,
    $$
    and contains a ball whose radius is comparable with the same quantity. Of course \eqref{bound-alpha} also shows the qualitative sharpness of our results.
\item It is remarkable that our results are formally the same both when $\mu_1,\mu_2\in(-1,0)$ and when $\mu_1,\mu_2\in(0,1)$. In fact, the first case corresponds to curvatures \it tending to zero polynomially at infinity\rm, whereas the second corresponds to curvatures bounded in terms of quantities \it diverging polynomially at infinity\rm. We find particularly striking the fact that the bounds valid in the first case are qualitatively similar to the ones valid in ${\mathbb H}^n$, and \it not \rm to the ones valid in ${\mathbb R}^n$, although curvatures tend to zero at infinity. In particular, the function of time appearing in the r.h.s.~of \eqref{bound-smooth-infty} does not depend on the space dimension even in that case.

\item We deal with long-time asymptotics only, short-time bounds being of course identical to the Euclidean ones.

%\item  To which data can the smoothing effect \eqref{bound-smooth-infty} be extended? {\color{red}One should find the dependence of constants on data,  if interested in this problem. MAYBE IT'S BETTER NOT TO MENTION IT, SINCE THIS PAPER REALLY DOESN'T DEAL WITH SMOOTHING EFFECTS, WE ONLY CONSIDER BOUNDED COMPACTLY SUPPORTED DATA (AT MOST SAY SOMETHING IN SECTION \ref{sect: comm-opp}).}
%\normalcolor
%\item The requirement $n\ge3$ in Theorems \ref{pro: upper-qe1}--\ref{pro: lower-qe1} depends on the fact that the proof of such results relies on a suitable change of variable which is not available as such when $n=2$.

\item
If both curvature bounds \eqref{hp-ricc-cri} and \eqref{hp-sect-cri} hold, of course the two-sided bound

\begin{equation}\label{bound-23}\begin{aligned}
&\kappa_0 \! \left[ \eta_0 \!  + \! \frac{\log\log(t+t_0)}{2} \! - \! \log r  \right]_+^{\frac1{m-1}} \\ & \le \! \frac{u(x,t)}{(t+t_0)^{\frac1{m-1}}} \! \le \! \kappa_1 \! \left[ \eta_1 \! + \! \frac{\log\log(t+t_0)}{2} \! - \! \log r  \right]_+^{\frac1{m-1}}
\end{aligned}\end{equation}

$ \forall x \in M \setminus B_{R_0}(o) \, ,  \forall t \ge 0 \, , $ $R_0$ large enough, which is again a global Harnack principle. %Once again, we comment that we deal only with the long-time asymptotics since the short-time one is Euclidean, and that curvature might be assumed to be nonpositive only outside a compact set.

\item Lower bounds on solutions clearly also hold for nonnegative data $u_0$ which are not necessarily compactly supported but belong e.g. to $L^1$.

\end{enumerate}
\end{rem}

%%%%%%%%%%%%%%%%%%%%%%%%%%%%%%%%%%%%%%%%%%%%%%%%%%%%%%%%%%%%%%%%%%%%%%%%
\section{Proofs in the quasi-hyperbolic, subcritical range}\label{sec.proofs}
%We shall prove our main results by constructing appropriate supersolutions, and supersolutions, to \eqref{pme}.   which are proved to satisfy the required inequalities by means of the above mentioned Laplacian comparison theorems.
Our strategy relies on the construction of suitable explicit barriers, namely supersolutions and subsolutions to \eqref{pme}. To this end, we shall make a constant use of Laplacian comparison theorems as recalled in Section \ref{sect: geom-pre}, by choosing a model manifold associated with an appropriate function $\psi$ that reproduces the assumed curvature bounds. In this regard, the next two lemmas will be crucial. The corresponding proofs are quite similar, so we give only the first one.
%The computations are similar in principle, but significantly different in detail, when $\alpha,\beta\in(1,2)$ and when $\alpha,\beta\in(0,1)$, therefore we divide the proof in two sections, starting from the first case.
%
%\color{blue}
%
%Say that we can assume, for simplicity, the bounds (\emph{i.e.}~with $Q=$)
%$$ \mathrm{Ric}_o(x) \ge - (1+\mu)^2 \, r^{2\mu} \, , \quad \mathrm{K}_\omega(x) \le - (1+\mu)^2 \, r^{2\mu} \quad \forall x \in M \setminus B_R(o) \, ,  $$
%in view of the change of variables $ u(\lambda^2 t, \lambda r) $, since this basically amounts to choosing $ \psi(r)=K \, e^{r^{1+\mu}} $ for large $r$. Explained in detail in Remark \ref{rem: curvature-one}.
%\normalcolor

\begin{lem}\label{lem: ode-comparison}
Let $R,Q>0$ and $ \mu \in (-1,1] $ be fixed parameters. Let $ \psi $ be the solution to the following ODE:
\begin{equation}\label{ode-psi-upper-lemma}
\psi''(r) = w(r) \, \psi(r) \quad \forall r>0 \, , \quad \psi'(0)=1 \, , \quad \psi(0)=0  \, ,
\end{equation}
where
\begin{equation}\label{ode-psi-upper-f-lemma}
w(r):=
\begin{cases}
0 & \forall r \in [0,R] \, , \\
\frac{Q \, (2R)^{2\mu}}{R} \, (r-R) & \forall r \in (R,2R] \, , \\
Q \, r^{2\mu} & \forall r > 2R \, .
\end{cases}
\end{equation}
Then $ \psi \in C^2([0,+\infty)) $ is positive on $ (0,+\infty) $ with $ \lim_{r \to +\infty} \psi(r)=+\infty $, $ \psi^\prime $ is positive on $ [0,+\infty) $, and there holds
\begin{equation}\label{ode-psi-estimate-fond-lemma}
\frac{\psi^\prime(r)}{\psi(r)} \sim \sqrt{Q} \, r^\mu \quad \textrm{as } r \to +\infty \, .
\end{equation}
\end{lem}
\begin{proof}
First of all note that, $w$ being a continuous function on $ [0,+\infty) $, by standard ODE theory $ \psi \in C^2([0,+\infty)) $. Moreover,
%since $ f_0 \ge 0 $,
given the initial conditions in \eqref{ode-psi-upper-lemma} it is straightforward to deduce that $ \psi' $ is bigger than $1$, so that $ \psi $ is positive on $(0,+\infty)$ and diverges as $ +\infty $.

We are therefore left with proving the asymptotic estimate \eqref{ode-psi-estimate-fond-lemma}. In view of the above properties, we know that $ \psi $ satisfies
\begin{equation*}\label{ode-psi-upper-lemma-asymp}
\psi''(r) = Q \, r^{2\mu} \, \psi(r) \quad \forall r>2R \, , \quad \psi'(2R)=a>0 \, , \quad \psi(2R)=b>0  \, .
\end{equation*}
Let us introduce the following change of variables:
\begin{equation*}\label{ode-psi-estimate-chvar}
G(r):=\frac{\psi^\prime(r)}{\psi(r) \, r^\mu} \quad \forall r \ge 2R \, ,
\end{equation*}
from which we deduce that $ G $ is a positive solution of
\begin{equation*}\label{ode-psi-estimate-chvar-ode}
G^\prime(r)=-r^\mu \, G^2(r) - \frac{\mu}{r} \, G(r) + Q \, r^\mu \quad \forall r > 2R \, , \quad G(2R)=\frac{a}{2^\mu R^\mu \, b} \, .
\end{equation*}
Hence, proving \eqref{ode-psi-estimate-fond-lemma} is equivalent to proving that
\begin{equation}\label{ode-psi-estimate-fond-chvar}
\lim_{r \to + \infty} G(r) = \sqrt{Q} \, .
\end{equation}
To this end, first of all let us show that $ G $ is bounded from above and from below by positive constants. Indeed, by comparison
%(the r.h.s.~of \eqref{ode-psi-estimate-chvar-ode} is analytic),
it is enough to find $ \overline{k}>\underline{k}>0 $ such that
$$ 0 \ge - \overline{k}^2 r^\mu - \overline{k} \, \frac{\mu}{r} + Q \, r^\mu \quad \forall r > 2R \, , \quad \overline{k} \ge \frac{a}{2^\mu R^\mu \, b} $$
and
$$ 0 \le - \underline{k}^2 r^\mu - \underline{k} \, \frac{\mu}{r} + Q \, r^\mu \quad \forall r > 2R \, , \quad \underline{k} \le \frac{a}{2^\mu R^\mu \, b} \, . $$
It is straightforward to check that the choices
$$  \overline{k} = \max\left\{ \frac{|\mu|}{2 \, (2R)^{\mu+1}} + \sqrt{Q + \frac{\mu^2}{4 \, (2R)^{2\mu+2}}} \, , \, \frac{a}{2^\mu R^\mu \, b} \right\}  $$
and
$$ \underline{k} = \min\left\{ -\frac{|\mu|}{2 \, (2R)^{\mu+1}} + \sqrt{Q + \frac{\mu^2}{4 \, (2R)^{2\mu+2}}} \, , \, \frac{a}{2^\mu R^\mu \, b} \right\} $$
will do. We have then shown that
\begin{equation}\label{ode-psi-estimate-fond-chvar-bounds}
0 < \underline{k} \le G(r) \le \overline{k} < \infty \quad \forall r \ge 2R \, .
\end{equation}
Thanks to \eqref{ode-psi-estimate-fond-chvar-bounds} and to the fact that $ \mu > -1 $, we infer that for every $ \varepsilon \in (0,1) $ there exists $ R_\varepsilon>2R $ such that
\begin{equation*}\label{ode-psi-estimate-fond-chvar-barriers}
G^\prime(r) \le r^\mu \left[ - (1-\varepsilon) \,G^2(r) + Q \right] \ \ \textrm{and} \ \ G^\prime(r) \ge r^\mu \left[ -(1+\varepsilon) \, G^2(r) + Q \right] \ \ \forall r > R_\varepsilon \, .
\end{equation*}
By comparison with the corresponding \emph{solutions} to the above ODEs, we end up with the inequalities
$$ \sqrt{\frac{Q}{1+\varepsilon}} \le \liminf_{r \to +\infty} G(r) \le \limsup_{r \to +\infty}  G(r) \le \sqrt{\frac{Q}{1-\varepsilon}} \, , $$
whence \eqref{ode-psi-estimate-fond-chvar} given the arbitrariness of $ \varepsilon $.
\end{proof}

\begin{lem}\label{lem: ode-comparison-2}
Let $R,Q,D>0$ and $ \mu \in (-1,1] $ be fixed parameters. Let $ \psi $ be the solution to the same ODE as in \eqref{ode-psi-upper-lemma}, where in the case $ \mu \in (-1,0) $ we set
\begin{equation}\label{ode-psi-lower-f-lemma}
w(r):=
\begin{cases}
D & \forall r \in [0,R] \, , \\
\frac{Q \, (2R)^{2\mu}}{R} \, (r-R) + \frac{ D }{R} \, (2R-r) & \forall r \in (R,2R] \, , \\
Q \, r^{2\mu} & \forall r > 2R \, ,
\end{cases}
\end{equation}
whereas in the case $ \mu \in [0,1] $ we set
\begin{equation}\label{ode-psi-lower-f-lemma-bis}
w(r):= \max \left\{ D \, , Q \, r^{2\mu} \right\} .
\end{equation}
Then $ \psi \in C^2([0,+\infty)) $ is positive on $ (0,+\infty) $ with $ \lim_{r \to +\infty} \psi(r)=+\infty $, $ \psi^\prime $ is positive on $ [0,+\infty) $, and \eqref{ode-psi-estimate-fond-lemma} holds.
%\begin{equation}\label{ode-psi-estimate-fond-lemma-lower}
%\frac{\psi^\prime(r)}{\psi(r)} \sim \sqrt{Q_2} \, r^\mu \quad \textrm{as } r \to +\infty \, .
%\end{equation}
\end{lem}
%\begin{proof}
%Above.
%\end{proof}

\begin{rem}\label{rem: f-K}\rm
It is not difficult to check that the function $w$ chosen as in \eqref{ode-psi-upper-f-lemma} satisfies
$$ w(r) \le Q \, r^{2\mu} \, \chi_{(R,+\infty)} \quad \forall r \in [0,\infty) \, . $$
Similarly, the function $ w $ chosen as in \eqref{ode-psi-lower-f-lemma} or \eqref{ode-psi-lower-f-lemma-bis} satisfies
$$ w(r) \ge Q \, r^{2\mu} \, \chi_{(R,+\infty)} + D \, \chi_{[0,R]}  \quad \forall r \in [0,\infty)  $$
provided $ D \ge Q R^{2\mu} $.
\end{rem}

\subsection{Upper barriers: Theorem \ref{pro: ltb-upper}}\label{sect: proofs-1}
%We shall perform explicit calculation in the case $\psi(r)=c\,e^{r^\alpha}$ for $r\ge1$ and a suitable $c>0$ (e.g. $c=e$ so that, see just below, $\psi$ is at least everywhere continuous). We moreover assume $\alpha,\beta\in(1,2)$. For $r>1$ we shall perform for simplicity the calculations in the case in which $\psi(r)=r$ for all $r\in[0,1]$, but notice that inessential modifications allow to treat the case in which one only has $\psi(r)\sim r$ as $r\to0$ (say at second order) and $\psi$ is $C^2$ on $[0,+\infty)$.

As concerns upper bounds, we shall first provide a family of supersolutions to \eqref{pme} on a model manifold identified by the function $\psi$ of Lemma \ref{lem: ode-comparison}, and then exploit the latter together with Laplacian-comparison results of Section \ref{sect: geom-pre} in order to prove Theorem \ref{pro: ltb-upper}.

\begin{lem}[Upper barriers, $ \mu \in (-1,1) $]\label{pro: upper-barrier}
Let $m>1$. Under the same assumptions and with the same notations as in Lemma \ref{lem: ode-comparison}, consider the function $\Ubar(r,t)$ defined, for all $ t \ge 0 $, by
\begin{equation}\label{upper-in-lemma}
(t+t_0)^{\frac1{m-1}} \, \Ubar(r,t)
:=
\begin{cases}
C \left(\gamma\left[\log(t+t_0)\right]^{\frac{1-\mu}{1+\mu}}-r^{1-\mu}\right)^{\frac1{m-1}}_+  & \!\! \forall r \ge R_0  \, , \\
C \left(\gamma\left[\log(t+t_0)\right]^{\frac{1-\mu}{1+\mu}}-\frac{1-\mu}{2 R_0^{1+\mu}}\,r^2-\!\frac{1+\mu}2 \, R_0^{1-\mu} \right)_+^{\frac1{m-1}} & \!\! \forall r\in[0,R_0) \, .
\end{cases}
\end{equation}
%\begin{equation}\label{upper-state}
%\Ubar(r,t) :=
%\begin{cases}
%C \, (t+e)^{-\frac1{m-1}} \left(\gamma\left[\log(t+e)\right]^{\frac{2-\alpha}\alpha}-\frac{2-\alpha}{2}r^2-\frac\alpha2\right)_+^{\frac1{m-1}} & \forall r \in [0,1] \, , \ \, \forall t \ge 0 \, , \\
%C \, (t+e)^{-\frac1{m-1}} \left(\gamma\left[\log(t+e)\right]^{\frac{2-\alpha}{\alpha}}-r^{2-\alpha}\right)^{\frac1{m-1}}_+ & \forall r > 1 \, , \ \, \forall t \ge 0 \, ,
%\end{cases}
%\end{equation}
%where $ C,\gamma>0 $ and $ m>1 $ is a fixed parameter.
Then $ \Ubar $ is a \emph{supersolution} to the equation
\begin{equation}\label{eq: pme-M}
%u_t = \Delta\left( u^m \right) \quad \textrm{in } M \times \mathbb{R}^+
u_t=(u^m)_{rr}+\frac{(n-1)\psi'}{\psi} \, (u^m)_r\quad \textrm{in }  \mathbb{R}^+ \! \times \mathbb{R}^+
\end{equation}
%with initial datum
%\begin{equation}\label{eq: pme-M-init-upper}
%u(r,0) =
%\begin{cases}
%C \, e^{-\frac1{m-1}} \left(\gamma-\frac{2-\alpha}{2}r^2-\frac\alpha2\right)_+^{\frac1{m-1}} & \forall r \in [0,1]  \, , \\
%C \, e^{-\frac1{m-1}} \left(\gamma-r^{2-\alpha}\right)^{\frac1{m-1}}_+ & \forall r > 1 \, ,
%\end{cases}
%\end{equation}
provided $ R_0 $ is larger than a lower bound depending only on $ n,m,\mu,R,Q $ and $ t_0 , C , \gamma $ comply with the following conditions:
\begin{equation}\label{eq: pme-M-init-upper-t_0}
t_0 > e^{R_0^{1+\mu}} \, ,
\end{equation}
\begin{equation}\label{eq: pme-M-init-upper-c1}
C \ge \left[\frac{2}{m \, (1-\mu) \, (n-1)\sqrt{Q/2}}\right]^{\frac1{m-1}} \vee \left[\frac{2 R_0^{1+\mu}}{m \, (1-\mu) \, (E+1)}\right]^{\frac1{m-1}} \, ,
\end{equation}
\begin{equation}\label{eq: pme-M-init-upper-c2}
\begin{aligned}
\gamma
\ge \left[ \frac{m (1+\mu)(1-\mu) C^{m-1}}{m-1} \right]^{\frac{1-\mu}{1+\mu}} \vee \left( 1 + \frac{m \, (1-\mu)^2 \, C^{m-1}}{(m-1) \! \left[ m(1-\mu)(1+E) C^{m-1} - R_0^{1+\mu} \right]} \right) ,
\end{aligned}
\end{equation}
%\begin{equation}\label{eq: pme-M-init-upper-c3-pre}
%R_0 \ge ? \, ,
%\end{equation}
%\begin{equation}\label{eq: pme-M-init-upper-c3}
%t_0 \ge e \, ,
%\end{equation}
where $ E $ is a positive constant depending only on $n,\mu,R,R_0,Q$.
\end{lem}
\begin{proof}
We split the proof in 3 steps, which correspond to a barrier in an outer region, a barrier in an inner region and the final global barrier.

\smallskip
\noindent $ \bullet $ \textsc{Step 1: an upper barrier for large $r$}.
%Let
%\begin{equation}\label{upperouter}
%\Ubar(r,t)=C\,(t+t_0)^{-\frac1{m-1}}\left\{\gamma\left[\log(t+t_0)\right]^{\frac{2-\alpha}{\alpha}}-r^{2-\alpha}\right\}^{\frac1{m-1}}_+\ \ \ \forall r\ge1, t\ge0
%\end{equation}
%where $C,\gamma,t_0$ are positive constants to be chosen later. We claim that $\Ubar$ is a supersolution to \eqref{eq.hpme} for $r\ge1$, $t\ge0$. The verification of this fact is standard but somewhat lengthy.
Let $ r \ge R_0 $. In view of \eqref{upper-in-lemma}, by direct computations we get:
\[\begin{aligned}
(t+t_0)^{\frac m{m-1}} \, \Ubar_t(r,t)=&-\frac{C}{m-1}\left(\gamma\left[\log(t+t_0)\right]^{\frac{1-\mu}{1+\mu}}-r^{1-\mu}\right)^{\frac1{m-1}}_+\\&+
\frac{C\gamma\,(1-\mu)\left[\log(t+t_0)\right]^{-\frac{2\mu}{1+\mu}}}{(1+\mu)(m-1)}\left(\gamma\left[\log(t+t_0)\right]^{\frac{1-\mu}{1+\mu}} - r^{1-\mu}\right)^{\frac{2-m}{m-1}}_+ ,
\end{aligned}
\]
\medskip
\[
\begin{aligned}
(t+t_0)^{\frac m{m-1}} \, \left(\Ubar^m\right)_{\!r}(r,t)=-\frac{m \, C^m (1-\mu)\, r^{-\mu}}{m-1}\left(\gamma\left[\log(t+t_0)\right]^{\frac{1-\mu}{1+\mu}}-r^{1-\mu}\right)^{\frac1{m-1}}_+ ,
\end{aligned}
\]
\medskip
\[
\begin{aligned}
(t+t_0)^{\frac m{m-1}}\left( \Ubar^m \right)_{\!rr}(r,t)=& \frac{m \, C^m \, \mu (1-\mu) \, r^{-(1+\mu)}}{m-1}\left(\gamma\left[\log(t+t_0)\right]^{\frac{1-\mu}{1+\mu}}-r^{1-\mu}\right)^{\frac1{m-1}}_+ \\
& + \frac{m \, C^m \, (1-\mu)^2 \, r^{-2\mu}}{(m-1)^2}\left(\gamma\left[\log(t+t_0)\right]^{\frac{1-\mu}{1+\mu}}-r^{1-\mu}\right)^{\frac{2-m}{m-1}}_+ .
\end{aligned}
\]
With a slight abuse of notations, the quantity $\{y\}_+^{(2-m)/(m-1)}=0$ is meant to be zero if $y\le0$, for all $ m>1 $. Note that $ \Ubar_{t}$, $ \left(\Ubar^m\right)_{\!r} $ and $ \left( \Ubar^m \right)_{\!rr} $ are always integrable as functions of $r$, so that it does make sense to consider $ \Ubar $ as a (weak) supersolution of \eqref{eq: pme-M}. Now let us choose $ R_0 $ so large that
\begin{equation}\label{eq: choice-R0}
\frac{(n-1)\psi'(r)}{\psi(r)\,r^\mu} - \frac{\mu}{r^{1+\mu}} \ge (n-1)\sqrt{Q/2}  \quad \forall r \ge R_0 \, .
\end{equation}
% I made this choice just not to put a fraction
We point out that such a choice is feasible since $ \mu>-1 $ and Lemma \ref{lem: ode-comparison} ensures that \eqref{ode-psi-estimate-fond-lemma} holds. In particular, a sufficient condition for $ \Ubar(r,t) $  to be a supersolution in the region $ [R_0,+\infty) \times \mathbb{R}^+ $ is
\begin{equation}\label{barrier1}
\begin{aligned}
& \left[m \, (1-\mu) \, (n-1)\sqrt{Q/2} \, C^{m-1}-1 \right] \left(\gamma\left[\log(t+t_0)\right]^{\frac{1-\mu}{1+\mu}}-r^{1-\mu}\right)^{\frac1{m-1}}_+ \\ & + \frac{\gamma\,(1-\mu)\left[\log(t+t_0)\right]^{-\frac{2\mu}{1+\mu}}}{1+\mu} \left(\gamma\left[\log(t+t_0)\right]^{\frac{1-\mu}{1+\mu}}-r^{1-\mu}\right)^{\frac{2-m}{m-1}}_+\\
\ge &\,\frac{m \, C^{m-1} \, (1-\mu)^2 \, r^{-2\mu}}{m-1}\left(\gamma\left[\log(t+t_0)\right]^{\frac{1-\mu}{1+\mu}}-r^{1-\mu}\right)^{\frac{2-m}{m-1}}_+ .
\end{aligned}
\end{equation}
%If this is true the $\Ubar$ is a supersolution. In fact this follows from the explicit values of derivatives just computed by performing in addition the estimate
%\[\begin{aligned}
%&C^m\frac{m(2-\alpha)(\alpha-1)}{m-1}r^{-\alpha}\left\{\gamma\left[\log(t+t_0)\right]^{\frac{2-\alpha}{\alpha}}-r^{2-\alpha}\right\}^{\frac1{m-1}}_+\\ \le\,
%&C^m\frac{m(2-\alpha)(\alpha-1)}{m-1}\left\{\gamma\left[\log(t+t_0)\right]^{\frac{2-\alpha}{\alpha}}-r^{2-\alpha}\right\}^{\frac1{m-1}}_+
%\end{aligned}\]
%for all $r\ge1$.
In order to satisfy \eqref{barrier1}, we have to pick  $C$, $\gamma$ and $t_0$ large enough. First of all, for any given $ k \ge 1 $ let us set
% In particular we choose $C$ satisfying, for a given $k\ge1$, \color{blue} this fixes a large mass. It is necessary for this function to be a barrier\normalcolor
\begin{equation}\label{C cond}
C=\left[\frac{2k}{m \, (1-\mu) \, (n-1)\sqrt{Q/2}}\right]^{\frac1{m-1}} ,
\end{equation}
so that \eqref{barrier1} becomes
\begin{equation*}
\begin{aligned}
& (2k-1) \left(\gamma\left[\log(t+t_0)\right]^{\frac{1-\mu}{1+\mu}}-r^{1-\mu}\right)^{\frac1{m-1}}_+ \\
& + \frac{\gamma\,(1-\mu)\left[\log(t+t_0)\right]^{-\frac{2\mu}{1+\mu}}}{1+\mu} \left(\gamma\left[\log(t+t_0)\right]^{\frac{1-\mu}{1+\mu}}-r^{1-\mu}\right)^{\frac{2-m}{m-1}}_+\\
\ge & \frac{2k \, (1-\mu) \, r^{-2\mu}}{(m-1)(n-1)\sqrt{Q/2}}\left(\gamma\left[\log(t+t_0)\right]^{\frac{1-\mu}{1+\mu}}-r^{1-\mu}\right)^{\frac{2-m}{m-1}}_+ .
\end{aligned}\end{equation*}
In the region
\begin{equation}\label{barrier-region}
R_0^{1-\mu} \le r^{1-\mu}\le \gamma\left[\log(t+t_0)\right]^{\frac{1-\mu}{1+\mu}} \, ,
\end{equation}
upon supposing with no loss of generality that $ \gamma \ge 1 $ and $ \log t_0 > R_0^{1+\mu} $, the previous inequality reads
\begin{equation}\label{barrier3}
\begin{aligned}
& \gamma \, (2k-1)\left[\log(t+t_0)\right]^{\frac{1-\mu}{1+\mu}} + \frac{\gamma\,(1-\mu)}{1+\mu} \left[\log(t+t_0)\right]^{-\frac{2\mu}{1+\mu}} \\
\ge &  \frac{2k \, (1-\mu) }{(m-1)(n-1)\sqrt{Q/2}} \, r^{-2\mu} + (2k-1) \, r^{1-\mu} \, .
\end{aligned}
\end{equation}
%\begin{equation}\label{barrier2}
%\begin{aligned}
%&(2k-1)\left(\gamma\left[\log(t+t_0)\right]^{\frac{1-\mu}{1+\mu}}-r^{1-\mu}\right) + \frac{\gamma\,(1-\mu)}{1+\mu} \left[\log(t+t_0)\right]^{-\frac{2\mu}{1+\mu}} \\
%\ge & \frac{2k \, (1-\mu) }{(m-1)(n-1)\sqrt{Q/2}} \, r^{-2\mu} \, .
%\end{aligned}
%\end{equation}
It is straightforward to check that one can choose $ R_0 $ so large (depending only on $n$, $m$, $ Q $ and independent of $ k \ge 1 $, $ \mu \in (-1,1)$) that the r.h.s.~of \eqref{barrier3} is monotone increasing for $ r \ge R_0 $. Hence, in view of \eqref{barrier-region}, it is enough to require that \eqref{barrier3} is satisfied at $ r=\gamma^{1/(1-\mu)}\left[\log(t+t_0)\right]^{1/(1+\mu)} $, which leads to the condition
\begin{equation}\label{barrier4}
\gamma \ge \left[ \frac{2k \, (1+\mu)}{(m-1)(n-1)\sqrt{Q/2}} \right]^{\frac{1-\mu}{1+\mu}}  \, .
\end{equation}
%The function of $r$
%\begin{equation}\label{g}
%g(r):=\frac{2k(2-\alpha)}{(m-1)[(n-2)\alpha+1]}r^{-2(\alpha-1)}+r^{2-\alpha}
%\end{equation}
%has no local maxima, hence it attains its maximum in the region considered either in $r=1$ or in $r=\gamma^{1/(2-\alpha)}\left[\log(t+t_0)\right]^{\frac{1}{\alpha}}$.
%Comparison with the value corresponding to $r=1$ gives, in view of the above bound,
%\[\gamma\ge\frac{\alpha}{2[(k-1)\alpha+1]}\left[\frac{2k(2-\alpha)}{(m-1)[(n-2)\alpha+1]}+1\right].\]
%Comparison with the other value of $r$ gives, starting from \eqref{barrier3}, recalling that $k\ge1$ and performing elementary calculations,
%\[
%\gamma\ge\left(\frac{2k\alpha}{(m-1)[(n-2)\alpha+1]}\right)^{\frac{2-\alpha}\alpha}.
%\]
%Hence we shall choose
%\begin{equation}\label{gamma cond}
%\gamma\ge\frac{\alpha}{2[(k-1)\alpha+1]}\left[\frac{2k(2-\alpha)}{(m-1)[(n-2)\alpha+1]}+1\right]\vee\left(\frac{2k\alpha}{(m-1)[(n-2)\alpha+1]}
%\right)^{\frac{2-\alpha}\alpha}\vee1.
%\end{equation}
We have therefore shown that $\Ubar$ as in \eqref{upper-in-lemma} is a supersolution to \eqref{eq: pme-M} in the region \eqref{barrier-region} provided \eqref{C cond} holds with $ k \ge 1 $, $ \gamma \ge 1 $ satisfies \eqref{barrier4}, $ \log t_0 > R_0^{1+\mu} $ and $R_0$ is large enough. On the other hand, it is apparent that \eqref{barrier1} is fulfilled in the region $r >  \gamma^{1/(1-\mu)}\left[\log(t+t_0)\right]^{1/(1+\mu)}$.
%In conclusion, the function $\Ubar$ in \eqref{upperouter} is a supersolution in the outer region if $t_0=e$, $C$ satisfies \eqref{C cond} for a given $k\ge1$ and, given this $k$, $\gamma$ satisfies \eqref{gamma cond}.

\smallskip
\noindent $ \bullet $ \textsc{Step 2: an upper barrier for small $r$}. We now move to the region $ r \in [0,R_0) $: note that the two branches of $ \Ubar $ match at $r=R_0$ with $ C^1 $ regularity. Still from \eqref{upper-in-lemma}, we get:
\[
\begin{aligned}
& (t+t_0)^{\frac m{m-1}} \, \Ubar_t(r,t) \\
= & - \frac{C}{m-1} \left(\gamma\left[\log(t+t_0)\right]^{\frac{1-\mu}{1+\mu}}-\frac{1-\mu}{2 R_0^{1+\mu}}\,r^2-\!\frac{1+\mu}2 \, R_0^{1-\mu} \right)_+^{\frac1{m-1}} \\
& + \frac{C\gamma\,(1-\mu)\left[\log(t+t_0)\right]^{-\frac{2\mu}{1+\mu}}}{(1+\mu)(m-1)} \left(\gamma\left[\log(t+t_0)\right]^{\frac{1-\mu}{1+\mu}}-\frac{1-\mu}{2 R_0^{1+\mu}}\,r^2-\!\frac{1+\mu}2 \, R_0^{1-\mu} \right)_+^{\frac{2-m}{m-1}} ,
\end{aligned}
\]
\medskip
\[
\begin{aligned}
& (t+t_0)^{\frac m{m-1}} \, \left(\Ubar^m\right)_{\!r}(r,t) \\
= & - \frac{m\,C^m \, (1-\mu) \, r }{(m-1) \, R_0^{1+\mu}} \left(\gamma\left[\log(t+t_0)\right]^{\frac{1-\mu}{1+\mu}}-\frac{1-\mu}{2 R_0^{1+\mu}}\,r^2-\!\frac{1+\mu}2 \, R_0^{1-\mu} \right)_+^{\frac1{m-1}} ,
\end{aligned}
\]
\medskip
\[
\begin{aligned}
& (t+t_0)^{\frac m{m-1}}\left( \Ubar^m \right)_{\!rr}(r,t) \\
= & - \frac{m\,C^m \, (1-\mu) }{(m-1) \, R_0^{1+\mu}} \left(\gamma\left[\log(t+t_0)\right]^{\frac{1-\mu}{1+\mu}}-\frac{1-\mu}{2 R_0^{1+\mu}}\,r^2-\!\frac{1+\mu}2 \, R_0^{1-\mu} \right)_+^{\frac1{m-1}} \\
& + \frac{m\,C^m \, (1-\mu)^2 \, r^2 }{(m-1)^2 \, R_0^{2+2\mu}} \left(\gamma\left[\log(t+t_0)\right]^{\frac{1-\mu}{1+\mu}}-\frac{1-\mu}{2 R_0^{1+\mu}}\,r^2-\!\frac{1+\mu}2 \, R_0^{1-\mu} \right)_+^{\frac{2-m}{m-1}} .
\end{aligned}
\]
The properties of $ \psi $ and $ \psi^\prime $ from Lemma \ref{lem: ode-comparison} ensure that there exists a positive constant $ E=E(n,\mu,R,R_0,Q) $ such that
\begin{equation}\label{eq: consequence-R0}
\frac{(n-1)\psi'(r) \, r}{\psi(r)} \ge E  \quad \forall r \in [0,R_0) \, .
\end{equation}
Hence, to make sure that $ \Ubar(r,t) $ is a supersolution in the region $ [0,R_0) \times \mathbb{R}^+ $ it is enough to require
\begin{equation}\label{inner}
\begin{aligned}
& \left[ \frac{m \, (1-\mu) \, (E+1)}{R_0^{1+\mu}} \, C^{m-1} - 1 \right] \left(\gamma\left[\log(t+t_0)\right]^{\frac{1-\mu}{1+\mu}}-\frac{1-\mu}{2 R_0^{1+\mu}}\,r^2-\!\frac{1+\mu}2 \, R_0^{1-\mu} \right)_+^{\frac1{m-1}} \\
& + \frac{\gamma\,(1-\mu)\left[\log(t+t_0)\right]^{-\frac{2\mu}{1+\mu}}}{1+\mu} \left(\gamma\left[\log(t+t_0)\right]^{\frac{1-\mu}{1+\mu}}-\frac{1-\mu}{2 R_0^{1+\mu}}\,r^2-\!\frac{1+\mu}2 \, R_0^{1-\mu} \right)_+^{\frac{2-m}{m-1}} \\
\ge  & \frac{m\,C^{m-1} \, (1-\mu)^2 \, r^2 }{(m-1) \, R_0^{2+2\mu}} \left(\gamma\left[\log(t+t_0)\right]^{\frac{1-\mu}{1+\mu}}-\frac{1-\mu}{2 R_0^{1+\mu}}\,r^2-\!\frac{1+\mu}2 \, R_0^{1-\mu} \right)_+^{\frac{2-m}{m-1}} .
\end{aligned}
\end{equation}
Upon recalling the choice of $C$ \eqref{C cond}, the assumptions $ \gamma \ge 1 $ and $ \log t_0 > R_0^{1+\mu} $, we deduce that \eqref{inner} reads
\begin{equation}\label{inner-bis}
\begin{aligned}
\left[ \frac{2k\,(E+1)}{(n-1)\sqrt{Q/2} \, R_0^{1+\mu} } - 1 \right] \left(\gamma\left[\log(t+t_0)\right]^{\frac{1-\mu}{1+\mu}}-\frac{1-\mu}{2 R_0^{1+\mu}}\,r^2-\!\frac{1+\mu}2 \, R_0^{1-\mu} \right) & \\
+ \frac{\gamma\,(1-\mu)\left[\log(t+t_0)\right]^{-\frac{2\mu}{1+\mu}}}{1+\mu} - \frac{2k \, (1-\mu)}{(m-1)\,(n-1)\sqrt{Q/2} \, R_0^{2+2\mu}} \, r^2 & \ge 0 \, .
\end{aligned}
\end{equation}
By neglecting the penultimate term we find that sufficient conditions for \eqref{inner-bis} to hold are e.g.
\[
k \ge \frac{(n-1)\sqrt{Q/2} \, R_0^{1+\mu}}{E+1} \quad \textrm{and} \quad \gamma\ge 1+\frac{2k\,(1-\mu)}{(m-1)\left[2k \,(E+1) - (n-1)\sqrt{Q/2} \, R_0^{1+\mu} \right]} \, ,
\]
which have to be understood in addition to $ k \ge 1 $ and \eqref{barrier4}.

\smallskip
\noindent $ \bullet $ \textsc{Step 3: the global upper barrier}. In conclusion, through steps 1 and 2 we have shown that the function $\Ubar$ defined by \eqref{upper-in-lemma} is a supersolution to \eqref{eq: pme-M} provided $ R_0 $ is large enough (with a lower bound depending only on $ n,m,\mu,R,Q $), $ \log t_0 > R_0^{1+\mu} $ and $ C $ and $ \gamma $ satisfy the following conditions:
\begin{equation}\label{C cond-bis-1}
C=\left[\frac{2k}{m \, (1-\mu) \, (n-1)\sqrt{Q/2}}\right]^{\frac1{m-1}} \quad \textrm{with} \quad k \ge 1 \vee \frac{(n-1)\sqrt{Q/2} \, R_0^{1+\mu}}{E+1} \, ,
\end{equation}
\begin{equation}\label{C cond-bis-2}
\gamma
\ge \left[ \frac{2k \, (1+\mu)}{(m-1)(n-1)\sqrt{Q/2}} \right]^{\frac{1-\mu}{1+\mu}} \vee \left( \! 1+ \! \frac{2k\,(1-\mu)}{(m-1) \! \left[2k (E+1) - (n-1)\sqrt{Q/2} \, R_0^{1+\mu} \right]} \right) ,
\end{equation}
which yield \eqref{eq: pme-M-init-upper-t_0}--\eqref{eq: pme-M-init-upper-c2} (the parameter $ k $ has been introduced for later purpose -- see the proof of Lemma \ref{pro: upper-barrier-2}).
\end{proof}
%\begin{rem} \rm The equation $u_t=(u^m)_{rr}+(n-1)\frac{\psi'}{\psi}(u^m)_r$ would coincide, should $\psi$ be smooth, with the porous media equation
%$u_t = \Delta\left( u^m \right)$ on the $n$-dimensional ($n\ge2$) model manifold $M$ with metric $ \mathrm{d}s^2 =   \mathrm{d}r^2 + \psi(r)^2 \mathrm{d}\Theta^2 $, when considered for radial functions (i.e. depending on the geodesic distance $r$ from a given pole only), provided $\psi(r)$ behaves like $r$ in $C^2$ sense near the origin. We are not stating it in that form since to avoid unnecessary technicalities it turns out to be simpler to perform our calculations with a function $\psi$ which is only Lipschitz. This comment applies to all the forthcoming Propositions as well.
%\end{rem}

\begin{proof}[Proof of Theorem \ref{pro: ltb-upper}]
First of all, by exploiting Lemma \ref{lem: ode-comparison} with $ Q=Q_1 $ and $ \mu=\mu_1 $ ($R$ is meant to be the same as in \eqref{hp-sect}), Remark \ref{rem: f-K} and recalling that $ \mathrm{K}_\omega \le 0 $ everywhere, we infer that the corresponding function $ \psi $ satisfies
$$
\mathrm{K}_\omega(x) \le -\frac{\psi^{\prime\prime}(r)}{\psi(r)} \quad \forall x \in M \setminus \{ o \} \, .
$$
Hence, we can apply the Laplacian-comparison results discussed in Section \ref{sect: geom-pre} along with Lemma \ref{pro: upper-barrier} to deduce that the function $ x \mapsto \Ubar(r(x),t)$, with $ \Ubar $ as in \eqref{upper-in-lemma}, is a supersolution to \eqref{pme}, since \eqref{eq: pme-M} is precisely the differential equation appearing in \eqref{pme}, for radial functions, on the Riemannian model identified by $ \psi $.

Let us set $  R_1 =R_0 $ and $ t_1 = t_0 $ satisfying \eqref{eq: pme-M-init-upper-t_0}. In order to complete the proof, we only have to show that the parameters $ C=C_1 $ and $ \gamma=\gamma_1 $, subject to \eqref{eq: pme-M-init-upper-c1} and \eqref{eq: pme-M-init-upper-c2}, can be chosen in such a way that $ \Ubar(r(x),0) \ge u_0(x) $ for all $ x \in M $. With no loss of generality, we can suppose that $ u_0 $ is supported in $ B_{\mathcal{R}}(o) $ for some $ \mathcal{R} \ge R_1 $ and has maximum $ \mathcal{M}>0 $. Then, a sufficient condition for $ \Ubar(r(x),0) $ to lie above $u_0(x)$ is
\begin{equation*}\label{cond}
C_1 \ge t_1^{\frac{1}{m-1}} \, \mathcal{M} \, , \quad \gamma_1 \ge  \frac{1+\mathcal{R}^{1-\mu_1}}{\left( \log t_1 \right)^{\frac{1-\mu_1}{1+\mu_1}}} \, ,
\end{equation*}
which must be added to \eqref{eq: pme-M-init-upper-c1} and \eqref{eq: pme-M-init-upper-c2} (with $ C=C_1 $ and $ \gamma=\gamma_1 $). Finally, as for \eqref{bound-smooth-infty}, we just remark that $ \Ubar(\cdot,t) $ is supported precisely in $ B_{\mathsf{R}_1(t)}(o) $ with $ \mathsf{R}_1(t) := \gamma_1^{{1}/{(1-\mu_1)}} \, \left[ \log(t+t_1) \right]^{1/(1+\mu_1)} $ and attains its maximum at $ r=0 $.
%The case $ \nu \in (-1,0) $. Suppose we want to put $ \Ubar(\cdot,0) $ above an initial datum supported in $ B_R $ (let $ R>1 $ with no loss of generality) with maximum $ \mathcal{M} $. Exactly as in the case $\alpha\in(1,2)$ it is enough to choose $ C $ and $ \gamma $ such that
%\eqref{cond} holds. Such conditions should then be added to \eqref{cbisbis}, \eqref{gammabis}.
\end{proof}

\subsection{Lower barriers: Theorem \ref{pro: ltb-lower}}

As for lower bounds, we proceed in a similar way as above, that is we first obtain a family of subsolutions to \eqref{pme} on a Riemannian model related to the curvature bounds \eqref{hp-ricc}, and then prove Theorem \ref{pro: ltb-lower} by means of such result.

\begin{lem}[Lower barriers, $ \mu\in(-1,1) $]\label{pro: lower-barrier}
Let $m>1$. Under the same assumptions and with the same notations as in Lemma \ref{lem: ode-comparison-2}, consider the function $\ubar(r,t)$ defined, for all $ t \ge 0 $, by
\begin{equation}\label{lower-in-lemma}
(t+t_0)^{\frac1{m-1}} \, \ubar(r,t) := \!
\begin{cases}
C \left(\gamma\left[\log(t+t_0)\right]^{\frac{1-\mu}{1+\mu}}-r^{1-\mu}\right)^{\frac1{m-1}}_+  & \! \forall r \ge R_0  \, , \\
C \left(\gamma\left[\log(t+t_0)\right]^{\frac{1-\mu}{1+\mu}}-\frac{1-\mu}{2 R_0^{1+\mu}}\,r^2-\frac{1+\mu}2 R_0^{1-\mu} \right)_+^{\frac1{m-1}} & \! \forall r\in[0,R_0) \, .
\end{cases}
\end{equation}
Then $ \ubar $ is a \emph{subsolution} to \eqref{eq: pme-M} provided $ R_0 $ is larger than a lower bound depending only on $ n,m,\mu,R,Q,D $ and the positive parameters $ C , \gamma , t_0 $ comply with the following conditions:
\begin{equation}\label{eq: pme-M-init-lower-c1}
C^{m-1} \le \frac{1}{m(1-\mu)} \left[ \frac{1}{2(n-1)\sqrt{2Q}} \wedge \frac{R_0^{1+\mu}}{4(F+1)} \wedge \frac{(m-1)R_0^{1+\mu}}{[2+(m-1)(F+1)]} \right] ,
\end{equation}
\begin{equation}\label{eq: pme-M-init-lower-c2}
\gamma  \le \left[ \frac{m (1+\mu)(1-\mu) C^{m-1}}{m-1} \right]^{\frac{1-\mu}{1+\mu}}  ,
\end{equation}
\begin{equation}\label{eq: pme-M-init-lower-t_0}
\frac{\gamma^{\frac{1+\mu}{1-\mu}}}{ R_0^{1+\mu}} \, \log t_0 \ge \left[ \frac{4(m-1)R_0^{1+\mu}-4m(1-\mu)[(m-1)(F+1)+1-\mu]\,C^{m-1}}{(3+\mu)(m-1)R_0^{1+\mu}-4m(m-1)(1-\mu)(F+1)\,C^{m-1}} \right]^{\frac{1+\mu}{1-\mu}} \, ,
\end{equation}
where $ F $ is a positive constant depending only on $n,\mu,R,R_0,Q,D$.
\end{lem}
\begin{proof}
As in Lemma \ref{pro: upper-barrier}, we split the proof in 3 steps.

\smallskip
\noindent $ \bullet $ \textsc{Step 1: a lower barrier for large $r$}. Let $ r \ge R_0 $ and take $R_0$ so large that
\begin{equation}\label{eq: choice-R0-lower}
\frac{(n-1)\psi'(r)}{\psi(r)\,r^\mu} - \frac{\mu}{r^{1+\mu}} \le (n-1)\sqrt{2Q}  \quad \forall r \ge R_0 \, ,
\end{equation}
which is feasible thanks to Lemma \ref{lem: ode-comparison-2} and to the fact that $ \mu > -1 $. In view of \eqref{lower-in-lemma} (note that it is the same expression as \eqref{upper-in-lemma}) and the explicit derivatives computed in the proof of Lemma \ref{pro: upper-barrier}, we infer that a sufficient condition for $ \ubar(r,t) $  to be a subsolution in the region $ [R_0,+\infty) \times \mathbb{R}^+ $ is
\begin{equation}\label{barrier1-lower}
\begin{aligned}
& \left[1- m \, (1-\mu) \, (n-1)\sqrt{2Q} \, C^{m-1} \right] \left(\gamma\left[\log(t+t_0)\right]^{\frac{1-\mu}{1+\mu}}-r^{1-\mu}\right)^{\frac1{m-1}}_+ \\ & - \frac{\gamma\,(1-\mu)\left[\log(t+t_0)\right]^{-\frac{2\mu}{1+\mu}}}{1+\mu} \left(\gamma\left[\log(t+t_0)\right]^{\frac{1-\mu}{1+\mu}}-r^{1-\mu}\right)^{\frac{2-m}{m-1}}_+\\
\ge & - \frac{m \, C^{m-1} \, (1-\mu)^2 \, r^{-2\mu}}{m-1}\left(\gamma\left[\log(t+t_0)\right]^{\frac{1-\mu}{1+\mu}}-r^{1-\mu}\right)^{\frac{2-m}{m-1}}_+ .
\end{aligned}
\end{equation}
Now let us set, for all $ h \in (0,1] $,
\begin{equation}\label{C cond-lower}
C=\left[\frac{h}{2 m \, (1-\mu) \, (n-1)\sqrt{2Q}}\right]^{\frac1{m-1}} ,
\end{equation}
so that \eqref{barrier1-lower} becomes
\begin{equation}\label{barrier1-lower-bis}
\begin{aligned}
& (1 - {h}/{2} ) \left(\gamma\left[\log(t+t_0)\right]^{\frac{1-\mu}{1+\mu}}-r^{1-\mu}\right)^{\frac1{m-1}}_+ \\ & - \frac{\gamma\,(1-\mu)\left[\log(t+t_0)\right]^{-\frac{2\mu}{1+\mu}}}{1+\mu} \left(\gamma\left[\log(t+t_0)\right]^{\frac{1-\mu}{1+\mu}}-r^{1-\mu}\right)^{\frac{2-m}{m-1}}_+\\
\ge & - \frac{h \, (1-\mu) \, r^{-2\mu}}{2(m-1)(n-1)\sqrt{2Q}}\left(\gamma\left[\log(t+t_0)\right]^{\frac{1-\mu}{1+\mu}}-r^{1-\mu}\right)^{\frac{2-m}{m-1}}_+ .
\end{aligned}
\end{equation}
%A sufficient condition for this to hold is, using the explicit value of derivatives given in Step 1 and neglecting one term:
%\[\begin{aligned}
%&C^{m-1}\frac{m(2-\beta)^2}{(m-1)^2}r^{2-2\beta}\left\{\gamma\left[\log(t+t_0)\right]^{\frac{2-\beta}{\beta}}-r^{2-\beta}\right\}^{\frac{2-m}{m-1}}_+\\
%-&C^{m-1}
%\frac{\beta (2-\beta)(n-1)m}{m-1}\left\{\gamma\left[\log(t+t_0)\right]^{\frac{2-\beta}{\beta}}-r^{2-\beta}\right\}^{\frac{1}{m-1}}_+\\
%\ge &-\frac1{m-1}\left\{\gamma\left[\log(t+t_0)\right]^{\frac{2-\beta}{\beta}}-r^{2-\beta}\right\}^{\frac{1}{m-1}}_+\\
%+&\gamma\frac{2-\beta}{\beta(m-1)}\left[\log(t+t_0)\right]^{\frac{2-2\beta}{\beta}}\left\{\gamma\left[\log(t+t_0)\right]^{\frac{2-\beta}{\beta}}-
%r^{2-\beta}\right\}^{\frac{2-m}{m-1}}_+.
%\end{aligned}\]
%Setting now, given $h\in(0,1)$,
%\begin{equation}\label{C cond sub}
%C=\left(\frac{h}{\beta(2-\beta)m(n-1)}\right)^{\frac1{m-1}}
%\end{equation}
To make sure that the quantity between brackets is not identically zero (i.e.~that our subsolution is not trivial), we assume that
\begin{equation}\label{barrier1-lower-t0}
\log t_0 > \gamma^{-\frac{1+\mu}{1-\mu}} \, R_0^{1+\mu} \, .
\end{equation}
Hence, we can deduce that \eqref{barrier1-lower-bis} is equivalent to
\begin{equation}\label{barrier1-lower-ter}
\begin{aligned}
& \gamma \, (1 - {h}/{2} ) \left[\log(t+t_0)\right]^{\frac{1-\mu}{1+\mu}} - \frac{\gamma\,(1-\mu)}{1+\mu} \left[\log(t+t_0)\right]^{-\frac{2\mu}{1+\mu}} \\
\ge & - \frac{h \, (1-\mu) \, r^{-2\mu}}{2(m-1)(n-1)\sqrt{2Q}} + (1-h/2)\,r^{1-\mu}  
\end{aligned}
\end{equation}
in the region \eqref{barrier-region}. As in the proof of Lemma \ref{pro: upper-barrier}, we can choose $ R_0 $ so large (depending only on $n$, $m$, $ Q $, $ \mu \in (-1,0) $ and independent of $ h \in (0,1] $, $ \mu \in [0,1)$) that the r.h.s.~of \eqref{barrier1-lower-ter} is monotone increasing for $ r \ge R_0 $. It is therefore enough to require that \eqref{barrier1-lower-ter} is satisfied at $ r=\gamma^{1/(1-\mu)}\left[\log(t+t_0)\right]^{1/(1+\mu)} $, which yields
\begin{equation}\label{barrier4-lower}
\gamma \le \left[ \frac{h \, (1+\mu)}{2(m-1)(n-1)\sqrt{2Q}} \right]^{\frac{1-\mu}{1+\mu}} \, .
\end{equation}
% il valore ``corretto'' di \gamma sembrerebbe essere questo con h=1; bisogna vedere, soprattutto per l'asintotica, se con qualche argomento di bootstrap, o piccolo miglioramento delle barriere, non si riesca a mostrarlo
We have then shown that $\ubar$ as in \eqref{lower-in-lemma} is a subsolution to \eqref{eq: pme-M} provided \eqref{C cond-lower} holds with $ h \in (0,1] $, $ \gamma $ satisfies \eqref{barrier4-lower}, $ t_0 $ satisfies \eqref{barrier1-lower-t0} and $R_0$ is large enough.

\smallskip
\noindent $ \bullet $ \textsc{Step 2: a lower barrier for small $r$}.
%We look for a subsolution of the form very similar to \eqref{super inner} given in Step 2, in the inner region $r\in(0,1)$, $t\ge0$, of course for different values of the parameters involved. In particular we require
%\begin{equation}\label{sub inner}
%\ubar(r,t)=C\,(t+t_0)^{-\frac{1}{m-1}}\left\{\gamma\left[\log(t+t_0)\right]^{\frac{2-\beta}\beta}-\frac{2-\beta}{2}r^2-\frac\beta2\right\}_+^{\frac1{m-1}}\ \ \ \forall r\in(0,1), t\ge0,
%\end{equation}
%where $t_0\ge e$. We assume that $C$ is given as in \eqref{C cond sub}, for a given $h\ge(0,1)$, and that $\gamma$ is as in Step 4, in particular that it satisfies \eqref{cond gamma sub}. Notice that $\ubar$ then matches its previous definition for the outer region in $r=1$ in the $C^1$ sense.
In the region $ r \in [0,R_0) $, by proceeding as in step 2 of the proof of Lemma \ref{pro: upper-barrier} and defining the positive constant $ F=F(n,\mu,R,R_0,Q,D) $ by
\begin{equation}\label{eq: consequence-R0-lower}
\frac{(n-1)\psi'(r) \, r}{\psi(r)} \le F  \quad \forall r \in [0,R_0) \, ,
\end{equation}
we find that, under assumptions \eqref{C cond-lower} and \eqref{barrier1-lower-t0}, in order to ensure that $\ubar$ is a subsolution in the region $ [0,R_0) \times \mathbb{R}^+ $ it is enough to ask
\begin{equation}\label{barrier3-lower}
\begin{aligned}
\left[ 1 - \frac{h\,(F+1)}{2(n-1)\sqrt{2Q} \, R_0^{1+\mu} } \right] \left(\gamma\left[\log(t+t_0)\right]^{\frac{1-\mu}{1+\mu}}-\frac{1-\mu}{2 R_0^{1+\mu}}\,r^2-\!\frac{1+\mu}2 \, R_0^{1-\mu} \right) & \\
- \frac{\gamma\,(1-\mu)\left[\log(t+t_0)\right]^{-\frac{2\mu}{1+\mu}}}{1+\mu} + \frac{h \, (1-\mu)}{2\,(m-1)\,(n-1)\sqrt{2Q} \, R_0^{2+2\mu}} \, r^2 & \ge 0 \, .
\end{aligned}
\end{equation}
%In order that $\ubar$ is a subsolution in the inner region we require, proceeding as in Step 2, that
%\begin{equation}\label{cond sub}
%\begin{aligned}
%&\left[1-\frac{nh}{\beta(n-1)}\right]\left\{\gamma\left[\log(t+t_0)\right]^{\frac{2-\beta}{\beta}}-\frac{2-\beta}{2}r^2-\frac\beta2\right\}_+\\
%&-\gamma
%\frac{2-\beta}{\beta}\left[\log(t+t_0)\right]^{\frac{2-2\beta}\beta}+\frac{h(2-\beta)}{\beta(n-1)(m-1)}r^2\ge0
%\end{aligned}\end{equation}
%for all $t\ge0$, $r\in(0,1)$.
Upon requiring
\begin{equation*}
\log t_0 \ge \frac{4}{1+\mu}
\end{equation*}
and
\begin{equation}\label{barrier3-lower-bis-pre}
h \le \frac{(n-1)\sqrt{2Q} \, R_0^{1+\mu}}{2\,(F+1)} \wedge \frac{2(m-1)(n-1)\sqrt{2Q} \, R_0^{1+\mu}}{2+(m-1)\,(F+1)} \, ,
\end{equation}
it is easy to see that \eqref{barrier3-lower} is satisfied provided e.g.
\begin{equation*}\label{barrier3-lower-bis}
\begin{aligned}
& \gamma \left[ 1 - \frac{h\,(F+1)}{2(n-1)\sqrt{2Q} \, R_0^{1+\mu} } - \frac{1-\mu}{4} \right] \left(\log t_0 \right)^{\frac{1-\mu}{1+\mu}} \\
\ge & R_0^{1-\mu} \left[1 - \frac{h\,(F+1)}{2(n-1)\sqrt{2Q} \, R_0^{1+\mu} } - \frac{h \, (1-\mu)}{2\,(m-1)\,(n-1)\sqrt{2Q} \, R_0^{1+\mu}} \right]
\end{aligned}
\end{equation*}
which can be rewritten as
\begin{equation}\label{barrier3-lower-ter}
\log t_0 \ge \frac{R_0^{1+\mu}}{\gamma^{\frac{1+\mu}{1-\mu}}}  \left[ \frac{4(m-1)(n-1)\sqrt{2Q}\,R_0^{1+\mu}-2h\left[(m-1)(F+1)+1-\mu \right]}{(3+\mu)(m-1)(n-1)\sqrt{2Q}\,R_0^{1+\mu}-2h(m-1)(F+1)} \right]^{\frac{1+\mu}{1-\mu}} .
\end{equation}
It is easy to check that the term between brackets in \eqref{barrier3-lower-ter} is larger than $1$ under \eqref{barrier3-lower-bis-pre}, so that \eqref{barrier1-lower-t0} is implied by \eqref{barrier3-lower-ter}.
%\[
%\log t_0\ge\left[\frac1\gamma\left(1-\frac{h(2-\beta)}{(m-1)[\beta(n-1)-nh]}\right)+\frac{(n-1)(2-\beta)}{\beta(n-1)-nh}\right]^{\frac{\beta}{2-\beta}} \, .
%\]
%\[
%\log t_0\ge \frac{(2-\alpha)(n-1)}{\alpha[\alpha(n-1)-nh]}.
%\]
%It is then easy to see that \eqref{cond sub} is satisfied e.g. provided the stronger condition
%\[
%\log t_0\ge \left(\frac{[\alpha+\gamma(2-\alpha)](n-1)}{\gamma[\alpha(n-1)-nh]}\right)^{\frac\alpha{2-\alpha}}
%\]
%holds.
%\color{red} Notice that this will entail that our lower barrier will work for data with support big enough (depending only on $m,\beta,n$, since it can be taken independent of $h$ for $h$ smaller than a constant depending only on $m,\beta,n$).\normalcolor

\smallskip
\noindent $ \bullet $ \textsc{Step 3: the global lower barrier}. Summing up, through steps 1 and 2 we have shown that the function $\ubar$ defined by \eqref{lower-in-lemma} is a subsolution to \eqref{eq: pme-M} provided $ R_0 $ is large enough (with a lower bound depending only on $ n,m,\mu,R,Q,D $) and $ C $, $ \gamma $, $t_0$ satisfy the following conditions:
\begin{equation}\label{C cond-bis-1-lower}
C=\left[\frac{h}{2 m \, (1-\mu) \, (n-1)\sqrt{2Q}}\right]^{\frac1{m-1}}
\end{equation}
with
\begin{equation*}\label{C cond-bis-1-b-lower}
0< h \le 1 \wedge \frac{(n-1)\sqrt{2Q} \, R_0^{1+\mu}}{2\,(F+1)} \wedge \frac{2(m-1)(n-1)\sqrt{2Q} \, R_0^{1+\mu}}{2+(m-1)\,(F+1)} \, ,
\end{equation*}
\begin{equation}\label{C cond-bis-2-lower}
0 < \gamma \le \left[ \frac{h \, (1+\mu)}{2(m-1)(n-1)\sqrt{2Q}} \right]^{\frac{1-\mu}{1+\mu}}
\end{equation}
and
\begin{equation}\label{C cond-bis-3-lower}
\frac{\gamma^{\frac{1+\mu}{1-\mu}}}{ R_0^{1+\mu}} \, \log t_0 \ge \left[ \frac{4(m-1)(n-1)\sqrt{2Q}\,R_0^{1+\mu}-2h\left[(m-1)(F+1)+1-\mu \right]}{(3+\mu)(m-1)(n-1)\sqrt{2Q}\,R_0^{1+\mu}-2h(m-1)(F+1)} \right]^{\frac{1+\mu}{1-\mu}} ,
\end{equation}
whence \eqref{eq: pme-M-init-lower-c1}--\eqref{eq: pme-M-init-lower-t_0} (again, the parameter $ h $ has been introduced for later purpose -- see the proof of Lemma \ref{pro: lower-barrier-2}).
\end{proof}

\begin{proof}[Proof of Theorem \ref{pro: ltb-lower}]
We shall use Lemma \ref{lem: ode-comparison-2} with the choices $ Q=Q_2 $, $\mu=\mu_2$, $R$ the same number appearing in \eqref{hp-ricc} and
\begin{equation}\label{eq: choice-B-curv}
D = \max\left\{ \frac{1}{n-1} \sup_{x \in {B}_{R}(o) } -\mathrm{Ric}_o(x) \, , \, Q_2 \, R^{2\mu_2}  \right\} ,
\end{equation}
so that the corresponding function $ \psi $ complies with (upon recalling Remark \ref{rem: f-K})
\[
\mathrm{Ric}_o(x) \ge - (n-1)\frac{\psi^{\prime\prime}(r)}{\psi(r)} \quad \forall x \in M \setminus \{o\} \, .
\]
By combining Laplacian-comparison results of Section \ref{sect: geom-pre} with Lemma \ref{pro: lower-barrier}, it follows that the function $ x \mapsto \ubar(r(x),t)$, with $ \ubar $ as in \eqref{lower-in-lemma}, is a subsolution to \eqref{pme}.

Let us set $  R_2 =R_0 $. In order to complete the proof, we have to show that the parameters $ C=C_2 $, $ \gamma=\gamma_2 $ and $ t_0=t_2 $, subject to \eqref{eq: pme-M-init-lower-c1}--\eqref{eq: pme-M-init-lower-t_0}, can be chosen in such a way that $ \ubar(r(x),0) \le u_0(x) $ for all $ x \in M $. To this end, first of all we point out that, in view of \eqref{eq: pme-M-init-lower-c1}, condition \eqref{eq: pme-M-init-lower-t_0} is implied by
\begin{equation*}\label{eq: pme-M-init-lower-t_0-bis}
\log t_2 \ge \gamma_2^{-\frac{1+\mu_2}{1-\mu_2}} \, R_2^{1+\mu_2} \, K^{\frac{1+\mu_2}{1-\mu_2}}
\end{equation*}
for some constant $ K> 1 $ that depends only on $ n $, $ m $, $ \mu $, $ R_2 $, $ Q_2 $, $F$ (which in particular can be taken to be independent of $ C_2 $). Then, with no loss of generality, we shall suppose that
\begin{equation}\label{eq: pme-M-init-lower-t_0-ter}
\inf_{x \in B_{\mathcal{R}}(o)} u_0(x) =: \mathcal{L}>0 \, , \quad \mathcal{R}:=K^{\frac{1}{1-\mu_2}} \, R_2 \, .
\end{equation}
Indeed, since \eqref{pme} written in local coordinates is a standard diffusion equation of porous medium type, upon waiting some finite time condition \eqref{eq: pme-M-init-lower-t_0-ter} is certainly satisfied (provided the initial datum is non trivial). Hence, under such assumptions, it is straightforward to verify that the condition $ \ubar(r(x),0) \le u_0(x) $ is finally met provided e.g.
$$
C_2 \le \frac{\mathcal{L}}{ \left( K - \frac{1+\mu_2}{2} \right)^{\frac{1}{m-1}} R_2^{\frac{1-\mu_2}{m-1}}  } \, ,
$$
\eqref{eq: pme-M-init-lower-c1}--\eqref{eq: pme-M-init-lower-c2} hold and one chooses $ t_2 $ as
$$  t_2 = e^{ \gamma_2^{-\frac{1+\mu_2}{1-\mu_2}} \, R_2^{1+\mu_2} \, K^{\frac{1+\mu_2}{1-\mu_2}}} \, . $$
As for the bounds \eqref{bound-smooth-infty-below}, it is enough to note that $ \ubar(\cdot,t) $ is supported precisely in $ B_{\mathsf{R}_1(t)}(o) $ with $ \mathsf{R}_1(t) := \gamma_2^{{1}/{(1-\mu_2)}} \, \left[ \log(t+t_2) \right]^{1/(1+\mu_2)} $ and its maximum is attained at $ r=0 $.
\end{proof}

%%%%%%%%%%%%%%%%%%%%%%%%%%%%%%%%%%%%%%%%%%%%%%%%%%%%%%%%%%%%%%%%%%%%%%%%%%%%%%%%%%%%%%%%%%
% Theorems \ref{pro: ltb-cri-upper} and \ref{pro: ltb-cri-lower}
\section{Proofs in the quasi-hyperbolic, upper-critical ca\-se: Theorems \ref{pro: ltb-cri-upper}, \ref{pro: ltb-cri-lower}} \label{sect: proofs-2}

The strategy we adopt here is very similar to the one we developed in Section \ref{sec.proofs}, namely we first provide supersolutions and subsolutions to \eqref{ode-psi-upper-lemma} with $ \mu=1 $ and then exploit the latter in order to prove Theorems \ref{pro: ltb-cri-upper} and \ref{pro: ltb-cri-lower}. In fact the construction of such barriers is a consequence of the results given by Lemmas \ref{pro: upper-barrier} and \ref{pro: lower-barrier}: it suffices to take limits carefully as $ \mu \to 1^- $.

\begin{lem}[Upper barriers, $\mu=1$]\label{pro: upper-barrier-2}
Let $m>1$. Under the same assumptions and with the same notations as in Lemma \ref{lem: ode-comparison} in the case $ \mu=1 $, consider the function $\Ubar(r,t)$ defined, for all $ t \ge 0 $, by
\begin{equation}\label{upper-in-lemma-critic}
%\begin{aligned}
%& \\
%&
(t+t_0)^{\frac1{m-1}} \, \Ubar(r,t) :=
\begin{cases}
\kappa \left[ \eta + \frac12 \log\log(t+t_0) - \log r  \right]_+^{\frac1{m-1}} & \forall r \ge R_0 \, , \\
\kappa \left[ \eta + \frac12 \log\log(t+t_0) - \frac{r^2}{2R_0^2} - \log R_0 + \frac12 \right]_+^{\frac1{m-1}} & \forall r \in [0,R_0) \, .
\end{cases}
% \end{aligned}
\end{equation}
%\begin{equation}\label{upper-state}
%\Ubar(r,t) :=
%\begin{cases}
%C \, (t+e)^{-\frac1{m-1}} \left(\gamma\left[\log(t+e)\right]^{\frac{2-\alpha}\alpha}-\frac{2-\alpha}{2}r^2-\frac\alpha2\right)_+^{\frac1{m-1}} & \forall r \in [0,1] \, , \ \, \forall t \ge 0 \, , \\
%C \, (t+e)^{-\frac1{m-1}} \left(\gamma\left[\log(t+e)\right]^{\frac{2-\alpha}{\alpha}}-r^{2-\alpha}\right)^{\frac1{m-1}}_+ & \forall r > 1 \, , \ \, \forall t \ge 0 \, ,
%\end{cases}
%\end{equation}
%where $ C,\gamma>0 $ and $ m>1 $ is a fixed parameter.
Then $ \Ubar $ is a \emph{supersolution} to \eqref{eq: pme-M} provided $ R_0 $ is larger than a lower bound depending only on $ n,m,R,Q $ and $ t_0 , \kappa , \eta $ comply with the following conditions:
\begin{equation}\label{eq: pme-M-init-upper-t_0-cri}
t_0 > e^{R_0^2} \, ,
\end{equation}
\begin{equation}\label{eq: pme-M-init-upper-k1}
\kappa \ge \left[\frac{2}{m \, (n-1)\sqrt{Q/2}}\right]^{\frac1{m-1}} \vee \left[\frac{2R_0^2}{m \, (E+1) }\right]^{\frac1{m-1}} \, ,
\end{equation}
\begin{equation}\label{eq: pme-M-init-upper-k2}
%\begin{aligned}
\eta \ge \frac{1}{2} \left\{ \log \left( \frac{2m \, \kappa^{m-1}}{m-1} \right) \vee \frac{2m\,\kappa^{m-1}}{(m-1) \left[m\,(E+1)\,\kappa^{m-1} - R_0^{2} \right]} \right\} ,
%\end{aligned}
\end{equation}
where $ E $ is a positive constant depending only on $n,R,R_0,Q$.
\end{lem}
\begin{proof}
Upon recalling \eqref{C cond-bis-1}--\eqref{C cond-bis-2}, let us choose $C$ and $ \gamma $ by
\begin{equation}\label{eq: cond-crit-kappa-eta}
C = \left[\frac{2k}{m \, (1-\mu) \, (n-1)\sqrt{Q/2}}\right]^{\frac1{m-1}} = \frac{\kappa}{(1-\mu)^{\frac{1}{m-1}}} \quad \textrm{and} \quad \gamma = 1 + \eta \, (1-\mu)
\end{equation}
for any $ \kappa $ and $ \eta $ satisfying
\begin{equation}\label{eq: cond-crit-kappa-eta-a}
\kappa > \left[\frac{2}{m \, (n-1)\sqrt{Q/2}}\right]^{\frac1{m-1}} \vee \left[\frac{2R_0^2}{m \, (E+1) }\right]^{\frac1{m-1}}
\end{equation}
and
\begin{equation}\label{eq: cond-crit-kappa-eta-b}
%\begin{aligned}
\eta > \frac{1}{2} \left\{ \log \left( \frac{2m \, \kappa^{m-1}}{m-1} \right) \vee \frac{2m\,\kappa^{m-1}}{(m-1) \left[m\,(E+1)\,\kappa^{m-1} - R_0^{2} \right]} \right\} ,
%\end{aligned}
\end{equation}
whereas we suppose that $ t_0 $ satisfies \eqref{eq: pme-M-init-upper-t_0-cri}. The definitions of $ R_0 $ and $E$ above are the same as in \eqref{eq: choice-R0} and \eqref{eq: consequence-R0} (under the additional constraint that the r.h.s.~of \eqref{barrier3} is increasing for $ r \ge R_0 $): in fact it is not difficult to show that they can be taken (large and small enough, respectively) to be independent of $\mu \in (-1,1] $. Note that the r.h.s.~of \eqref{eq: cond-crit-kappa-eta-b} is obtained by taking the first-order approximation of condition \eqref{C cond-bis-2} with respect to $ 1-\mu $, where $C$ is as in \eqref{eq: cond-crit-kappa-eta} subject to \eqref{eq: cond-crit-kappa-eta-a} (the latter, up to the strict inequality, is equivalent to \eqref{C cond-bis-1}). Hence, in view of Lemma \ref{pro: upper-barrier}, the function $ \Ubar=\Ubar_\mu $ defined by \eqref{upper-in-lemma} with the choices \eqref{eq: pme-M-init-upper-t_0-cri} and \eqref{eq: cond-crit-kappa-eta}--\eqref{eq: cond-crit-kappa-eta-b} satisfies
\begin{equation}\label{eq: pme-M-mu}
\left(\Ubar_\mu\right)_{\!t} \ge \left(\Ubar_\mu^m\right)_{\!rr}+\frac{(n-1)\psi'_\mu}{\psi_\mu} \left(\Ubar_\mu^m\right)_{\!r} \quad \textrm{in }  \mathbb{R}^+ \! \times \mathbb{R}^+
\end{equation}
for $ \mu $ close enough to $1$, where, to avoid ambiguity, by $ \psi_\mu $ we denote the function $ \psi $ of Lemma \ref{lem: ode-comparison} corresponding to $\mu$. The statement then follows by passing to the limit in \eqref{upper-in-lemma} and \eqref{eq: pme-M-mu} as $ \mu \to 1^- $, upon observing that $ \psi_\mu $ locally converges to $ \psi_1 $ along with its derivatives. The removal of the strict inequalities in \eqref{eq: cond-crit-kappa-eta-a} and \eqref{eq: cond-crit-kappa-eta-b} is an immediate consequence of another passage to the limit.
%It is just a matter of passing to the limit as $ \nu \to 1^- $ in the barriers \eqref{upper-in-lemma} constructed in Lemma \ref{pro: upper-barrier}, by taking care of the asymptotic developments of the constants $C,\gamma,R_0,t_0 $ under \eqref{eq: pme-M-init-upper-c1}-\eqref{eq: pme-M-init-upper-c3}. Since, in the differential inequality, we have to pass to the limit on $ \psi $ as well (as a solution of \eqref{ode-psi-upper-lemma} with \eqref{ode-psi-upper-f-lemma}), we just have to notice that such $ \psi $ is locally (and globally because of the condition on $R_0$, recall \eqref{eq: psi-simple-3}) stable.
\end{proof}

\begin{proof}[Proof of Theorem \ref{pro: ltb-cri-upper}]
Once we have at disposal Lemma \ref{pro: upper-barrier-2}, we can let $ Q=Q_1 $ and proceed as in the proof of Theorem \ref{pro: ltb-upper}. In particular, we set $ R_1=R_0 $, $ t_1=t_0 $ satisfying \eqref{eq: pme-M-init-upper-t_0-cri} and suppose that $ u_0 $ is supported in $ B_{\mathcal{R}}(o) $ for some $ \mathcal{R} \ge R_1 $ and has maximum $ \mathcal{M}>0 $. In order to ensure that $ \Ubar(r(x),0) \ge u_0(x) $ it is enough to require that $ \kappa=\kappa_1 $ and $ \eta=\eta_1 $ fulfil \eqref{eq: cond-crit-kappa-eta-a}, \eqref{eq: cond-crit-kappa-eta-b},
$$ \kappa_1 \ge t_1^{\frac{1}{m-1}} \, \mathcal{M} \quad \textrm{and} \quad \eta_1 \ge 1+\log \mathcal{R} - \frac12 \log\log t_1 \, . $$
As for \eqref{smoothing2}, it suffices to remark that $ \Ubar(\cdot,t) $ is supported precisely in $ B_{\mathsf{R}_1(t)}(o) $ with $ \mathsf{R}_1(t) := e^{\eta_1}\left[ \log(t+t_1) \right]^{1/2} $ and attains its maximum at $ r=0 $.
%Same observations as in the proof of Theorem \ref{pro: ltb-upper}. Then we have to pick the parameters $ \kappa,\eta,t_0 $ carefully so that $ \Ubar(\rho(x),0) \ge u_0(x) $.
%If we want to put $ \Ubar(\cdot,0) $ as in \eqref{upper-in-lemma-critic} above an initial datum supported in $ B_R $ (let $ R>1 $) with maximum $ \mathcal{M} $, we can e.g.~require
%$$ \kappa \ge e^{\frac{1}{m-1}} \, \mathcal{M} \, , \quad \eta \ge 1+\log R \, . $$
%Such conditions should then be added to \eqref{eq: pme-M-init-upper-c1-2}-\eqref{eq: pme-M-init-upper-c2-2}.
\end{proof}

\begin{lem}[Lower barriers, $ \mu=1 $]\label{pro: lower-barrier-2}
Let $m>1$. Under the same assumptions and with the same notations as in Lemma \ref{lem: ode-comparison-2} in the case $\mu=1 $, consider the  function $\ubar(r,t)$ defined by \eqref{upper-in-lemma-critic}. Then $ \ubar $ is a \emph{subsolution} to the equation \eqref{eq: pme-M} provided $ R_0 $ is larger than a lower bound depending only on $ n,m,R,Q,D $ and $ \kappa , \eta, t_0 $ comply with the following conditions:
%\begin{equation}\label{lower-state-2}
%\ubar(r,t) :=
%\begin{cases}
%\kappa \, (t+t_0)^{-\frac1{m-1}} \left[ -\eta + \frac12 \log\log(t+t_0) - \frac{r^2}{2} + \frac12 \right]_+^{\frac1{m-1}}  & \! \forall r \in [0,1] \, , \ \, \forall t \ge 0 \, , \\
%\kappa \, (t+t_0)^{-\frac1{m-1}} \left[ -\eta + \frac12 \log\log(t+t_0) - \log r  \right]_+^{\frac1{m-1}} & \! \forall r > 1 \, , \ \, \forall t \ge 0 \, ,
%\end{cases}
%\end{equation}
%where $ \kappa,\eta \ge 0 $ and $ t_0>1 $. Then $ \Ubar $ is a \emph{supersolution} to \eqref{eq: pme-M}
%with initial datum
%\begin{equation}\label{eq: pme-M-init-lower-2}
%u(r,0) =
%\begin{cases}
%\kappa \, t_0^{-\frac1{m-1}} \left(-\eta + \frac12 \log\log t_0 -\frac{r^2}{2}+\frac12\right)_+^{\frac1{m-1}} & \forall r \in [0,1]  \, , \\
%\kappa \, t_0^{-\frac1{m-1}} \left(-\eta + \frac12 \log\log t_0 -\log r \right)^{\frac1{m-1}}_+ & \forall r > 1 \, ,
%\end{cases}
%\end{equation}
%for all $ \kappa $, $ \eta $ and $ t_0 $ satisfying the following conditions:
\begin{equation}\label{eq: pme-M-init-lower-c1-2}
0< \kappa \le \left[ \frac{1}{2m \, (n-1)\sqrt{2Q}} \right]^{\frac{1}{m-1}}  \wedge \left[ \frac{R_0^2}{4m \, (F+1)} \right]^{\frac{1}{m-1}}  \wedge \left[ \frac{(m-1)\,R_0^2}{m\left[2+(m-1)(F+1) \right]} \right]^{\frac{1}{m-1}} ,
\end{equation}
\begin{equation}\label{eq: pme-M-init-lower-c2-2}
\eta \le \frac{1}{2} \log \left( \frac{2m\,\kappa^{m-1}}{m-1} \right) ,
\end{equation}
\begin{equation}\label{eq: pme-M-init-lower-c2-3}
\eta + \frac12 \log\log t_0 \ge \log R_0 + \frac{(m-1)\,R_0^2-4m\,\kappa^{m-1} }{4(m-1)\,R_0^2-4m(m-1)\,(F+1)\,\kappa^{m-1}} \, ,
\end{equation}
where $ F $ is a positive constant depending only on $n,R,R_0,Q,D$.
\end{lem}
%\color{blue}
%\begin{rem}\label{rem: curvature-one}\rm
%In the assumptions on curvatures we can pick, with no loss of generality, $ Q=(1+\mu)^2 $. In fact, on radial functions (and we use radial barriers) we can always resort to the scaling $ u(\lambda r , \lambda^2 t ) $ with
%$$ \lambda =  \left[ \frac{Q}{(1+\mu)^2} \right]^{\frac{1}{2(1+\mu)}} \, . $$
%\end{rem}
%\normalcolor
%\subsection{The case  \boldmath$\alpha,\beta\in(0,1)$ in Theorems \ref{pro: ltb-lower}, \ref{pro: ltb-upper}} \label{sec.proofs.3}
\begin{proof}
We mimick the proof of Lemma \ref{pro: upper-barrier-2}. That is, based on \eqref{C cond-bis-1-lower}--\eqref{C cond-bis-3-lower}, we let
\begin{equation}\label{eq: cond-crit-kappa-eta-lower}
C = \left[\frac{h}{2 m \, (1-\mu) \, (n-1)\sqrt{2Q}}\right]^{\frac1{m-1}} = \frac{\kappa}{(1-\mu)^{\frac{1}{m-1}}} \, , \quad \gamma = 1 + \eta \, (1-\mu)
\end{equation}
and pick $ \kappa $, $ \eta $, $ t_0 $ satisfying \eqref{eq: pme-M-init-lower-c1-2}--\eqref{eq: pme-M-init-lower-c2-3} with \emph{strict} inequalities. Again, the definitions of $ R_0 $ and $F$ above are the same as in \eqref{eq: choice-R0-lower} and \eqref{eq: consequence-R0-lower} (under the additional constraint that the r.h.s.~of \eqref{barrier1-lower-ter} is increasing for $ r \ge R_0 $), since they can be taken to be independent of $\mu \in (0,1] $. Conditions \eqref{eq: pme-M-init-lower-c2-2}--\eqref{eq: pme-M-init-lower-c2-3} are obtained by means of a first-order approximation of \eqref{C cond-bis-2-lower}--\eqref{C cond-bis-3-lower} with respect to $1-\mu$, where $C$ as in \eqref{eq: cond-crit-kappa-eta-lower}. Lemma \ref{pro: lower-barrier} ensures that $ \ubar=\ubar_\mu $ defined by \eqref{lower-in-lemma} subject to \eqref{eq: pme-M-init-lower-c1-2}--\eqref{eq: pme-M-init-lower-c2-3} satisfies
\begin{equation}\label{eq: pme-M-mu-lower}
\left(\ubar_\mu\right)_{\!t} \le \left(\ubar_\mu^m\right)_{\!rr}+\frac{(n-1)\psi'_\mu}{\psi_\mu} \left(\ubar_\mu^m\right)_{\!r} \quad \textrm{in }  \mathbb{R}^+ \! \times \mathbb{R}^+
\end{equation}
for $ \mu $ close enough to $1$, where by $ \psi_\mu $ we denote the function $ \psi $ of Lemma \ref{lem: ode-comparison-2} corresponding to such $\mu$. One then passes to the limit in \eqref{lower-in-lemma} and \eqref{eq: pme-M-mu-lower} exactly as in the proof of Lemma  \ref{pro: upper-barrier-2}.
% The thesis then just follows by passing to the limit in \eqref{upper-in-lemma} and \eqref{eq: pme-M-mu}, upon observing that $ \psi_\mu $ locally converges to $ \psi_1 $ along with its derivatives. The removal of the strict inequalities in \eqref{eq: cond-crit-kappa-eta-a} and \eqref{eq: cond-crit-kappa-eta-b} is an immediate consequence of another passage to the limit.
%We notice, for further reference, that as $\beta\to2$, we can require e.g. what follows. Assume that
%\[0<h<\frac12\wedge(n-1)(m-1).\]
%We can the require that, again as $\beta\to2$, there exists $\varepsilon=\varepsilon(n,m,h)$ such that for all $\beta\in(2-\varepsilon,2)$ one has
%\[
%0<\gamma<1+(2-\beta)\log\left[\frac h{(n-1)(m-1)}\right].
%\]
%If one sets $\gamma=1-\eta(2-\beta)$ the above condition holds, given $h$ and $\beta$ satisfying the previous constraints, provided (for instance)
%\[
%\log\left[\frac{(n-1)(m-1)}h\right]<\eta<\frac1{2(2-\beta)}.
%\]
%Moreover, straightforward first-order developments yield the existence of a positive constant $\varepsilon=\varepsilon(n,m,h,\eta)$ such that for all $\beta\in(2-\varepsilon,2)$ condition \eqref{limit2} holds if e.g.
%\[
%t_0 \ge e^{e^{2\eta + 4\frac{(n-1)(m-1)-h}{(m-1)[2(n-1)-nh]} }} \, .
%\]
\end{proof}

\begin{proof}[Proof of Theorem \ref{pro: ltb-cri-lower}]
In view of Lemma \ref{pro: lower-barrier-2}, we set $ Q=Q_2 $, $ R_2=R_0 $, $ \kappa=\kappa_2 $, $ \eta=\eta_2 $, $ t_0=t_2 $
$$ D = \max\left\{ \frac{1}{n-1} \sup_{x \in {B}_{R}(o) } -\mathrm{Ric}_o(x) \, , \, Q_2 \, R^{2}  \right\} , $$
with $ R $ as in \eqref{hp-ricc-cri}, and argue similarly to the proof of Theorem \ref{pro: ltb-lower}. Namely, first of all we remark that condition \eqref{eq: pme-M-init-lower-c2-3} subject to \eqref{eq: pme-M-init-lower-c1-2} is implied by
\begin{equation*}\label{eq: pme-M-init-lower-t_0-bis-critical}
\eta_2 + \frac12 \log\log t_2 \ge \log R_2 + K
\end{equation*}
for some constant $ K > 0 $ that depends only on $n$, $ m $, $ R_2 $, $Q_2$, $F$. With no loss of generality, we also suppose that \begin{equation*}\label{eq: pme-M-init-lower-t_0-ter-critical}
\inf_{x \in B_{\mathcal{R}}(o)} u_0(x) =: \mathcal{L}>0 \, , \quad \mathcal{R}:=e^K R_2 \, .
\end{equation*}
%In order to ensure that $ \Ubar(r(x),0) \ge u_0(x) $ it is enough to require that $ \kappa=\kappa_1 $ and $ \eta=\eta_1 $ fulfil \eqref{eq: cond-crit-kappa-eta-a}, \eqref{eq: cond-crit-kappa-eta-b},
%$$ \kappa_1 \ge t_1^{\frac{1}{m-1}} \, \mathcal{M} \quad \textrm{and} \quad \eta_1 \ge 1+\log \mathcal{R} - \frac12 \log\log t_0 \, . $$
%As for \eqref{smoothing2}, it suffices to remark that $ \Ubar(\cdot,t) $ is supported precisely in $ B_{\mathsf{R}^\prime(t)}(o) $ with $ \mathsf{R}^\prime(t) := e^{\eta}\left[ \log(t+t_0) \right]^{1/2} $ and attains its maximum at $ r=0 $.
%Let us set $  R_2 =R_0 $. In order to complete the proof, we have to show that the parameters $ C=C_2 $, $ \gamma=\gamma_2 $ and $ t_0=t_2 $, subject to \eqref{eq: pme-M-init-lower-c1}--\eqref{eq: pme-M-init-lower-t_0}, can be chosen in such a way that $ \ubar(r(x),0) \le u_0(x) $ for all $ x \in M $. To this end, first of all we point out that, in view of \eqref{eq: pme-M-init-lower-c1}, condition \eqref{eq: pme-M-init-lower-t_0} is implied by
It is then straightforward to verify that $ \ubar(r(x),0) \le u_0(x) $ holds provided e.g.
$$
\kappa_2 \le \frac{\mathcal{L}}{\left( K + \frac{1}{2} \right)^{\frac{1}{m-1}} } \, ,
$$
\eqref{eq: pme-M-init-lower-c1-2}--\eqref{eq: pme-M-init-lower-c2-2} are satisfied and one picks $ t_2 $ as
$$  t_2 = e^{R_2^2 \, e^{2(K-\eta_2)}} \, . $$
Finally, as concerns the bounds \eqref{smoothing2-low}, note that $ \ubar(\cdot,t) $ is supported precisely in $ B_{\mathsf{R}_1(t)}(o) $ with $ \mathsf{R}_1(t) := e^{\eta_2}\left[ \log(t+t_2) \right]^{1/2} $ and its maximum is attained at $ r=0 $.
%it suffices to remark that $ \Ubar(\cdot,t) $ is supported precisely in $ B_{\mathsf{R}^\prime(t)}(o) $ with $ \mathsf{R}^\prime(t) := e^{\eta}\left[ \log(t+t_0) \right]^{1/2} $ and attains its maximum at $ r=0 $.
%Suppose now we want to put $ \ubar(\cdot,0) $ as in \eqref{upper-in-lemma-critic} below an initial datum having minimum $ \mathcal{L} $ in $ B_R $. First of all let us assume, with no loss of generality, that $ R $ is such that
%$$ R \ge e^2 \, . $$
%Under this assumption, it is easy to check that by choosing any $ \kappa $ satisfying \eqref{eq: pme-M-init-lower-c1-2} plus
%$$ \kappa \le \left( \frac{2}{5} \right)^{\frac{1}{m-1}}  \mathcal{L}  \, , $$
%any $ \gamma $ complying with \eqref{eq: pme-M-init-lower-c2-2} and by setting
%$$ t_0 = e^{e^{2\eta+4}} \, , $$
%the function $ \ubar $ is indeed a lower barrier lying below the given initial datum at $ t=0 $.
\end{proof}
%\textcolor{blue}{
%We have not performed explicit calculations when $\nu=\mu=0$. Such case can be dealt with with minor modifications of the methods of \cite{V}.}
%By taking inspiration from Section \ref{sect: exa-mod} (Type I), let us construct a function $ \psi $ (Lipschitz, for simplicity) such that
%\begin{equation}\label{hp-constr-psi}
%\mathrm{K}_\omega(x) \le - \frac{\psi^{\prime\prime}(\rho)}{\psi(\rho)} \quad \forall x \in M \, ,
%\end{equation}
%under the assumption that
%\begin{equation}\label{hp-constr-psi-1}
%\mathrm{K}_\omega(x) \le -Q_1 \, r^{2 \nu} \quad \forall x \in M \setminus B_R(o)
%\end{equation}
%for some $ R>0 $ and $ \nu \in (-1,1) $. The idea is to connect the function
%$$  r \mapsto e^{\frac{\sqrt{Q_1-\varepsilon}}{\nu+1} \, r^{\nu+1}} $$
%with the function $ r \mapsto r $, where $ \varepsilon \in (0,Q_1) $ is a fixed number, arbitrarily small.
%The above computations are useful to our purposes provided we can show that the solution of the ODE
%\begin{equation}\label{ode-psi-estimates}
%\frac{\psi''(r)}{\psi(r)} = Q_0 \, r^{2\nu} \quad \forall r>R \, , \quad \psi'(R)=a>0 \, , \quad \psi(R)=b>0  \, ,
%\end{equation}
%where $ R>0 $ is an arbitrarily large parameter, $ Q_0>0 $ and $ \nu \in (-1,1) $, is such that
%\begin{equation}\label{ode-psi-estimate-fond}
%\frac{\psi^\prime(r)}{\psi(r)} \sim \sqrt{Q_0} \, r^\nu \, .
%\end{equation}

%%%%%%%%%%%%%%%%%%%%%%%%%%%%%%%%%%%%%%%%%%%%%%%%%%%%%%%%%%%%%%%%%%%%%%%%%%%%
\section{Statement of the results in the quasi-Euclidean range}\label{sec.smr.qEr}
Now we turn to quasi-Euclidean settings. First we deal with subcritical powers, namely powers strictly smaller than $-2$.

\begin{thm}[Upper bounds, quasi-Euclidean subcritical powers]\label{pro: upper-qe1}
Let $ M $, $ \mathrm{K}_\omega $, $u$ and $ \mathsf{R}(t) $ be as in  Theorem \ref{pro: ltb-upper}. Suppose that
\begin{equation}\label{hp-sect-qe1}
\mathrm{K}_\omega(x) \le -Q_1 \, r^{2\mu_1} \quad \forall x \in M \setminus B_R(o)
\end{equation}
for some $ \mu_1 < -1 $ and $ Q_1,R>0 $. Then the following bound holds:
\begin{equation}\label{bound-qe1-2}
u(x,t) \le \frac{C_1}{(t+t_1)^{\frac{n}{2+n(m-1)}}} \left(\gamma_1 - \frac{r^2}{(t+t_1)^{\frac2{2+n(m-1)}}} \right)^{\frac1{m-1}}_+ \quad \forall (x,t) \in M \times \mathbb{R}^+ \, ,
\end{equation}
where $  C_1 , \gamma_1 , t_1 $ are positive constants depending on $ n,m,Q_1,R,u_0 $. In particular, the estimates
\begin{equation}\label{smoothing-qe1}
\|u(t)\|_\infty \le \frac{\widetilde{C}_1}{t^{\frac{n}{2+n(m-1)}}} \quad \textrm{and} \quad \mathsf{R}(t) \le \frac{\widetilde{\gamma}_1}{t^{\frac{1}{2+n(m-1)}}}
\end{equation}
hold for suitable $ \widetilde{C}_1 , \widetilde{\gamma}_1  >0 $ and all $t$ large enough.
\end{thm}

\begin{thm}[Lower bounds, quasi-Euclidean subcritical powers]\label{pro: lower-qe1}
Let $ M , u $ be as in  Theorem \ref{pro: ltb-upper} and $ \mathrm{Ric}_o , \mathsf{R}(t) $ as in Theorem \ref{pro: ltb-lower}. Suppose that
\begin{equation}\label{hp-sect-qe1-low}
\mathrm{Ric}_o(x) \ge -(n-1)\, Q_2  \, r^{2\mu_2} \quad \forall x \in M \setminus B_R(o)
\end{equation}
for some $ \mu_2 < -1 $ and $ Q_2,R>0 $. Then the following bound holds:
\begin{equation}\label{bound-qe1-2-low}
u(x,t) \ge \frac{C_2}{(t+t_2)^{\frac{n}{2+n(m-1)}}} \left(\gamma_2 - \frac{r^2}{(t+t_2)^{\frac2{2+n(m-1)}}} \right)^{\frac1{m-1}}_+ \quad \forall (x,t) \in M \times \mathbb{R}^+ \, ,
\end{equation}
where $  C_2 , \gamma_2 , t_2 $ are positive constants depending on $ n,m,Q_2,R,u_0 $ and on $ \inf_{x \in B_R(o)}$ $\mathrm{Ric}_o(x) $. In particular, the estimates
\begin{equation}\label{smoothing-qe1-low}
\|u(t)\|_\infty \ge \frac{\widetilde{C}_2}{t^{\frac{n}{2+n(m-1)}}} \quad \textrm{and} \quad \mathsf{R}(t) \ge \frac{\widetilde{\gamma}_2}{t^{\frac{1}{2+n(m-1)}}}
\end{equation}
hold for suitable $ \widetilde{C}_2,\widetilde{\gamma}_2>0 $ and all $t$ large enough.
\end{thm}

We finally consider the quasi-Euclidean setting in the critical case where curvatures are bounded by inverse-quadratic powers.

\begin{thm}[Upper bounds, quasi-Euclidean critical power]\label{pro: upper-qe1-cri}
Let $ M $, $ \mathrm{K}_\omega $, $u$ and $ \mathsf{R}(t) $ be as in  Theorem \ref{pro: ltb-upper}. Suppose that
\begin{equation}\label{hp-sect-qe1-cri}
\mathrm{K}_\omega(x) \le -Q_1 \, r^{-2} \quad \forall x \in M \setminus B_R(o)
\end{equation}
for some $ Q_1,R>0 $. Let
\begin{equation}\label{hp-qe1-cri-param}
q := \frac{1+\sqrt{1+4Q_1}}{2} \, ,\quad n_q := 1+q(n-1) \, .
\end{equation}
Then the following bound holds:
\begin{equation}\label{bound-qe1-cri}
u(x,t) \le \frac{C_1}{(t+t_1)^{\frac{n_q}{2+n_q(m-1)}}} \left(\gamma_1 - \frac{r^{2}}{(t+t_1)^{\frac{2}{2+n_q(m-1)}}} \right)^{\frac1{m-1}}_+ \quad \forall (x,t) \in M \times \mathbb{R}^+ \, ,
\end{equation}
where $  C_1 , \gamma_1 , t_1 $ are positive constants depending on $ n,m,Q_1,R,u_0 $. In particular, the estimates
\begin{equation}\label{smoothing-qe1-cri}
\|u(t)\|_\infty \le \frac{\widetilde{C}_1}{t^{\frac{n_q}{2+n_q(m-1)}}} \quad \textrm{and} \quad \mathsf{R}(t) \le \frac{\widetilde{\gamma}_1}{t^{\frac{1}{2+n_q(m-1)}}}
\end{equation}
hold for suitable $ \widetilde{C}_1 , \widetilde{\gamma}_1 >0 $ and all $t$ large enough.
\end{thm}

\begin{thm}[Lower bounds, quasi-Euclidean critical power]\label{pro: lower-qe1-cri}
Let $ M , u $ be as in  Theorem \ref{pro: ltb-upper} and $ \mathrm{Ric}_o , \mathsf{R}(t) $ as in Theorem \ref{pro: ltb-lower}. Suppose that
\begin{equation}\label{hp-sect-qe1-low-cri}
\mathrm{Ric}_o(x) \ge -(n-1) \, Q_2  \, r^{-2} \quad \forall x \in M \setminus B_R(o)
\end{equation}
for some $ Q_2,R>0 $. Let $ q $ and $ n_q$ be as in \eqref{hp-qe1-cri-param}, with $ Q_1 $ replaced by $ Q_2 $. Then the following bound holds:
\begin{equation}\label{bound-qe1-2-low-cri}
u(x,t) \ge \frac{C_2}{(t+t_2)^{\frac{n_q}{2+n_q(m-1)}}} \left(\gamma_2 - \frac{r^{2}}{(t+t_1)^{\frac{2}{2+n_q(m-1)}}} \right)^{\frac1{m-1}}_+ \quad \forall (x,t) \in M \times \mathbb{R}^+ \, ,
\end{equation}
where $  C_2 , \gamma_2 , t_2 $ are positive constants depending on $ n,m,Q_2,R,u_0 $ and on $ \inf_{x \in B_R(o)}$ $\mathrm{Ric}_o(x) $. In particular, the estimates
\begin{equation}\label{smoothing-qe1-low-cri}
\|u(t)\|_\infty \ge \frac{\widetilde{C}_2}{t^{\frac{n_q}{2+n_q(m-1)}}} \quad \textrm{and} \quad \mathsf{R}(t) \ge \frac{\widetilde{\gamma}_2}{t^{\frac{1}{2+n_q(m-1)}}}
\end{equation}
hold for suitable $ \widetilde{C}_2 , \widetilde{\gamma}_2 >0 $ and all $t$ large enough.
\end{thm}

The proofs of all the above results rely on the change of variables that we will discuss in detail in the next Sections \ref{sec: chvar} and \ref{sec.proofs-euclid}.

\begin{rem}\label{rem: quasi-sub}\rm
Note that the bounds in Theorems \ref{pro: upper-qe1}--\ref{pro: lower-qe1} are dimension dependent (but independent of curvature) and coincide with corresponding Euclidean ones, in contrast with the quasi-hyperbolic bounds given in Theorems \ref{pro: ltb-upper}--\ref{pro: ltb-cri-lower}. This is expected when curvature decays to zero fast enough. On the other hand, the curvature decay $-Q\,r^{-2}$ is critical for our goals, in the sense that bounds in Theorems \ref{pro: upper-qe1-cri}--\ref{pro: lower-qe1-cri} depend both on the dimension and on the multiplicative constants $Q$. In this regard, note also that \eqref{bound-qe1-cri} and \eqref{bound-qe1-2-low-cri} can formally be seen as analogues of the corresponding Euclidean bounds in the ``fractional'' dimension $n_q>n$ defined in \eqref{hp-qe1-cri-param}.
%\begin{equation} \label{eq: nq-p}
%n_Q:= 2 \, \frac{n-p}{2-p} \, , \quad p=p(n,\alpha):=\frac{2(n-1)(\alpha-1)}{(n-1)\alpha-1} \, .
%\end{equation}
Clearly, as $ Q \to 0 $ it follows that $ q \to 1 $ and then $ n_q \to n $. Hence, the above bounds interpolate smoothly between the exponents of the Euclidean and the hyperbolic space, although a logarithmic correction appears in the latter case.
\end{rem}

%\color{red} Questions \\
%\noindent 1) mettere qui anche l'enunciato pesato, oppure lasciarlo in Section \ref{sect: chvar-cors}? IO SAREI PER LASCIARLO NELLA SEZIONE \ref{sect: chvar-cors}, MA CHIEDERE A JLV \\
%2) nel caso di curvatura che vada come $-r^{a}$ con $a>2$, si vede facilmente usando il cambio di variabili e i risultati noti nel pesato euclideo che esiste una sol a sepazione di variabili, con le ovvie conseguenze sull'asintotica almeno locale (poi ci sono anche le Barenblatt, forse bisogna fare un matching). Siamo indecisi sul mettere o meno qualcosa, perche' probabilmente non sarebbe brevissimo, e l'articolo gia' non e' cortissimo. dicci tu. NON SI METTE NULLA, AL PI\`U ACCENNARE QUALCOSA SEMPRE NELLA SEZIONE \ref{sect: comm-opp}.
%\normalcolor

%%%%%%%%%%%%%%%%%%%%%%%%%%%%%%%%%%%%%%%%%%%%%%%%%%%%%%%%%%%%%%%%%%%%%%%%%%%%%%%%%%
\section{Radial change of variables into Euclidean diffusion with weight: a general approach}\label{sec: chvar}
In this section,  we transform radial solutions of the porous medium equation on certain model manifolds into radial solutions of the same equation in Euclidean space with a particular weight (or density) by means of a suitable change of variables, which is presented in Section \ref{sect: the-change}.  Such a transformation, which by the way applies to a larger class of diffusion problems, has been introduced by \cite[Section 6]{V}, and was an essential tool in the study of the asymptotics of solutions of the porous medium equation in the special case of \emph{hyperbolic space}. Here, we aim at extending it to a quite general class of \emph{quasi-hyperbolic} (according to the terminology adopted in Section \ref{desc-res}) or \emph{super-hyperbolic} (namely, corresponding to $\mu>1$) Riemannian models of dimension $ n \ge 3 $. This is discussed in detail in Section \ref{sect: dir-pb}. By means of slight modifications to our approach, it is also possible to deal with \emph{quasi-Euclidean} Riemannian models and to cover the 2-dimensional cases: the former are addressed in Section \ref{sect: chvar-exa}, along with a large class of explicit examples, whereas the latter are investigated in Section \ref{sect: 2-dim}. Nevertheless, the assumptions that $n \ge 3$ and that the model is \emph{super-Euclidean} (i.e.~quasi- or super- hyperbolic) allow us to obtain a very general result, see Theorem \ref{thm: chvar-gen} below.

%%%%%%%%%%%%%%%%%%%%%%%%%%%%%%%%%%%%%%%%%%%%%
\subsection{The change of variables}\label{sect: the-change}
Let us present the main ideas. We start with a radial solution $u=u(r,t)$ (for all $ t>0 $) of the porous medium equation in some general geometry of the type of a Riemannian model manifold $M$ with metric given by $\mathrm{d}l^2=\mathrm{d}r^2+ [\psi(r)]^2 \, \mathrm{d}\omega^2$ with $r\ge0$,  $\omega\in \mathbb{S}_{n-1}$,
that is
\begin{equation}\label{eq.hpme}
u_t=\Delta_g\!\left(u^m\right) \quad \mathrm{with} \quad \Delta_g v = v^{\prime\prime} + \frac{(n-1)\psi^\prime}{\psi} \, v^\prime
\end{equation}
for all regular radial functions $ v=v(r) $, where $ \psi \in \mathcal{A} $ and the class $ \mathcal{A} $ is defined by \eqref{eq: class-mod}. At this stage we take  dimension $n\ge 3$ for convenience. Through the change of variables $s=s(r)$, we want to transform \eqref{eq.hpme} into an equation for ${\hat u}(s,t)= u(r(s),t)$ of the form
\begin{equation}\label{eq.wpme}
\rho(s) \, \hat{u}_t= \Delta_s \! \left( \hat{u}^m \right) ,
\end{equation}
where now $\Delta_s$ is the Euclidean \emph{radial} Laplacian in dimension $n$ (the radial coordinate being $s=|x|$),
%namely $ \Delta_s=\Delta_r $ with $ \psi(s)=s $,
and $ \rho(s) $ is a suitable positive weight. It follows that we need
\begin{equation}\label{eq.diff}
\frac{\mathrm{d}s}{s^{n-1}}=\frac{\mathrm{d}r}{[\psi(r)]^{n-1}} \, .
\end{equation}
By integrating between $r$ and $+\infty$ we find the identity
\begin{equation}\label{eq: chvar-initial}
\frac1{(n-2) \, s^{n-2}} = \int_r^{+\infty}\frac1{[\psi(t)]^{n-1}}\,{\rm d}t \, .
\end{equation}
Hence, an additional assumption we need require on $ \psi $ is
\begin{equation}\label{eq: chvar-hyp-int-1}
\int_1^{+\infty} \frac1{[\psi(t)]^{n-1}}\,{\rm d}t < +\infty \, .
\end{equation}
% remark: if this integral is infinite the method does not work in dimension n>=3, since we can represent only part of $M$; however it may work in dimension $n=2$, since the left-hand integral gives a log(s)
%On the other hand, since $ \psi $ represents a model manifold (i.e.~it belongs to $ \mathcal{A} $), it has to satisfy
Recalling that, by assumption,
\begin{equation}\label{eq: ass-psi-zero}
\lim_{r \to 0^+} \psi(r) = 0 \, , \quad \lim_{r \to 0^+} \psi^\prime(r) = 1 \, ,
\end{equation}
it follows that
\begin{equation}\label{eq: ass-psi-zero-S}
\lim_{r \to 0^+} \frac{s(r)}{r} = 1 \, , \quad  \lim_{r \to 0^+} s^\prime(r) = 1\,,
\end{equation}
so that the  change of variables is very uniform at $s=0$. It is then easy to show that
\begin{equation}\label{rho.def}
\rho(s)=\frac{\left[\psi\!\left(r(s)\right) \right]^{2(n-1)}}{s^{2(n-1)}} \quad \forall s \in (0,+\infty) \, .
\end{equation}
In view of \eqref{eq: ass-psi-zero}--\eqref{eq: ass-psi-zero-S} we then deduce that $ \rho \in C^\infty((0,+\infty)) \cap C([0,+\infty)) $ is a positive function such that
\begin{equation}\label{eq: expr-rho-zero-intro}
\lim_{s \to 0^+} \rho(s) = 1 \, , \quad \lim_{s \to 0^+} s \, \rho^\prime(s) = 0 \, .
\end{equation}
For future convenience, analogous to \eqref{eq: class-mod}, let us label the class of functions the weight $ \rho $ belongs to:
\begin{equation}\label{eq: class-wei}
\begin{aligned}
& \mathcal B:=
& \left\{ \rho \in C^\infty((0,+\infty))\cap C([0,+\infty)): \ \rho(0)=1 \, , \ \lim_{s \to 0^+} s \, \rho^\prime(s) = 0 \, , \ \rho>0  \right \} .
\end{aligned}
\end{equation}
We are really interested in the behaviour of $\rho(s)$ for large $s$. It depends on the behaviour of $\psi(r)$ (and of $ \psi^\prime(r) $) as $ r \to +\infty $. We shall discuss such issues in detail in Subsections \ref{sect: dir-pb}--\ref{sect: chvar-exa}. For the moment just note that, by \eqref{eq.diff} and \eqref{rho.def}, the inverse transformation to \eqref{eq: chvar-initial} reads
\begin{equation}\label{eq: int-rho12-int}
r=\int_{0}^s \left[ \rho(\tau) \right]^{\frac{1}{2}} \mathrm{d}\tau \, .
\end{equation}
% comment: we may consider $\rho$ defined in a finite interval $[0,s_1)$ on the condition that the integral above be still divergent at $s_1$ (in the examples, use inverse powers of the distance to the border). Not here.

%%%%%%%%%%%%%%%%%%%%%%%%%%%%%%%%%%%%%%%%%%%%%%%%%%%%%%%%%%%%%%%%%
\subsection{An equivalence result for quasi- and super-hyperbolic Riemannian models}\label{sect: dir-pb}

% \color{red} This section is completely revised. Please re-read it and approve or change \normalcolor
Let us now focus the precedent ideas on model manifolds that fall within the quasi-hyperbolic or super-hyperbolic framework, thus significantly differing from the Euclidean space. We are able to deal with a class of models more general than the ones considered in Section \ref{sect: exa-mod}, namely we need not curvatures to behave like powers at infinity. Such a class of models will give rise to weights $ \rho(s) $ which turn out to be very small perturbations of $ s^{-2} $ at infinity. We also consider the opposite problem, namely we start from weights that are suitable small perturbations of $ s^{-2} $ and show that the associated manifolds belong precisely to the same class we start from.

Note that calculating $\rho$ given $\psi$ is in principle simple: integration of equation \eqref{eq: chvar-initial} allows us to find the relation between $s$ and $r$ and then formula \eqref{rho.def} yields the expression for $\rho$. The converse transformation proceeds similarly: given $\rho$ we use formula \eqref{eq: chvar-initial} to find $s$ as a function of $r$, and then use again \eqref{rho.def} to get
\begin{equation}\label{rho.def-inv}
\psi(r)=s(r)\left[\rho\!\left(s(r)\right)\right]^{\frac{1}{2(n-1)}} \quad \forall r \in (0,+\infty) \, .
\end{equation}
The nontrivial part is to express $ r $ as a function of $s$ and vice versa: in general one can only obtain asymptotic estimates for such relation, which will have to be plugged into \eqref{rho.def} and \eqref{rho.def-inv} to get in turn asymptotic information on $ \rho(s) $ and $ \psi(r) $, respectively.

\begin{thm}\label{thm: chvar-gen}
Let $ n \ge 3 $. Assume that $ \psi \in \mathcal{A} $ where $\mathcal{A}$ is given in \eqref{eq: class-mod}, with $ \lim_{r \to +\infty} \psi(r)=+\infty $ and
%\begin{equation}\label{eq: ass-f-3}
%\lim_{r \to +\infty} \frac{e^{-(n-1)f(r)}}{f^\prime(r)} = 0 \, ,
%\end{equation}
that there exists a positive function $ h \in C^1(\mathbb{R}) $ satisfying
\begin{equation}\label{eq: ass-g-1}
\lim_{\tau \to +\infty} \frac{h^\prime(\tau)}{h(\tau)} = 0
\end{equation}
and
\begin{equation}\label{eq: ass-h-scale}
\lim_{\tau \to +\infty} \frac{h\!\left( \tau + \frac{1}{n-1} \log h(\tau) + \frac{1}{n-1} \log\left(\frac{n-1}{n-2}\right) + o(1) \right)}{h( \tau )} = 1
\end{equation}
% DA RIVEDERE
%\begin{equation}\label{eq: ass-g-2}
%\lim_{\tau \to +\infty} \tau \, g^\prime(\tau) = 0
%\end{equation}
% le condizioni \eqref{eq: ass-g-1} e \eqref{eq: ass-g-2} sono in generale indipendenti, basta considerare ad esempio le funzioni g(\tau)=e^{-\tau} e g(\tau)=t^{-1}*|sin(t)|+t^{-1/2}
such that
\begin{equation}\label{eq: ass-fg}
\lim_{r \to +\infty} \, \frac{ \left(\log \psi (r) \right)^\prime}{h\!\left( \log \psi(r) \right)} = 1 \, .
\end{equation}
Then the weight $ \rho $ related to $ \psi $ by \eqref{eq: chvar-initial} and \eqref{rho.def} belongs to $ \mathcal{B} $ where $\mathcal{B} $ is given in \eqref{eq: class-wei} and obeys the asymptotic laws
\begin{equation}\label{eq: chvar-as-final}
\rho(s) \sim \frac{ \left( \frac{n-2}{n-1} \right)^2}{s^2 \left[ h\!\left( \frac{n-2}{n-1} \, \log s \right) \right]^2 } \quad \textrm{as} \ s \to +\infty
\end{equation}
and
\begin{equation}\label{eq: chvar-as-final-deriv}
\lim_{s \to + \infty} \frac{s \, \rho^\prime(s)}{\rho(s)} = - 2 \, .
\end{equation}

\noindent Conversely, assume that $ \rho \in \mathcal{B} $ satisfies \eqref{eq: chvar-as-final-deriv} and \eqref{eq: chvar-as-final} for some positive function $ h \in C^1(\mathbb{R}) $ complying with \eqref{eq: ass-g-1}, \eqref{eq: ass-h-scale} and
\begin{equation}\label{eq: ass-int-h}
\int_1^{+\infty} \frac{1}{h(\tau)} \, \mathrm{d}\tau = + \infty \, .
\end{equation}
Then the model function $ \psi $ related to $ \rho $ by \eqref{eq: int-rho12-int} and \eqref{rho.def-inv} belongs to $ \mathcal{A} $, satisfies $ \lim_{r \to +\infty} \psi(r)=+\infty $ and obeys to \eqref{eq: ass-fg}.
\end{thm}

\begin{proof}  We divide the proof into two steps.

\noindent $\bullet$ {\sc Step 1. Direct problem: from $ \psi $ to $ \rho $}.

\noindent It is convenient to rewrite $ \psi $ as $ \psi(r)=e^{f(r)} $, where $ \log \psi = f \in C^\infty((0,+\infty)) $ satisfies
\begin{equation}\label{eq: ass-f-mid}
\lim_{r \to +\infty} f(r)=+\infty \, , \quad \lim_{r \to 0^+} f(r) = -\infty \, , \quad \lim_{r \to 0^+} f(r) + \log f^\prime(r) = 0
\end{equation}
in view of the assumptions on $ \psi $. Moreover, \eqref{eq: ass-g-1} implies
% si puo dimostrare facilmente per assurdo, senza usare de l'hopital
\begin{equation}\label{eq: ass-g-1-bis}
\lim_{\tau \to +\infty} \frac{\log h(\tau)}{\tau} = 0  \, ;
\end{equation}
in particular, by means of the change of variable $ \tau=f(t) $ we can deduce that
\begin{equation}\label{eq: ass-f-2}
\int_{1}^{+\infty} e^{-(n-1)f(t)} \, \mathrm{d}t < +\infty \, ,
\end{equation}
which ensures the validity of \eqref{eq: chvar-hyp-int-1}. By \eqref{eq: ass-g-1}, \eqref{eq: ass-fg}, \eqref{eq: ass-f-mid}, \eqref{eq: ass-g-1-bis} and \eqref{eq: ass-f-2} we get
\begin{equation}\label{eq: cons-f-1}
\int_{r}^{+\infty} e^{-(n-1)f(t)} \, \mathrm{d}t \sim \frac{e^{-(n-1)f(r)}}{(n-1) \, h\!\left( f(r) \right)} \to 0 \quad \textrm{as } r \to +\infty \, ,
\end{equation}
as follows from L'H\^{o}pital's rule. Upon defining the new space variable $ s \in (0,+\infty) $ through \eqref{eq: chvar-initial}, thanks to \eqref{eq: cons-f-1} there holds
\begin{equation}\label{eq: chvar-as-2}
s \sim \left[ \left( \frac{n-1}{n-2}  \right) h\!\left( f(r) \right) \right]^{\frac{1}{n-2}} e^{ \frac{n-1}{n-2} \, f(r)} \to +\infty \quad \textrm{as } r \to +\infty \, ,
\end{equation}
whence, by taking logarithms,
\begin{equation*}\label{eq: chvar-as-3}
\lim_{r \to +\infty} \left[ \frac{n-2}{n-1} \, \log s - \frac{1}{n-1} \, \log\left( \frac{n-1}{n-2} \right) - \frac{1}{n-1} \log h\!\left( f(r) \right) - f(r) \right] = 0 \, ,
\end{equation*}
that is
\begin{equation}\label{eq: chvar-as-4}
\frac{n-2}{n-1} \, \log s = f(r) + \frac{1}{n-1} \log h\!\left( f(r) \right) + \frac{1}{n-1} \, \log\left( \frac{n-1}{n-2} \right) + o(1) \, .
\end{equation}
Now we want to compute $ \rho(s) $ according to \eqref{rho.def}: as explained above, the fact that $ \rho \in \mathcal{B} $ is a direct consequence of the fact that $ \psi \in \mathcal{A} $. In order to establish the asymptotic behaviour as $ s \to +\infty $, by means of \eqref{eq: ass-h-scale} and \eqref{eq: chvar-as-2} we obtain:
\begin{equation}\label{eq: expr-rho-pre}
\rho(s)=\frac{e^{2(n-1)f(r)}}{s^{2(n-1)}} \sim  \frac{\left( \frac{n-2}{n-1} \right)^2}{s^2 \left[ h\!\left( f(r) \right) \right]^2 } \sim \frac{ \left( \frac{n-2}{n-1} \right)^2}{s^2 \left[ h\!\left( f(r) + \frac{\log h \left( f(r) \right)}{n-1}  + \frac{\log\left( \frac{n-1}{n-2} \right)}{n-1} + o(1) \right) \right]^2 } \, ,
\end{equation}
where for greater readability we have dropped the dependence of $r$ on $s$. Hence, the validity of \eqref{eq: chvar-as-final} is a consequence of \eqref{eq: chvar-as-4}. Furthermore, it is straightforward to show that \eqref{eq: chvar-as-final-deriv} also holds in view of \eqref{eq.diff}, \eqref{eq: ass-fg}, \eqref{eq: chvar-as-2} and \eqref{eq: expr-rho-pre}.

\noindent  {\sc Step 2. Inverse problem: from $ \rho $ to $ \psi $}.
%\noindent Let $ \rho \in C([0,+\infty)) \cap C^1((0,+\infty)) $ be a positive function satisfying \eqref{eq: expr-rho-zero-intro}, \eqref{eq: chvar-as-final-deriv} and \eqref{eq: chvar-as-final} for some positive $ h \in C^1(\mathbb{R}) $ that complies with \eqref{eq: ass-g-1}, \eqref{eq: ass-int-h} and \eqref{eq: ass-h-scale}.

\noindent Upon recalling \eqref{eq: int-rho12-int}, let us set
\begin{equation*}\label{conv-chvar}
r=\int_{0}^s \left[ \rho(\tau) \right]^{\frac{1}{2}} \mathrm{d}\tau =:g(s) \quad \forall s \in [0,+\infty) \, .
\end{equation*}
Note that the assumptions on $\rho$ along with \eqref{eq: ass-int-h} ensure that
$$ \int_{0}^{+\infty} \left[ \rho(\tau) \right]^{\frac{1}{2}} \mathrm{d}\tau = + \infty \, . $$
Hence, thanks to \eqref{rho.def-inv}, we have that $ \psi(r)=e^{f(r)} $ with
\begin{equation}\label{conv-chvar-2}
f(r)= \log g^{-1}(r) + \frac{1}{2(n-1)} \, \log \rho\!\left( g^{-1}(r) \right) \quad \forall r \in (0,+\infty) \, ,
\end{equation}
so that
\begin{equation}\label{conv-chvar-3}
f^\prime(r) = \frac{2(n-1)\, \rho\!\left( g^{-1}(r) \right) + g^{-1}(r) \, \rho^\prime\!\left( g^{-1}(r) \right) }{2(n-1) \, g^{-1}(r) \left[ \rho\!\left( g^{-1}(r) \right) \right]^{\frac{3}{2}} } \quad \forall r \in (0,+\infty) \, .
\end{equation}
Since $ \rho \in \mathcal{B} $, one checks that $ f \in C^\infty((0,+\infty)) $ satisfies \eqref{eq: ass-f-mid}, so that $ \psi \in \mathcal{A} $ and $ \lim_{r \to +\infty} \psi(r) = +\infty $. Moreover, due to \eqref{conv-chvar-2}--\eqref{conv-chvar-3},
\begin{equation}\label{conv-chvar-4}
\begin{aligned}
\lim_{r \to +\infty} \frac{f^\prime(r)}{h\!\left(f(r)\right)} = & \lim_{s\to+\infty} \frac{f^\prime\!\left( g(s) \right)}{h\!\left(f\!\left( g(s) \right)\right)} \\
= & \lim_{s\to+\infty} \frac{ 2(n-1)\, \rho(s) + s \, \rho^\prime( s ) }{2(n-1) \, s \left[ \rho(s) \right]^{\frac{3}{2}} h\!\left( \log s + \frac{1}{2(n-1)} \, \log \rho(s) \right) } \, .
\end{aligned}
\end{equation}
In addition, it is straightforward to check that \eqref{eq: ass-h-scale} is equivalent to
\begin{equation}\label{eq: ass-h-scale-2}
\lim_{\tau \to +\infty} \frac{h\!\left( \tau - \frac{1}{n-1} \log h(\tau) - \frac{1}{n-1} \log\left(\frac{n-1}{n-2}\right) + o(1) \right)}{h( \tau )} = 1 \, .
\end{equation}
As a consequence of \eqref{eq: chvar-as-final} and \eqref{eq: ass-h-scale-2}, we therefore obtain
\begin{equation}\label{conv-chvar-5}
\begin{aligned}
& h \! \left( \log s + \frac{1}{2(n-1)} \, \log \rho(s) \right) \\
= & h \! \left( \log s^{\frac{n-2}{n-1}} - \frac{1}{n-1} \, h\!\left( \log s^{\frac{n-2}{n-1}} \right) -\frac{1}{n-1} \log\left( \frac{n-1}{n-2} \right)+ o(1) \right) \\
\sim & h\!\left( \frac{n-2}{n-1} \, \log s  \right) \sim \frac{n-2}{(n-1) \, s \left[\rho(s)\right]^{\frac{1}{2}} } \, ,
\end{aligned}
\end{equation}
so that from \eqref{eq: chvar-as-final-deriv}, \eqref{conv-chvar-4} and \eqref{conv-chvar-5} we finally deduce that
$$ \lim_{r \to +\infty} \frac{f^\prime(r)}{h\!\left(f(r)\right)} = 1 \, , $$
namely \eqref{eq: ass-fg}.
\end{proof}

\begin{rem}\label{rem: funct-h} \rm
In Theorem \ref{thm: chvar-gen} it is apparent that a major role is played by the function $h$. Even though the latter is not given explicitly in terms of $ f=\log \psi $ and $ f^\prime $, we shall see in Section \ref{sect: chvar-exa} below a certain number of significant examples that fall within our framework where everything is easily computable. In fact by \eqref{eq: ass-fg} we are requiring that, at infinity, the function $f$ satisfies (approximately) the autonomous equation $ f^\prime = h(f) $. As it is clear from the proof, assumption \eqref{eq: ass-g-1} (which to some extent means the model is super-Euclidean) on $ h $ is crucial. Also \eqref{eq: ass-int-h}, which basically means that $ \psi(r) $ can be at most doubly exponential, is essential: it is a compatibility condition between \eqref{eq: ass-fg} and the fact that $ \psi(r) \to +\infty $ as $ r \to +\infty $. On the other hand, the technical assumption \eqref{eq: ass-h-scale} has been used in order to simplify the discussion and give a clear asymptotic profile for $ \rho(s) $. In principle it could be weakened. Finally, note that the hypothesis $ n \ge 3 $ is essential to have \eqref{eq: chvar-initial}.
%A few words on the role of the function $h$ within our assumptions, and on the various implications of our assumptions on it.
%
%\noindent In fact by \eqref{eq: ass-fg} we are implicitly requiring \eqref{eq: ass-int-h}.
%
%Note that \eqref{eq: chvar-hyp-int-1}--\eqref{eq: ass-psi-zero} hold. That is, a few word on the fact that under our assumptions the change of variables is indeed doable.
\end{rem}

\subsection{Some significant examples and quasi-Euclidean models}\label{sect: chvar-exa}

%\color{red} This section is completely revised. Please re-read it and approve or change \normalcolor
We first provide some explicit examples of model functions $ \psi(r) $ that satisfy the assumptions of Theorem \ref{thm: chvar-gen}, so as to illustrate what kind of perturbations of the weight $ s^{-2} $ one deals with in a super-Euclidean setting. We refer in part to the terminology of Section \ref{sect: exa-mod}. For the reader's convenience, at the end of this section we provide Table \ref{tab: table-psi-rho} which sums up, in all of our examples, the correspondence between $ \psi(r) $ and $ \rho(s) $.

\noindent $\bullet$ {\bf Quasi- and super-hyperbolic models.} Here we have in mind functions $\psi$ that correspond to models of Type I, i.e.~$\psi(r)=e^{a r^{1+\mu}}$ for large $r$, with $ a>0 $ and $ \mu > -1 $. Then it is readily seen that \eqref{eq: ass-fg} holds with
\begin{equation}\label{eq: exa-1}
h(\tau)=\kappa \, \tau^{\frac{\mu}{1+\mu}} \, , \quad \kappa := a^{\frac{1}{1+\mu}} \left( 1+\mu \right)
\end{equation}
for large $ \tau $. It is easy to see that such a function $ h $ does satisfy \eqref{eq: ass-g-1}, \eqref{eq: ass-h-scale}, \eqref{eq: ass-fg} and \eqref{eq: ass-int-h}. More generally, it is easy to check that the class of functions $ \psi \in \mathcal{A} $ for which \eqref{eq: ass-fg} holds for some $ h(\tau) $ as in \eqref{eq: exa-1} is given by all $ \psi \in \mathcal{A} $ such that
\begin{equation}\label{eq: exa-1-ratio}
\frac{\psi^\prime(r)}{\psi(r)} \sim C \, r^{\mu} \quad \textrm{as } r \to +\infty \, , \quad C:=\frac{\kappa^{1+\mu}}{(1+\mu)^\mu} \, .
\end{equation}
By integrating \eqref{eq: exa-1-ratio} we find that the latter is equivalent to $ \psi(r)=e^{f(r)} $ with
\begin{equation*}\label{eq: exa-1-ratio-2}
f(r) \sim \frac{C}{1+\mu} \, r^{1+\mu} \, , \quad f^\prime(r) \sim C \, r^{\mu} \quad \textrm{as } r \to +\infty \, .
\end{equation*}
Hence, in this case \eqref{eq: chvar-as-final} reads
\begin{equation}\label{eq: exa-1-b}
\rho(s) \sim \frac{ \left( \frac{n-2}{n-1} \right)^{\frac{2}{1+\mu}}}{\kappa^2 \, s^2 \left( \log s \right)^{\frac{2\mu}{1+\mu}} } \quad \textrm{as} \ s \to +\infty \,,
\end{equation}
which is a {\sl log} perturbation of the famous inverse square potential, $\rho(s)=c/s^2$.
As concerns the change of variables $ r=r(s) $ (and back), from \eqref{eq: chvar-as-4} we deduce that
\begin{equation*}\label{eq: exa-1-c}
r \sim c_1 \left( \log s \right)^{\frac{1}{1+\mu}} \quad \textrm{as } s \to + \infty \, , \quad c_1 := \left( \frac{1+\mu}{C} \, \frac{n-2}{n-1} \right)^{\frac{1}{1+\mu}} .
\end{equation*}
Finally let us remark that the case $ \mu = 0 $, where the logarithm in the weight $ \rho(s) $ disappears, corresponds to \emph{hyperbolic} models, including the hyperbolic space itself (where $ \psi(r)=\sinh r $), for which the connection with the weight $ s^{-2} $ has already been deeply investigated in \cite{V}.

\smallskip

\noindent $\bullet$ {\bf Extremal super-hyperbolic models: doubly exponential growth.} As the reader may note, the power of the logarithm in the denominator of \eqref{eq: exa-1-b} is strictly smaller than $2$, such a power being formally attained at $ \mu=+\infty $. Let us verify that the latter is associated with model functions $ \psi \in \mathcal{A} $ with doubly exponential growth. More precisely, let \eqref{eq: ass-fg} be satisfied with
%\begin{equation}\label{eq: exa-2}
%h(\tau)=C \, \tau  \quad
%\end{equation}
$h(\tau)=C \, \tau$ for some $C>0$ and large $ \tau $ (clearly \eqref{eq: ass-g-1}--\eqref{eq: ass-fg} and \eqref{eq: ass-int-h} are still fulfilled).

It is readily seen that the corresponding model functions are precisely those which can be written as $ \psi(r)=e^{f(r)} $ with $f(r)=e^{\ell(r)}$
%\begin{equation}\label{eq: exa-2-a}
%f(r)=e^{\ell(r)}
%\end{equation}
for large $r$, for some regular $ \ell $ such that
\begin{equation*}\label{eq: exa-2-b}
\ell(r) \sim C \, r \, , \quad \ell^\prime(r) \sim C \quad \textrm{as } r \to +\infty \, .
\end{equation*}
Hence, the weight $ \rho $ associated with such $ \psi $ through \eqref{eq: chvar-as-final} satisfies
\begin{equation*}\label{eq: exa-2-c}
\rho(s) \sim \frac{1}{C^2 \, s^2 \left( \log s \right)^2} \quad \textrm{as} \ s \to +\infty \, .
\end{equation*}
In this case the change of variables from $ r $ to $ s $ yields
\begin{equation*}\label{eq: exa-1-d}
r \sim  \frac{1}{C} \, \log \log s \quad \textrm{as } s \to +\infty \, .
\end{equation*}

\smallskip

Now we turn to quasi-Euclidean models: namely we want to consider, for large $r$, $ \psi(r) = a \, r^{q} $ for some $ q \ge 1 $ and $ a > 0 $, or small perturbations of this kind of functions (we refer to Types II--IV, according to the terminology adopted in Section \ref{sect: exa-mod}). Note that, with such a choice, condition \eqref{eq: ass-fg} is satisfied with $ h(\tau)=q \, e^{-\tau/q} $. Hence, in this case \eqref{eq: ass-g-1} fails and we can no more apply Theorem \ref{thm: chvar-gen}. Nevertheless, we shall see that the particular structure of these models still allows us to get precise information over the asymptotic behaviour of the corresponding weight $ \rho(s) $.
\smallskip

\noindent $\bullet$ {\bf Quasi-Euclidean models I.}  Let $ \psi \in \mathcal{A} $ fall within the framework of Type II models, namely let it satisfy
\begin{equation}\label{eq: t1-asymp}
\psi(r) \sim a \, r^{q} \quad \textrm{as } r \to +\infty
\end{equation}
for some $ q>1 $ and $ a > 0 $. Set
\begin{equation}\label{eq: def-p}
 p_q := \frac{2(n-1)(q-1)}{(n-1)q-1} \, ;
\end{equation}
note that $ 0<p_q<2 $. Direct computations show that in this case \eqref{eq: chvar-initial} yields
\begin{equation}\label{eq: t1-asymp-chvar}
r \sim c \, s^{\frac{2-p_q}{2}} \quad \textrm{as } s \to +\infty \, , \quad c:= a^{-\frac{n-1}{(n-1)q-1}} \left[ \frac{n-2}{(n-1)q-1} \right]^{\frac{1}{(n-1)q-1}} \, ,
\end{equation}
so that the weight $ \rho(s) $ related to $ \psi(r) $ by \eqref{rho.def} satisfies
\begin{equation}\label{eq: exa-t1-euc}
\rho(s) \sim  c_1 \, s^{-p_q} \quad \textrm{as} \ s \to +\infty \, , \quad c_1 := a^{2(n-1)} \, c^{2(n-1)q} \, .
\end{equation}
%Next, we consider the functions $\psi$ that correspond to manifolds of Type II, i.\,e., $\psi(r)= cr^{\alpha}$, $\alpha>1$, for large $r$.
%% Note that ${\textrm K}_\omega(x)\sim -\alpha(\alpha-1)\,r^{-2}$, which is the \it critical \rm decay of curvature singled out in our main results.
%Performing the change of variables is in this case quite easy and we get a weighted PME with
%$\rho(s)\sim s^{-p} $, $0<p< 2$.  We get the formulas
%$$
%\alpha=\frac{2(n-1)-p}{(2-p)(n-1)}, \quad p=\frac{2(n-1)(\alpha-1)}{(n-1)\alpha-1}
%$$
%and
%$$
%r\sim cs^{(2-p)/2}; \quad \psi(r)\sim cs^{ 1-\frac{p}{2(n-1)}}\,. \normalcolor
%$$
%We see that  $\alpha$ goes from $\alpha=1$  for $p=0$ to $\alpha=\infty$ for $p=2$.
There is an intuitive calculation behind these transformations. Indeed, going back to formula \eqref{formula-laplacian-model} (for the radial Laplacian), we see that using $\psi(r)\sim a \, r^{q}$ is more or less equivalent to working in a ``fractional'' Euclidean dimension $n_q$ such that $n_q-1=q(n-1)$ (precisely the one defined by \eqref{hp-qe1-cri-param}), so this does not really produce essentially new geometrical results, at least when dealing with radial solutions. But it must be said that it will allow us to handle the long-time behaviour of the PME on such type of (and more general) manifolds for a wider class of data, and this is very relevant for completeness of our presentation (see Section \ref{sec.proofs-euclid}).

\smallskip

\noindent $\bullet$ {\bf Quasi-Euclidean models II.} A closer approximation to Euclidean space happens in Type III models, where we suppose that
\begin{equation*}\label{eq: t2-asymp}
\psi(r) \sim c \, r \, (\log r)^q \quad \textrm{as } r \to +\infty
\end{equation*}
for some $ q,c>0 $. By using L'H\^{o}pital's rule in \eqref{eq: chvar-initial}, it is not difficult to deduce that
\begin{equation*}\label{eq: t2-asymp-chvar}
r \sim c^{-\frac{n-1}{n-2}} \, s \left( \log s \right)^{-q \, \frac{n-1}{n-2}} \quad \textrm{as } s \to +\infty \, ,
\end{equation*}
whence
\begin{equation}\label{eq: exa-t2-euc}
\rho(s)\sim c^{-\frac{2(n-1)}{n-2}} \, (\log s)^{-\frac{2q(n-1)}{n-2}} \quad \textrm{as } s \to +\infty \, .
\end{equation}
In contrast with case I, this kind of models are \emph{not} equivalent to a Euclidean manifold up to a change of spatial dimension. This is well illustrated by the logarithmic weight appearing in \eqref{eq: exa-t2-euc}, and we are not aware of any available result for the corresponding weighted Euclidean PME.

\smallskip

\noindent $\bullet$ {\bf Quasi-Euclidean models III.} Finally, let us consider functions $ \psi $ that basically correspond to manifolds of Type IV. Namely, let $ \psi \in \mathcal{A} $ satisfy
\begin{equation}\label{eq: t3-asymp}
\psi(r) \sim a r \quad \textrm{as } r \to +\infty
\end{equation}
for some $ a >0 $. It is readily seen that everything works exactly as for quasi-Euclidean models I, up to letting $ q=1 $ and so $p=0$ there. Hence, in this case we have that $ \rho(s) $ tends to a positive constant at infinity:
\begin{equation}\label{eq: t3-asymp-rho}
\lim_{s \to +\infty} \rho(s) = a^{-2 \, \frac{n-1}{n-2}} \, .
\end{equation}
This situation is then in some sense fully Euclidean (i.e.~with the same dimension $n$).
%i.\,e., $ \psi(r)=r e^{c r^{-\alpha}} $ for some $ c \neq 0 $ and $ \alpha>0 $, so that $ K(r) \approx r^{2\mu} $ as $ r \to +\infty $ with $ \mu=-(\alpha+2)/2 < -1 $, apart from the special case $\mu=-3/2$ which corresponds to $ \psi(r)=r e^{-c \, \frac{\log r}{ r}} $. It is quite elementary to show that $r\sim s$ as $r\to+\infty$ and that there exists $K>0$ such that $\rho(s)\to K$ as $s\to+\infty$.
%\smallskip
%
%Let us recapitulate the scenario in the above examples by means of a simple table.

\bigskip

\begin{table}[ht]
\renewcommand{\arraystretch}{1.7}
\setlength{\tabcolsep}{7.0pt}
\centering
%\small
\begin{tabular}[c]{| l | c | c |}
  \hline			
 {Type of model manifold}  & $ \psi(r) $ & $ \rho(s) \sim $ \\
  \hline
  {Quasi/super-hyperbolic} & $ e^{f(r)} \, , \ \ f(r) \sim r^{1+\mu} \, , \ \ \mu > -1 $  & $ s^{-2} \left( \log s \right)^{-\frac{2\mu}{1+\mu}} $ \\
  {Doubly exponential} & $ e^{e^{\ell(r)}} \, , \ \ \ell(r) \sim  r $  & $ s^{-2} \left( \log s \right)^{-2} $ \\
  {Quasi-Euclidean I} & $ \sim r^q \, , \ \ q>1$ & $  s^{-\frac{2(n-1)(q-1)}{(n-1)q-1}}  $ \\
  {Quasi-Euclidean II} & $ \sim r \left( \log r \right)^q \, , \ \ q > 0 $ & $ (\log s)^{-\frac{2q(n-1)}{n-2}} $  \\
  {Quasi-Euclidean III} & $ \sim r $ & $ \textrm{const.} $  \\
   \hline
\end{tabular}
\caption{Correspondence between the model function $ \psi(r) $ and the weight $ \rho(s) $ after the change of variables of Section \ref{sect: the-change}. The asymptotic symbol $ \sim $ is understood as $ r $ and $ s $ go to $ + \infty $, and wherever it appears it has to be meant up to a multiplicative constant. Moreover, the analogues of the asymptotic assumptions on $ f(r) $ and $ \ell(r) $ must be required on $ f^\prime(r) $ and $ \ell^\prime(r) $ as well (we dropped them here for simplicity).}\label{tab: table-psi-rho}
\end{table}

\newpage

\subsection{The 2-dimensional cases}\label{sect: 2-dim}

When the spatial dimension is $ n=2 $, clearly formula \eqref{eq: chvar-initial} cannot be used for the following reason: integration of \eqref{eq.diff} in this dimension  yields
\begin{equation}\label{eq: chvar-log}
\log s = A + \int_{1}^r \frac{1}{\psi(t)} \, \mathrm{d}t
\end{equation}
for some arbitrary real constant $ A $. However, if $ 1/\psi(r) $ is integrable at infinity, as it happens e.g.~in quasi- and super-hyperbolic models or quasi-Euclidean models I (we refer to Section \ref{sect: chvar-exa}), the variable $s$ as a function of $r$ cannot range in the whole $ (0,+\infty) $ since it is forced by \eqref{eq: chvar-log} to stop at some finite $ s_0>0 $. For the same reasons, the weight $ \rho(s) $ given by \eqref{rho.def} becomes singular at such $s_0$. This is well illustrated in \cite[Section 9]{V}, where the focus is on a function $ \psi $ of hyperbolic type, i.e.~$ \psi(r) \sim \sinh(r) $.

In order to overcome these issues, we propose to give up the request that the spatial dimension associated with $s$ and $r$ be the same: in other words, we can impose
\begin{equation*} \label{eq.diff-2dim}
\frac{\mathrm{d}s}{s^{n_1-1}}=\frac{\mathrm{d}r}{\psi(r)}
\end{equation*}
for some (possibly non-) integer $ n_1 > 2 $, so that \eqref{eq: chvar-log} becomes
\begin{equation}\label{eq: chvar-initial-2dim}
\frac1{(n_1-2) \, s^{n_1-2}} = \int_r^{+\infty}\frac1{\psi(t)}\,{\rm d}t
\end{equation}
and equation \eqref{eq.hpme} in dimension $n=2$ is transformed into \eqref{eq.wpme} in dimension $ n_1 $, with
\begin{equation}\label{rho.def-2dim}
\rho(s)=\frac{\left[\psi\!\left(r(s)\right) \right]^{2}}{s^{2(n_1-1)}} \quad \forall s \in (0,+\infty) \, .
\end{equation}
The drawback now is that \eqref{eq: ass-psi-zero-S} and \eqref{eq: expr-rho-zero-intro} do not hold anymore (whereas they do if $ n_1=n=2 $ provided we ask a bit more on $ \psi(r) $ as $ r \to 0^+ $, see below): this is natural since locally we are turning a Euclidean space of dimension $ n=2 $ into a Euclidean space of dimension $ n_1 > 2 $. In particular, one sees that $ \rho(s) $ \emph{vanishes exponentially} as $ s \to 0^+ $. Nevertheless, this is not a major problem because to our purposes the change of variables \eqref{eq: chvar-initial-2dim} is really important only in the complement of a ball (see Section \ref{sect: sub-qe-proofs}). On the other hand, in the cases where $ 1/\psi(r) $ is \emph{not} integrable at infinity (think of quasi-Euclidean models III for instance), in principle \eqref{eq: chvar-log} still yields a well-defined one-to-one change of variables between $ \mathbb{R}^+ $ and $ \mathbb{R}^+ $.

Let us then check briefly what the situation looks like for some of the explicit models of Section \ref{sect: chvar-exa}, in which we are especially interested having in mind the asymptotics of \eqref{eq.hpme}.

%\noindent $\bullet$ {\bf Quasi- and super-hyperbolic models.}
\noindent $\bullet$ {\bf Quasi-Euclidean models I.} In view of \eqref{eq: chvar-initial-2dim}, we have that the (asymptotic) relation between $r$ and $s$ is
\begin{equation}\label{eq: t1-asymp-chvar-2dim}
s \sim c \, r^{\frac{q-1}{n_1-2}} \quad \textrm{as } r \to +\infty \, , \quad c:= \left[ \frac{a(q-1)}{n_1-2} \right]^{\frac{1}{n_1-2}} \, ,
\end{equation}
while \eqref{rho.def-2dim} reads
\begin{equation}\label{rho.def-2dim-t1}
\rho(s) \sim \widetilde{c} \, s^{2 \, \frac{n_1-q-1}{q-1}} \quad \textrm{as } s \to +\infty \, , \quad \widetilde{c} := a^2 \left[ \frac{n_1-2}{a(q-1)} \right]^{\frac{2 q}{q-1}} .
\end{equation}
Note that given $q>1$ one can make the exponent $2[n_1-q-1]/(q-1)$ in the expression for $\rho$ be \it negative\rm. Besides, for all choices of $n_1>2$, such exponent is \it supercritical\rm, namely strictly larger than $-2$.

\medskip

\noindent $\bullet$ {\bf Quasi-Euclidean models III.} In this case we need not change the spatial dimensions and use \eqref{eq: chvar-log}. However, as the latter only provides us with asymptotic information over logarithms of $r$ and $s$, asking \eqref{eq: t3-asymp} is not enough in order to describe precisely the behaviour of $ \rho(s) $. Since we would like it to tend to a positive constant both as $ s \to 0^+ $ and $ s \to +\infty $, it is reasonable to assume that $ \psi \in \mathcal{A} $ and, for some $a>0$,
\begin{equation}\label{eq: t3-asymp-2dim}
\psi(r) \sim a r \quad \textrm{as } r \to +\infty  \quad ; \quad \int_{0}^{1} \left| \frac{1}{\psi(t)} - \frac{1}{t} \right| \mathrm{d}t + \int_{1}^{+\infty} \left| \frac{1}{\psi(t)} - \frac{1}{at} \right| \mathrm{d}t< +\infty \, .
\end{equation}
From \eqref{eq: chvar-log} and \eqref{eq: t3-asymp-2dim} we therefore end up with
\begin{equation*}\label{eq: t3-asymp-chvar-2dim}
s \sim c \, r \quad \textrm{as } r \to +\infty \, , \quad c := e^{A+\int_{1}^{+\infty} \left( \frac{1}{\psi(t)} - \frac{1}{at} \right) \, \mathrm{d}t} \, ,
\end{equation*}
whence $\lim_{s \to + \infty} \rho(s) = c^{-2}$.

%\begin{equation}\label{rho.def-2dim-t3}
%\lim_{s \to + \infty} \rho(s) = c^{-2} \, .
%\end{equation}
It is then straightforward to verify that one can choose $ A $ so that \eqref{eq: ass-psi-zero-S} and \eqref{eq: expr-rho-zero-intro} are fulfilled.
%(otherwise they are up to multiplicative constants).

%%%%%%%%%%%%%%%%%%%%%%%%%%%%%%%%%%%%%%%%%%%%%%%%%%%%%%%%%%%%%%%%%%%%%%%%%%%%%%
\begin{rem}\label{obs: sub-sup}\rm
Throughout this whole section we have dealt with radial \emph{solutions} to \eqref{eq.hpme} and \eqref{eq.wpme}. However, it is apparent that everything continues to work if we consider \emph{sub-} or \emph{supersolutions} instead: in other words, radial subsolutions and supersolutions to \eqref{eq.hpme} are transformed, by means of the change of variables of Section \ref{sect: the-change}, into radial subsolutions and supersolutions, respectively, to \eqref{eq.wpme}. This is important also in the light of Laplacian comparison results recalled in Section \ref{sect: geom-pre}.
\end{rem}

\section{Proof of the main results in the quasi-Euclidean range}\label{sec.proofs-euclid}

In this section we sketch the proofs of our asymptotic results when the assumptions on curvatures are of quasi-Euclidean type, namely Theorems \ref{pro: upper-qe1}--\ref{pro: lower-qe1-cri}. In such a range it is convenient to exploit the results of Section \ref{sec: chvar} and work on the Euclidean weighted PME \eqref{eq.wpme}. We shall separate the critical and subcritical cases, since some significant differences arise.

\subsection{The lower-critical case $\boldsymbol{\mu=-1}$: Theorems \ref{pro: upper-qe1-cri} and \ref{pro: lower-qe1-cri}}\label{sec.proofs-euclid.n3}
First of all, let $\psi$ be the same model function as Lemma \ref{lem: ode-comparison} or in Lemma \ref{lem: ode-comparison-2}. By definition, $ \psi \in \mathcal{A} $, $ \psi \in C^2([0,+\infty)) $ and it satisfies $ \psi^{\prime\prime}(r) = Q \, r^{-2} \, \psi(r) $ for all $ r > 2R $. In particular, since solutions to the latter equation are explicit, we have that $ \psi(r)=a_1 \, r^{q_1} + a_2 \, r^{q_2}  $ for all $ r>2R $ and some constants $a_1,a_2$, where $ q_1 := (1+\sqrt{1+4Q})/2 > 1 $ and $ q_2 := (1-\sqrt{1+4Q})/2 < 0$ are the solutions of $ q(q-1)=Q $. The positivity of $\psi,\psi^\prime,\psi^{\prime\prime}$ trivially implies $a_1>0$, so that \eqref{eq: t1-asymp} holds with $ q = q_1 $ and $ a=a_1 $. Hence, in the case $n\ge3$, to such a $ \psi(r) $ we can apply the results of Section \ref{sect: chvar-exa} (quasi-Euclidean models I), which associate it with a weight $ \rho(s) $ satisfying \eqref{eq: exa-t1-euc}, the latter falling within the framework of \cite{RV}. In particular, the Barenblatt-type function
\begin{equation}\label{eq: barenblatt-1}
\mathfrak{B}(s,t) := \frac{C_0}{(t+t_0)^{\frac{n_q}{2+n_q(m-1)}}} \left(\gamma_0 - \frac{s^{2-p_q}}{(t+t_0)^{\frac{2}{2+n_q(m-1)}}} \right)^{\frac1{m-1}}_+ \quad \forall (s,t) \in (R_0,+\infty) \times \mathbb{R}^+  \, ,
\end{equation}
where $ n_q $ is defined by \eqref{hp-qe1-cri-param} and $ p_q $ is defined by \eqref{eq: def-p}, is either a subsolution or a supersolution to \eqref{eq.wpme} provided $ t_0 , R_0 $ are large enough and $ C_0,\gamma_0 $ are small or large enough, respectively. Hence, the change of variables \eqref{eq: t1-asymp-chvar}, Remark \ref{obs: sub-sup} and Laplacian-comparison results recalled in Section \ref{sect: geom-pre} easily imply that the barrier functions appearing in \eqref{bound-qe1-cri} and \eqref{bound-qe1-2-low-cri} are indeed supersolutions and subsolutions to the PME in \eqref{pme}, respectively, at least in the complement of a ball. In order to show that they are also inside a ball, one can make direct computations as in Lemmas \ref{pro: upper-barrier} and \ref{pro: lower-barrier} (step 2 of the corresponding proofs).

Finally, in order to manage the case $n=2$, we can resort to the discussion of Section \ref{sect: 2-dim}, in particular to \eqref{eq: t1-asymp-chvar-2dim} and \eqref{rho.def-2dim-t1}, and argue exactly as above: the expression of $\mathfrak{B}$ is the same as \eqref{eq: barenblatt-1} up to choosing $p_q := 2(q+1-n_1)/(q-1) $ instead.

\subsection{The subcritical range $\boldsymbol{\mu<-1}$: Theorems \ref{pro: upper-qe1} and \ref{pro: lower-qe1}} \label{sect: sub-qe-proofs}
If $ \mu<-1 $ the situation is more delicate since the equation for $ \psi $ reads $ \psi^{\prime\prime}=Q \, r^{2\mu} \, \psi $ for large $ r $, and no explicit nontrivial solution is available for it. In order to understand the asymptotics of $ \psi(r) $ as $ r \to +\infty $ in such cases, we therefore need a preliminary technical lemma.
\begin{lem}\label{lem: ode-comparison-qe}
Let $R,Q>0$ and $ \mu < -1 $ be fixed parameters. Let $ \psi $ satisfy
\begin{equation}\label{ode-psi-upper-lemma-sub}
\psi''(r) = w(r) \, \psi(r) \quad \forall r>2R \, , \quad \psi'(2R)>0 \, , \quad \psi(2R)>0  \, ,
\end{equation}
where $ w(r):=Q \, r^{2\mu} $. Then $ \psi \sim c r $ as $ r \to +\infty $, for some $c>0$.
\end{lem}
\begin{proof}
We proceed by means of a fixed point argument. To this end, let $R_0>2R$ to be chosen later. By integrating twice \eqref{ode-psi-upper-lemma-sub} between $R_0$ and $r>R_0$, we get:
\begin{equation}\label{eq: fixed-pt-1}
%\begin{aligned}
\psi(r)=\psi(R_0)+\psi^\prime(R_0)(r-R_0)+\int_{R_0}^r (r-\tau)w(\tau)\psi(\tau) \, \rd \tau \, .
%\psi(r)&=\psi(R_0)+\psi^\prime(R_0)(r-R_0)+\int_{R_0}^r \rd t \int_{R_0}^t w(\tau)\psi(\tau) \, \rd \tau \\
% &=u(R_0)+u^\prime(R_0)(r-R_0)+\int_{R_0}^r\rd \tau\,w(\tau)\psi(\tau)\int_\tau^r\rd t\\
%\end{aligned}
\end{equation}
Let $v(r):=\psi(r)/r$, so that $\psi(R_0)=R_0\,v(R_0)$, $\psi^\prime(R_0)=R_0\,v^\prime(R_0)+v(R_0)$ and \eqref{eq: fixed-pt-1} reads
%\[
%r\,v(r)=R_0\,v(R_0)+(r-R_0)\left[R_0\,v^\prime(R_0)+v(R_0)\right]+\int_{R_0}^r \tau (r-\tau) w(\tau) v(\tau) \, \rd \tau \, ,
%\]
%or equivalently
\[
v(r)=a+b\,(r-R_0)\,\frac{R_0}r+ \frac1r \int_{R_0}^r \tau (r-\tau) w(\tau) v(\tau) \, \rd \tau =:(T_{a,b}\,v)(r) \, ,
\]
where we have set $v(R_0)=:a$ and $v^\prime(R_0)=:b$. Consider now the Banach space $X:=C([R_0,+\infty))\cap \LL^\infty((R_0,+\infty)) $ endowed with the $ \LL^\infty $ norm. Clearly $T_{a,b}\,v \in X$ for all $v\in X$ since $g(r):=a+b(r-R_0){R_0}/r$ is bounded and the structure of $w$ implies that, for all $ r > R_0 $,
\begin{equation}\label{eq: fix-pt-est}
\begin{aligned}
& \frac1r \left| \int_{R_0}^r \tau (r-\tau) w(\tau) v(\tau) \, \rd \tau \right| \le \|v\|_{\LL^\infty((R_0,+\infty))} \, \frac{Q}{r} \,  \int_{R_0}^r\tau^{1+2\mu}\,(r-\tau) \, \rd \tau \\
\le & \|v\|_{\LL^\infty((R_0,+\infty))} \, Q \, \int_{R_0}^r \tau^{1+2\mu} \, \rd \tau \le - Q \, \frac{R_0^{2+2\mu}}{2+2\mu} \, \|v\|_{\LL^\infty((R_0,+\infty))} \, .
\end{aligned}
\end{equation}
We claim that $T_{a,b}$ is a contraction on $X$ provided one picks $ R_0 $ large enough. In fact
$$ T_{a,b}\,v_1-T_{a,b}\,v_2 = \frac1r \int_{R_0}^r \tau (r-\tau) w(\tau) \left[ v_1(\tau)-v_2(\tau) \right] \rd \tau =:S\,(v_1-v_2) \quad \forall v_1 , v_2 \in X \, , $$
and estimate \eqref{eq: fix-pt-est} implies that $S$ is a bounded linear operator on $X$ with norm at most $ -Q \, {R_0^{2+2\mu}}(2+2\mu) $.
%\[
%\begin{aligned}
%\left|(S_{a,b} \, v)(r)\right| & \le \|v\|_{\LL^\infty((R_0,+\infty))} \, \frac{Q_1}{r} \, \int_{R_0}^r\tau^{1+2\mu}\,(r-\tau) \, \rd \tau \\
%& \le \|v\|_{\LL^\infty((R_0,+\infty))} \, Q_1 \, \int_{R_0}^r \tau^{1+2\mu} \, \rd \tau \le - Q_1 \, \frac{R_0^{2+2\mu}}{2+\mu} \, \|v\|_{\LL^\infty((R_0,+\infty))} \ \ \forall r > R_0 \, . \\
%\end{aligned}
%\]
Since $ \mu < -1 $, by choosing $R_0$ large enough we can make such norm as small as needed (independently of $a,b$), so that $T_{a,b}$ is indeed a contraction on $X$ with a fixed point $ \overline{v} $. By construction, $ \phi(r):=r\,\overline{v}(r) $ satisfies $ \phi^{\prime\prime}=w \, \phi $ in $ (R_0,+\infty) $ and we can choose $a,b$ so that $ \phi(R_0)=\psi(R_0) $ and $ \phi^\prime(R_0)=\psi^\prime(R_0) $. By uniqueness, it follows that $ \phi=\psi $. Convexity of $\psi$ implies that $ \psi^\prime $ is increasing, so it has a limit as $r\to+\infty$, and such limit has to be finite (and nonzero) by the above fixed point argument. Hence $\psi(r)\to+\infty$ as $r\to+\infty$. L'H\^opital's rule yields the claim.
\end{proof}

In the case $ n \ge 3 $ we can now reason exactly as in Section \ref{sec.proofs-euclid.n3}: the model functions $ \psi $ of Lemmas \ref{lem: ode-comparison}--\ref{lem: ode-comparison-2} clearly fulfil the assumptions of Lemma \ref{lem: ode-comparison-qe}, thus falling within the class labelled as quasi-Euclidean models III in Section \ref{sect: chvar-exa}, and therefore being associated with a density $ \rho(s) $ that satisfies \eqref{eq: t3-asymp-rho}. Hence, equation \eqref{eq.wpme} is purely Euclidean, so that the barrier functions appearing in \eqref{bound-qe1-2} and \eqref{bound-qe1-2-low} are classical Barenblatt profiles.

In the case $ n = 2 $ we can argue similarly: in order to apply the results of Section \ref{sect: 2-dim} we only have to prove, in addition, that the function $ \psi $ of Lemma \ref{lem: ode-comparison-qe} satisfies the integral estimates in \eqref{eq: t3-asymp-2dim} with $ a=c $. To this end, first of all note that $ \int_{0}^{1} \left| {1}/{\psi(t)} - {1}/{t} \right| \mathrm{d}t $ is finite as a trivial consequence of the fact that $ \psi \in \mathcal{A} \cap C^2([0,1)) $. As for finiteness of $ \int_{1}^{+\infty} \left| {1}/{\psi(t)} - {1}/{c t} \right| \mathrm{d}t $, it is enough to show that $ \psi(r) = cr + A(r) $ with $A(r)=O(r^{1-\epsilon}) $ as $ r \to +\infty $, for some $ \epsilon \in (0,1) $. Lemma \ref{lem: ode-comparison-qe} ensures that $ A(r)=o(r) $; moreover, it satisfies (for large $r$) the equation $A''=w\,(cr+A)$, so that $A$ is eventually convex and therefore $A'$ has a limit as $r\to+\infty$. The fact that $A(r)=o(r)$ forces such a limit to be zero: in particular, we can write $A'(r)=-\int_{r}^{+\infty} A''(\tau)\,{\rm d}\tau$. Since $A''(\tau) \sim Q \, c \,\tau^{2\mu+1}$ as $\tau\to+\infty$ and $\mu<-1$, it is clear that there exists $\epsilon \in (0,1)$ such that $A(r)=O(r^{1-\epsilon})$, and this concludes the proof.
\section{Asymptotic behaviour of a weighted Euclidean porous medium equation} \label{sect: chvar-asymp}

By combining the results of Sections \ref{sec.smr} and \ref{sec: chvar} (we refer in particular to Sections \ref{sect: dir-pb}--\ref{sect: chvar-exa}), we can provide asymptotic information over the behaviour of solutions, corresponding to compactly supported data, to the following \emph{weighted porous medium equation}:
\begin{equation}\label{eq.wpme-2}
\begin{cases}
\rho(x) u_t= \Delta \! \left( {u}^m \right) & \textrm{in } \mathbb{R}^n \times \mathbb{R}^+ \, , \\
u(\cdot,0)=u_0  & \textrm{in } \mathbb{R}^n  \, ,
\end{cases}
\end{equation}
provided $ \rho $ is a positive weight having suitable decay properties at infinity.

\begin{thm}\label{thm: peso}
Let $ n \ge 2 $ and $ u $ be the solution of the weighted porous medium equation \eqref{eq.wpme-2} corresponding to a nonnegative, bounded and compactly supported initial datum $u_0 \not\equiv 0 $. Suppose that the positive weight $ \rho \in \LL^\infty_{\rm loc}(\mathbb{R}^n) $, with $ \rho^{-1} \in \LL^\infty_{\rm loc}(\mathbb{R}^n) $, satisfies
\begin{equation}\label{eq: cond-weight-infty}
\frac{1}{C \, |x|^2 \left( \log |x| \right)^{\nu} } \, \le \rho(x) \le \frac{C}{|x|^2 \left( \log |x| \right)^{\nu}} \,  \quad \forall x \in B_2^c
\end{equation}
for some $ C>0 $ and $ \nu \in (-\infty,1) $. Then the following bound holds:
\begin{equation}\label{bound-peso-1}
\begin{aligned}
&\frac{C_0\,\left[\log(t+t_0)\right]^{\frac{1-\nu}{m-1}}}{(t+t_0)^{\frac1{m-1}}} \left[\gamma_0 - \left(\frac{\log|x|}{\log(t+t_0)}\right)^{\! 1-\nu}\right]_+^{\frac1{m-1}} \le\,
u(x,t) \le \\
&\frac{C_1\,\left[\log(t+t_0)\right]^{\frac{1-\nu}{m-1}}}{(t+t_0)^{\frac1{m-1}}} \left[\gamma_1 - \left(\frac{\log|x|}{\log(t+t_0)}\right)^{\! 1-\nu}\right]_+^{\frac1{m-1}} \end{aligned}
\end{equation}
$$ \forall x \in B_2^c \, , \quad \forall t \ge 0 \, , $$
where $ C_0 $, $ C_1 $, $ \gamma_0 $, $ \gamma_1 $, $ t_0 $ are positive constants depending on $ n,m,\nu,C,u_0 $.
\end{thm}

\begin{thm}\label{thm: peso-critico}
Let the same assumptions as in Theorem \ref{thm: peso} be satisfied, with \eqref{eq: cond-weight-infty} replaced by
\begin{equation}\label{eq: cond-2-infty-critical}
\frac{1}{C \, |x|^2 \, \log |x| } \, \le \rho(x) \le \frac{C}{|x|^2 \, \log |x|} \,  \quad \forall x \in B_2^c
\end{equation}
for some $ C>0 $. Then the following bound holds:
\begin{equation}\label{bound-peso-1-critical}
\begin{aligned}
&\frac{C_0}{(t+t_0)^{\frac1{m-1}}} \left[ -\eta_0  - \log \left( \frac{\log |x|}{\log(t+t_0)} \right) \right]_+^{\frac1{m-1}}
\le  u(x,t) \le
\\ &\frac{C_1}{(t+t_0)^{\frac1{m-1}}} \left[ \eta_1  - \log \left( \frac{\log |x|}{\log(t+t_0)} \right)  \right]_+^{\frac1{m-1}}
\end{aligned}
\end{equation}
$$ \forall x \in B_2^c \, , \quad \forall t \ge 0 \, , $$
where $ C_0 $, $ C_1 $, $ \eta_0 $, $ \eta_1 $, $ t_0 $ are positive constants depending on $ n,m,C,u_0 $.
\end{thm}
The rigorous application of Theorem \ref{thm: chvar-gen}, up to modifications to be taken into account to deal with the 2-dimensional case (Section \ref{sect: 2-dim}) and Remark \ref{obs: sub-sup}, would allow us to state the above results only for positive \emph{radial} weights $ \rho \in C^\infty(\mathbb{R}^n \setminus \{0\} ) \cap C(\mathbb{R}^n) $ satisfying
\begin{equation}\label{eq: cond-1-local}
\rho(0)=1 \, , \quad \lim_{x \to 0} x \cdot \nabla{\rho}(x) = 0
\end{equation}
and
\begin{equation}\label{eq: cond-2-infty}
\lim_{|x| \to +\infty} \frac{x \cdot \nabla{\rho}(x)}{\rho(x)} = - 2 \, , \quad \rho(x) \sim \frac{c}{|x|^2 \left( \log |x| \right)^{\! \nu} } \quad \textrm{as } |x| \to +\infty
\end{equation}
for some $ \nu \in (-\infty,1] $ and $c>0$. However, direct computations (which we omit) show that the left-hand and right-hand sides of \eqref{bound-peso-1} and \eqref{bound-peso-1-critical} yield good barriers (outside a ball) for \eqref{eq.wpme-2} under weaker assumptions on $ \rho $ as in the statement of Theorem \ref{thm: peso} or \ref{thm: peso-critico}.

\section{Comments and open problems}\label{sect: comm-opp}

 \noindent$\bullet$ {\bf Extensions.} \noindent One may wonder what happens outside the range of $\mu$ we have treated. The supercritical case $ \mu>1 $ contains manifolds with very negative curvature at infinity. It is then known that {\sl stochastic completeness} is lost or, equivalently, that mass conservation for integrable solutions of the corresponding heat equation does not hold anymore (see for instance \cite{Grig}). Having in mind similar arguments as in \cite{BK} and \cite{KRV},  we expect that separable solutions come up, which should play the main role for asymptotics at least in inner regions.  We plan to treat that issue elsewhere  in the future.

%
%We expect a complex behaviour where the solutions will take the form of separate variables on compact sets, while the free boundary corresponds to the formation of a moving tail (one has to exploit matching asymptotics as in \cite{BK}, but barriers aren't trivial). \color{red} cancellare ? It is probably the most interesting development of this paper we should carry out.\normalcolor

\medskip

 \noindent$\bullet$ {\bf Comment on the linear case.}
We do not know of explicit, detailed results for the linear equation $u_t=\Delta_g(u)$ under general negative-curvature conditions like the ones used in this paper, except for the  well-studied case of the hyperbolic space. Such lack of information was  surprising for us. Some partial results are however available, and most likely other ones can be derived using the general theory developed e.g. in \cite{Grig2, Grig3, SC, CG}. We confine ourselves here to making some relevant comments. First, Grigor'yan condiders the volume $V(r)$ of Riemannian balls of a model, centered at a pole $o$, and states in \cite[Example 5.36]{Grig3} that  when $V(r)$ behaves like $e^{r^\alpha}$ for large $r$ with $\alpha\in(0,1)$ (i.e., with our notation, $\psi(r)\sim r^{(\alpha-1)/(n-1)}e^{r^{\alpha}/(n-1)}$), the heat kernel satisfies
\[
    p_t(o,o)\approx \,e^{-c_1t^{\frac{\alpha}{2-\alpha}}}
\]
for a suitable $c_1>0$ and all $t>1$. A similar \it lower \rm bound holds also when $r^2/\log(V(r))$ is increasing for large $r$,
%$V(r)$ being the volume of geodesic balls centered at a given, fixed, $x\in M$,
a condition which holds when $V(r)\sim e^{cr^\alpha}$ in the wider range $\alpha\in(0,2)$, on \it general \rm manifolds, which can be seen e.g. applying Theorem 6.1 of \cite{CG}. In the case $\alpha\in(1,2)$ we could not find a corresponding upper bound in the literature, although it might be possible to derive it at least on models by Corollary 5.34 in \cite{Grig3}.

\medskip

 \noindent$\bullet$ {\bf Volume estimates.} Let us make some volume-propagation considerations, that is, let us compute the volume occupied by a compactly supported solution at large times (which we label $ \mathsf{V}(t) $), relying on our estimates on $ \mathsf{R}(t) $. In Euclidean context, in view of \eqref{eq: euclid-intro} we deduce that
\begin{equation}\label{eq: euclid-intro-volume}
\mathsf{V}(t) \approx t^{\frac{n}{n(m-1)+2}} \, .
\end{equation}
In hyperbolic space, since $ \mathsf{V}(t) $ behaves like $ e^{(n-1) \, \mathsf{R}(t)} $, estimate \eqref{eq: hyperb-intro} yields
\begin{equation}\label{eq: hyperb-intro-volume}
\mathsf{V}(t) \approx t^{\frac{1}{m-1}} \, ,
\end{equation}
which reads \eqref{eq: euclid-intro-volume} as $ n \to +\infty $. In particular, we remark that the speed of propagation is \emph{slower} on $ \mathbb{H}^n $ with respect to $ \mathbb{R}^N $ if one looks at \emph{distances} covered, while it is slightly \emph{faster} if one looks at \emph{volumes} covered.

\medskip

\noindent $\bullet$ The restriction of compact support on the data is important for the upper barriers we construct. On the other hand, the assumption of bounded initial data is only made for convenience, since it is not difficult to prove that, like in the Euclidean case, integrable and compactly supported initial data implies bounded and compactly supported solutions for small times.

\medskip

\medskip

\noindent $\bullet$  Smoothing effects, meaning that $u_0$ is assumed to be only in $L^1$, are not discussed here. The problem of a general smoothing effect in our context (we only have barriers for bounded and compactly supported initial data) is open and appears to be non trivial.

\medskip

%\noindent \color{red} possiamo eloiminare questo paragrafo ? \color{blue} $\bullet$  Possible generalisation to bounds over curvatures that are not comparable to power functions, namely something like $ \mathrm{K}_\omega(x) \le -f(r) $.\normalcolor
%
%\medskip

\noindent $\bullet$
Finally, some of our results still hold if we allow the manifold to have positive curvature in a compact set. In fact, for the lower bounds it is actually enough to choose $R$ so large that $ \inf_{x \in B_{R}(o) } \mathrm{Ric}_o(x) = 0 $ in \eqref{eq: choice-B-curv}. However, in the case of upper bounds apparently there are problems in assuming this relaxed sign condition, in fact one needs some ``smallness'' condition either over curvature or diameter of the ball where curvature is positive so that the function $ \psi $ defining the model manifold used for comparison remains positive.
\vskip 1cm

%%%%%%%%%%%%%%%%%%%%%%%%%%%%%%%%%%%%%%%%%%%%%%%%%%%%%%%%%%%%%%%%%%%%%%%%%%%
\noindent {\textsc{Acknowledgements}. G.G.~has been partially supported by the PRIN project {\em Equazioni alle derivate parziali di tipo ellittico e parabolico: aspetti geometrici, disuguaglianze collegate, e applicazioni} (Italy). M.M.~has been partially supported by the PRIN project \emph{Calculus of Variations} (Italy). Both G.G.~and M.M.~have also been supported by the Gruppo Nazionale per l'Analisi Matematica, la Probabilit\`a e le loro Applicazioni (GNAMPA) of the Istituto Nazionale di Alta Matematica (INdAM, Italy). J.L.V.~has been partially funded by Projects MTM2011-24696 and MTM2014-52240-P (Spain). M.M.~thanks Universidad Aut\'onoma de Madrid and J.L.V.~thanks Politecnico di Milano for hospitality during visits to work on this project.

\vskip 1cm

%%%%%%%%%%%%%%%%%%%%%%%%%%%%%%%%%%%%%%%%%%%%%%%%%%%%%%%%

\medskip\noindent \it 2010 Mathematics Subject Classification\rm: 35B40, 35K55, 35K65, 58J13.

\noindent\it Keywords and phrases\rm: Porous medium equation, flows on manifolds, asymptotic  behaviour, finite propagation.

%\newpage
%
%{\bf Some Notes for personal use }
%
%\medskip
%
%1) sui modelli, la curvatura sezionale rispetto ai piani che contengono la direzione radiale \`e $-\psi''/\psi$, mentre quella relativa ai piani ortogonali
%a tale direzione \`e $(1-\psi'^2)/\psi^2.$\color{red} I have put this in the section of geometric preliminaries\normalcolor
%
%2) Rimane anche da decidere se ci si vuole occupare anche del caso con curvatura ancora piu' negativa (tipo $-r^{2+a} $ con $a $ positivo), dove c'e' la separazione di variabili, li' c'e' del lavoro da fare ancora e francamente non so se valga la pena appesantire ancora, ad ogni modo le domande che abbiamo le ho segnate nel file.

\end{document}